\documentclass[reqno,10pt]{amsart}
\usepackage[a4paper]{geometry}
\usepackage{amsfonts, amsmath, amssymb, amsthm, color}

\usepackage{fullpage}

\allowdisplaybreaks[1]

\usepackage{graphicx,xcolor}

\usepackage{verbatim} 

\usepackage{tikzsymbols}  %% smiles

\usepackage[colorlinks]{hyperref}
\hypersetup{linkcolor=black,citecolor=black,filecolor=black,urlcolor=blue}

\newtheorem{thm}{Theorem}[section]
\newtheorem{lemma}[thm]{Lemma}
\newtheorem{prop}[thm]{Proposition}
\newtheorem{cor}[thm]{Corollary}

\newtheorem*{theorem_a}{Theorem A}
\newtheorem*{theorem_b}{Theorem B}
\newtheorem*{theorem_c}{Theorem C}

\theoremstyle{definition}
\newtheorem{remark}[thm]{Remark}
\newtheorem{definition}[thm]{Definition}
\usepackage{soul}

\theoremstyle{definition}
 \numberwithin{equation}{section}

\newcommand{\R}{{\mathbb R}}

\DeclareMathOperator{\dist}{dist}

\let\div\undefined
\DeclareMathOperator{\div}{div}

\def\sideremark#1{\ifvmode\leavevmode\fi\vadjust{\vbox to0pt{\vss% the remark
 \hbox to 0pt{\hskip\hsize\hskip1em%                          will appear only
 \vbox{\hsize2.1cm\tiny\raggedright\pretolerance10000%          on the side
  \noindent #1\hfill}\hss}\vbox to15pt{\vfil}\vss}}}%

\allowdisplaybreaks

\title[Optimal uniform bounds]{Optimal uniform bounds for competing variational elliptic systems with variable coefficients}
\author{Manuel Dias and Hugo Tavares}
\date{\today}

\begin{document}
\begin{abstract}

Let $\Omega \subset \mathbb{R}^N$ be an open set. In this work we consider solutions of the following gradient elliptic system
\[
-\div(A(x)\nabla u_{i,\beta}) = f_i(x,u_{i,\beta}) + a(x)\beta |u_{i, \beta}|^{\gamma -1}u_{i, \beta}
	\mathop{\sum_{j=1}^l}_{j\neq i} |u_{j, \beta}|^{\gamma + 1},
\]
for $i=1,\ldots, l$. We work in the competitive case, namely $\beta<0$. Under suitable assumptions on $A$, $a$, $f_i$ and on the exponent $\gamma$, we prove that uniform $L^\infty$--bounds on families of positive solutions $\{u_\beta\}_{\beta<0}=\{(u_{1,\beta},\ldots, u_{l,\beta})\}_{\beta<0}$ imply uniform Lipschitz bounds (which are optimal).

One of the main points in the proof are suitable generalizations of Almgren's and Alt-Caffarelli-Friedman's monotonicity formulas for solutions of such systems. Our work generalizes previous results, where the case $A(x)=Id$  (i.e. the operator is the Laplacian)  was treated. 
\end{abstract}
\maketitle

\tableofcontents

\section{Introduction}
\label{chapter:introduction}

\subsection{Statement of the main result}

Let $\Omega \subset \mathbb{R}^N$ be an open set, $N\geq 1$. Consider $u_\beta = (u_{1,\beta},...,u_{l,\beta})$, a solution of the variational system of equations given by:
\begin{align}
	\label{equation}
	&-\div(A(x)\nabla u_{i,\beta}) = f_i(x,u_{i,\beta}) + a(x)\beta |u_{i, \beta}|^{\gamma -1}u_{i, \beta}
	\mathop{\sum_{j=1}^l}_{j\neq i} |u_{j, \beta}|^{\gamma + 1}
\end{align}
for all $i=1,...,l$, where $\beta < 0$, $\gamma \geq 1$ and $x \in \Omega$. Under natural assumptions on $A$, $a$ and $f_i$, in this paper we  obtain uniform optimal bounds in $\beta$ for classes of solutions $\{u_\beta\}_{\beta<0}$. 

\smallbreak

More precisely, we make the following assumptions. For the matrix $A(x)$:

\begin{itemize}
\item[\textbf{(A1)}] There exists $\theta >0$ such that:
\begin{equation*}
	\langle A(x) \xi , \xi\rangle> \theta |\xi|^2 \quad \quad \forall x \in \Omega, \xi \in \mathbb{R}^N.
\end{equation*}
\item[\textbf{(A2)}] $A(\cdot) \in C^{0,1}(\Omega, \text{Sym}^{N\times N})$, and
\begin{align*}
	\sup_{x \in \Omega}\|A(x)\| \leq M, \quad\quad \sup_{x \in \Omega}\|DA(x)\| \leq M.
\end{align*}
\end{itemize}

For the functions $f_i$, we assume that:
\begin{itemize}
\item[\textbf{(F)}] $f_i(\cdot,\cdot) \in C(\Omega\times \mathbb{R})$, and
\end{itemize}
\begin{equation*}
	\sup_{x \in \Omega}|f_i(x,s)| = O(s),\quad \text{ as $s \rightarrow 0$ for all $i=1,...,l$. }
\end{equation*}

Finally, we make the following assumption on the function $a(x)$:
\begin{itemize}
\item[\textbf{(a)}] $a(\cdot) \in C^1(\Omega)$ and there exists $\delta>0$ such that:
\begin{equation*}
	a(x)>\delta>0 \quad \quad \forall x \in \Omega.
\end{equation*}
\end{itemize}
Our main result reads as follows.

\begin{thm}
	\label{DesiredTheorem}
	Under the previous assumptions on $A$, $a$, $f_i$, assume moreover that $\frac{\gamma N}{\gamma + 1} < 2$.	Let $\{u_\beta\}_{\beta < 0}$ be a family of positive solutions to the system  \eqref{equation} such that 
	\begin{equation}
    \label{boundedness}
 \text{there exists $m>0$ such that }     \sup_{\beta <0} \|u_\beta\|_{L^\infty(\Omega)}
    \leq
    m.
\end{equation}	
Then, given $K \Subset \Omega$, there exists a constant $C>0$ such that 
	\begin{equation}\label{eq:Lipbounds}
	\sup_{\beta <0}
	\|\nabla u_{i,\beta}\|_{L^\infty(K)}\leq
	C. \qquad \text{ for all $i \in \{1,...,l\}$.}
	\end{equation}
\end{thm}
\noindent To see the dependencies of the constant $C$ appearing in \eqref{eq:Lipbounds}, see Remark \ref{rem:dependence_of_constant} below. For a direct consequence in the framework of elliptic systems in Riemannian manifolds, see Corollary \ref{cor_Riem}.

\smallbreak

In the remainder of this introduction, we give background for this result, explain its proof and provide the structure of the paper.

\subsection{Background}

Systems of type \eqref{equation} have been widely considered in the literature in the case $A(x)=Id$ and $a(x)=1$, when the system reads
\begin{equation}\label{equation:Laplacian}
-\Delta u_{i,\beta} = f_i(x,u_{i,\beta}) + 
	\beta |u_{i, \beta}|^{\gamma -1}u_{i, \beta}\mathop{\sum_{j=1}^l}_{j\neq i} |u_{j, \beta}|^{\gamma + 1}.
\end{equation}
From a physical point of view, these systems arise naturally when looking for standing wave solutions of associated systems of Gross-Pitaevskii/nonlinear Schr\"odinger equations. The later model important phenomena in Nonlinear Optics \cite{AkAn} and Bose-Einstein condensation \cite{Rogel_Salazar_2013, Timmermans}. In the models, the solutions are the corresponding condensate amplitudes, the term $f_i(x,u_{i,\beta})$ regulates self-interactions within the same component, while $\beta$ expresses the strength and the type of interaction between different components $i$ and $j$. When $\beta>0$ this represents cooperation, while $\beta<0$ represents  competition. In the important case $f_i(x,u_{i,\beta})=|u_{i,\beta}|^{2\gamma}u_{i,\beta}$, starting from \cite{LinWei}, there is a vast literature regarding existence, multiplicity and classification of solutions to \eqref{equation:Laplacian}; we simply refer to the papers \cite{BartschDancerWang, ClappPistoia2022, ClappSzulkin2019, CorreiaJDE2016,CorreiaNA2016,DancerWeiWeth, Mandel,OliveiraTavares,PengWangWang2019,Soave,SoaveTavares,TianWang, WeiWu} (in the subcritical case) and \cite{ChenLinZou,ChenZouARMA2012, ChenZou2, ClappPistoia2018,PengPengWang2016,TavaresYou,TavaresYouZou, YinZou} (in the critical case) for more details and to check other references.

\smallbreak

In many situations (see for instance \cite{BartschDancerWang,ChenZouARMA2012,ChenZou2,ClappPistoia2018,DancerWeiWeth, TavaresYou,TavaresYouZou, TianWang}), one can build, using variational methods,  families of solutions $\{u_\beta\}_{\beta<0}$ which have uniform bounds in $L^\infty(\Omega)$, namely that satisfy \eqref{boundedness}. It is therefore a natural question to understand what is the asymptotic behaviour of such solutions as $\beta\to -\infty$, what are the optimal bounds, and how to characterize the limiting profiles. This was done in \cite{uniformHolderBoundsHugoTerraciniNoris} for $\gamma=1$ (see also \cite{CaffarelliLin,CLLL}).  Using the same strategy, the general case of system \eqref{equation:Laplacian} was done in the survey paper \cite{SOAVE2016388}.

\begin{theorem_a}[{{\cite[Theorems 1.2 \& 1.5]{SOAVE2016388}}}]\label{theoremA}
Take $\gamma>0$ and $f_i$ satisfying \textbf{(F)}. Let $\{u_\beta\}_{\beta<0}$ be a family of solutions of \eqref{equation:Laplacian} satisfying the uniform $L^\infty$--bound \eqref{boundedness}. Then, for every $K \Subset \Omega$ and $\alpha\in (0,1)$, there exists $C>0$ such that
\begin{equation}\label{eq:Holder1}
\sup_{\beta<0}\| u_\beta\|_{C^{0,\alpha}(K)}\leq C.
\end{equation}
In particular, there exists a limiting function $u=(u_1,\ldots, u_l)$, where each $u_i$ is Lipschitz continuous in $\Omega$, such that, up to a subsequence,
\begin{enumerate}
\item $u_\beta \to u$ strongly in $H^1_{loc}(\Omega)\cap C_{loc}^{0,\alpha}(\Omega)$ for every $0<\alpha<1$ and, for every compact $K \Subset \Omega$, we have
\[
\beta \int_{K} |u_{i,\beta}|^{p+1} |u_{j,\beta}|^{p+1} \to 0\qquad \text{ as $\beta \to \infty$, whenever $i\neq j$;}
\]
\item $u_i u_j\equiv 0$ whenever $i\neq j$, and  $-\Delta u_i=f_{i}(x,u)$ in the open set $\{|u_i|>0\}$.
\end{enumerate}
\end{theorem_a}

Moreover, for any limiting profile $u=(u_1,\ldots,u_l)$ as in the previous theorem, by \cite[Theorems 1.1 \& 8.1]{HugoTerraciniWeakReflectionLaw}, one deduces the structure of the free boundary $\{u=0\}$: it is, up to a set of Hausdorff dimension at most $N-2$, a regular hypersurface. Theorefore, by this regularity result, Theorem A-(2) and Hopf's lemma, one concludes that Lipschitz regularity is optimal for the limiting profiles $u$ (the gradient has a jump on the regular part of the free boundary $\{u=0\}$). The next natural question is whether one can obtain uniform Lipschitz bounds for $L^\infty$-bounded sequences of solutions $\{u_\beta\}_{\beta<0}$. This was positively answered  by Soave and Zilio in \cite{SoaveZilio}.

\begin{theorem_b}[{{\cite[Theorem 1.3]{SoaveZilio}}}]\label{theoremB}
Let $\gamma,N\geq 1$ be such that $\frac{\gamma N}{\gamma + 1} \leq 2$, and take $f_i$ satisfying \textbf{(F)}. Let $\{u_\beta\}_\beta$ be a family of solutions of \eqref{equation:Laplacian} satisfying the uniform $L^\infty$--bound \eqref{boundedness}. Then, for every $K \Subset \Omega$, there exists $C>0$ such that
\begin{equation}\label{eq:Holder2}
\sup_{\beta<0} \| u_\beta\|_{C^{0,1}(K)}\leq C.
\end{equation}
\end{theorem_b}

Therefore, our Theorem \ref{DesiredTheorem} is an extension of this result to the framework of systems of type \eqref{equation}. We explain in Subsection \ref{subsec:structure} which are the main difficulties one faces when passing from the case of the Laplacian operator to a divergence operator with variable coefficients.

\smallbreak

Observe that passing from H\"older to Lipschitz bounds is a nontrivial task. The proofs of H\"older bounds are based, among other things, on the fact that there exist no harmonic functions, apart from the constants, which have bounded $\alpha$-H\"older seminorm in $\mathbb{R}^N$ for some $\alpha \in ]0,1[$. The proof of \eqref{eq:Holder1} proceeds by contradiction and by performing a blowup argument close to the region where one does not have a bound. One then reaches a contradiction, in the end, by studying all possible cases for the \emph{limiting profiles}, excluding them using Liouville type results like the one just stated (within this process, an Almgren's monotonicity formula is proved for the limits). However, this type of proof does not translate to the Lipschitz setting. In \cite{SoaveZilio}, in order to prove \eqref{eq:Holder2}, a contradiction is not obtained at the limit of the blowups, but instead along the blowup sequence; for this, the authors combine in a very nice way an Almgren and an Alt-Caffarelli-Friedman monotonicity formula for rescaled solutions of system \eqref{equation:Laplacian}.

\smallbreak

Let us also point out that our work is also a natural follow up of the following theorem, obtained in \cite{HugoHolderVariable}, concerning $\alpha$-H\"older bounds ($\alpha \in ]0,1[$) for solutions of the system \eqref{equation}.

\begin{theorem_c}[{\cite[Theorem B.1]{HugoHolderVariable}}]
	\label{theoremForHolder}
	If $\{u_\beta\}_{\beta <0}$ is a family of solutions of \eqref{equation} satisfying the conditions \textbf{(A1)}, \textbf{(A2)}, \textbf{(F)}, \textbf{(a)} and \eqref{boundedness} then, for each $\alpha \in ]0,1[$ and $K \Subset \Omega$, there exists a constant $C$ such that
	$$
	\sup_{\beta < 0}
	\| u_{\beta}\|_{C^{0,\alpha}(K)}
	\leq
	C.
	$$
\end{theorem_c}

\begin{remark}
    A generalization of this result also easily follows using the same proof as in \cite{HugoHolderVariable}. Let, for each $\beta<0$, $A_{\beta}(\cdot) \in C^\infty(\Omega, \text{Sym}^{N\times N})$ be a matrix satisfying conditions \textbf{(A1)} and \textbf{(A2)} uniformly in $\beta$, and $f_{i,\beta} \in C^\infty(\Omega,\mathbb{R})$ satisfying \eqref{boundForF} uniformly. Take $v_\beta = (v_{1,\beta},...,v_{l,\beta})$  a positive solution of the system 
    \begin{equation*}
        -\div(
            A_\beta(x)
            \nabla v_{i,\beta}
        )
        =
        f_{i,\beta}
        (x,v_{i,\beta})
        +
        a(x) \beta
        \sum_{j\neq i} |v_{i,\beta}|^{\gamma-1}
        v_{i,\beta}
        |v_{j,\beta}|^{\gamma+1}
    \end{equation*}
 satisfying $\|v_\beta\|_{L^\infty(\Omega)}<m$ for some $m>0$,  Then, for each $\alpha \in ]0,1[$ and $K \Subset \Omega$, there exists $C>0$ depending on $\alpha, K,\gamma,N$ and the constants in the conditions \textbf{(A1)},\textbf{(A2)},\textbf{(F)},\textbf{(a)} such that
    $$
        \sup_{\beta<0}
        \|u_\beta\|_{C^{0,\alpha}(K)}
        \leq
        C.
    $$
\end{remark}

As a matter of fact, as we discuss next, these results that give uniform H\"older bounds are absolutely essential in our arguments to prove Theorem \ref{DesiredTheorem}. They will, in particular, be used in Section \ref{chapter:resultsChap4} to prove an Alt-Caffarelli-Friedman type monotonicity formula (see Lemma \ref{sphereLemma}). This is a key difference between our approach and the one used in \cite{SoaveZilio}.

\vspace{2mm}

\subsection{Structure of the paper and proof strategy}\label{subsec:structure}

We now give a brief description of the structure of this work and of the proof of Theorem \ref{DesiredTheorem}. The proof follows the blowup argument and the scheme found in \cite{SoaveZilio} (where, we recall, Theorem B is proved, which corresponds to Theorem \ref{DesiredTheorem} with $A(x) = Id$, $a(x)$). It relies on two monotonicity formulas, the Almgren monotonicity formula and the Alt-Caffarelli-Friedman monotonicity formula that are proved for blowup sequences (and not only for the blowup limits). 

Apart from the natural technical issues that arise from the fact that we have a more complicated operator, the main difficulty in our case is how to  generalize appropriately these monotonicity formulas to our setting of divergence type operators with variable coefficients.

\smallbreak

As stated before, the proof is based on a \emph{contradiction argument} using a normalized blowup sequence. It is a blowup done along the points $x_n$ where $\max_{j=1,...,l}|\nabla u_{j,\beta_n}|$ attains its maximum, for a sequence $\beta_n \rightarrow -\infty$. The normalization is done in such a way that the new sequence has bounded Lipschitz seminorm. Section \ref{chapter:background} is devoted to analyzing this blowup sequence and its properties. The results found are generalizations of arguments in \cite{SoaveZilio}, with adaptations for the variable coefficient case.

\smallbreak

Section \ref{chapter:implementation} is devoted to a generalization of the Almgren monotonicity formula for the variable coefficients case. In the context of limits $\beta\to -\infty$ of solutions to systems \eqref{equation:Laplacian} or to limits of blowup sequences, this formula has been used for instance in \cite{CaffarelliLin,uniformHolderBoundsHugoTerraciniNoris, SOAVE2016388, HugoTerraciniWeakReflectionLaw}; see \cite[Appendixes B \& C]{HugoHolderVariable}  for the case of system \eqref{equation}. In the later case, a crucial point is to perform a change of variables, changing (locally) the operator to become a perturbation of the Laplacian (see \eqref{tildeU} below, which is inspired by the previous works  \cite{Mariana2,GAROFALO2014682,Kukavica,SoaveWeth}). Here, we generalize \cite{SoaveZilio,SphereDensityAlt} (which deal with the Laplacian case) and prove an Almgren monotonicity formula for blowup \emph{sequences} associated to \eqref{equation}, see Theorem \ref{almgrenMonotonicity} below.  Since our objective is to obtain a monotonicity formula for the blowup sequence, and not the limit, there are extra terms that have to be considered. In the article \cite{SoaveZilio}, these terms are circumvented by taking the dimension to satisfy $\frac{\gamma N}{\gamma + 1} \leq 2$ (which implies $N\leq 4$). In our case, we can only obtain a monotonicity formula when the inequality is strict, that is $\frac{\gamma N}{\gamma + 1} < 2$ (which gives $N\leq 3$). This is due to extra terms coming from the variable coefficients, and it is the only place in the paper where the restriction is needed (see the proof of Theorem \ref{almgrenMonotonicity} for the details, in particular inequality \eqref{eq:RESTRICTION}).

\smallbreak

Section \ref{chapter:resultsChap4} is where we prove Theorem \ref{AltCaffMonotonicity}, which is a  generalization of the Alt-Caffarelli-Friedman formula found in \cite[Theorem 3.14]{SoaveZilio}. This is where our work differs the most from previous proofs, and it is one of the main contributions of our paper. Regarding this topic, there are two main problems in working with operators with variable coefficients.

Firstly, recall that the core of the proof of the classical Alt-Caffarelli-Friedman formula \cite{FriedmanSphere} is a result about a spectral optimal partition problem on the sphere, which says that:
\begin{equation}
\label{minimizationOnSphereBasic123}
\min
\{
\gamma(\lambda_1(\Omega_1))
+
\gamma(\lambda_1(\Omega_2))
:
\Omega_1,
\Omega_2
\subset 
\partial B_1,
\,
\Omega_1\cap \Omega_2 = \emptyset
\}
\geq
2,
\end{equation}
where $\lambda_1(\Omega)$ is the first Dirichlet eigenvalue of the Laplace-Beltrami operator on the sphere, $\Delta_\theta$, of the set $\Omega$, and $\gamma(t) = \sqrt{(\frac{N-2}{2})^2 + t} - \frac{N-2}{2}$. The proof of \cite[Theorem 3.14]{SoaveZilio} relies on a lower bound of a certain functional defined on the sphere, which is similar to the one found in \cite[Lemma 4.2]{SphereDensityAlt}, but with extra terms to account for the remaining terms in equation \eqref{equation:Laplacian}. Since in these papers (or in \eqref{minimizationOnSphereBasic123}) the functionals are related to the Laplacian, the proofs use a symmetrization argument which simplifies the procedure. This is not possible in our case due to the variable coefficients in our equations.  The result in our work, in the form of Lemma \ref{sphereLemma}, even though has the same structure of \cite{SoaveZilio, SphereDensityAlt}, obtains similar bounds through very different approaches. In particular, due to the lack of symmetrization, we cannot conclude the minimizing functions are uniformly Lipschitz, and to circumvent this we use Theorem C to obtain uniform H\"older bounds and make nontrivial use of the equation in a way it is enough for our purposes.

Secondly, the other main idea of the classical Alt-Caffarelli-Friedman formula is that (in dimension $N\geq 3$), $|y|^{2-N}$ is a fundamental solution of the Laplacian, that is $-\Delta (|y|^{2-N}) = C\delta$, for some $C$ depending on the dimension $N$. In our case, we are dealing (after a change of variables) with an operator $-\div(\tilde A_n(x)\nabla (\cdot))$, where $\tilde A_n(y) \sim Id$ for $y$ close to the origin.  The idea is to approximate this operator by $-\Delta(|y|^{2-N})$ plus an ``error'' term, and then use  Almgren's monotonicity formula to bound this error term. This allows an estimate like Lemma \ref{derivativeLemma}, and then a generalization of \eqref{minimizationOnSphereBasic123} in the form of Lemma \ref{sphereLemma}, which is the core of the proof of the monotonicity formula.

\smallbreak

Section \ref{sec:Lip} contains the proof of Theorem \ref{DesiredTheorem}.
In this section, more refined properties of the blowup sequence related to the Almgren monotonicity formula are studied. One also shows that there exists a radius $R_0>0$ such that two components of the limit of blowup sequences are nontrivial in $B_{R_0}$. This nontriviality is used to show that the Alt-Caffarelli-Friedman monotonicity formula can't go to zero in the limit. This is where the Alt-Caffarelli-Friedman and Almgren's monotonicity formulas are combined to obtain a contradiction on the blowup sequence, concluding the proof.  Here we follow \cite[Section 4]{SoaveZilio}, but adjustments for the variable coefficients case are (again) needed. 
\smallbreak

Section \ref{sec:conditions} is devoted to proving that the conditions of the Alt-Caffarelli-Friedman formula of Section \ref{chapter:resultsChap4} are satisfied for the blowup sequence. This is based on the characterization of certain limits of certain blowups and blowdowns. For this, we make use of some of the theorems from Section \ref{chapter:implementation}, in particular  Almgren's monotonicity formula.

In the appendices, we present important results that are used throughout the paper. Appendix \ref{chapter:sphereDivergence} shows a relation between the divergence operator on the $N-1$ dimensional sphere and the divergence in $\mathbb{R}^{N-1}$ through a stereographic projection. Appendix \ref{appendix:classG} makes a quick overview of results for functions that belong to the class $\mathcal{G}(\Omega)$ introduced in \cite{HugoTerraciniWeakReflectionLaw}; this is a set which has a strong relation with blowups of competitive systems. Appendix \ref{chapter:mult1PointsApp} is where one states (based on \cite{MULT1DANCERWANG}) that for limits $\lim_{\beta}u_\beta = v \in \mathcal{G}(\Omega)$ of competition systems like \eqref{equation}, when $\beta \rightarrow -\infty$, and points $x_0$ where $v(x_0) = 0$, then for every neighborhood $V_{x_0}$ of $x_0$ we must have two nontrivial components of $v = (v_1,...,v_l)$.  Finally, in Appendix \ref{chapter:AuxResultsApp} we collect other results.

\smallbreak

We conclude this introduction with one remark and an immediate corollary of Theorem \ref{DesiredTheorem}

\begin{remark}\label{rem:dependence_of_constant}
We notice that, by \textbf{(F)}, given $m>0$ there exists $d>0$ such that for all $i \in \{1,...,l\}$:
$$
    \sup_{x\in \Omega, s \in [-m,m]}
    |f_i(x,s)|
\leq
    d|s|.
$$
Combining the observation above with \eqref{boundedness} we obtain the existence of  $d >0$ such that
\begin{equation}
	\tag{\textbf{Fd}}\label{boundForF}
	|f_i(x,u_{i,\beta}(x))|
\leq
    d|u_{i,\beta}(x)| \qquad \text{for all $i \in \{1,...,l\}$.}
\end{equation}
The constant $C$ in Theorem \ref{DesiredTheorem} depends on the dimension $N$, the exponent $\gamma$, the compact $K$, the ellipticity constant $\theta$ in \textbf{(A1)}, and the upper bounds $M,m$ from \textbf{(M)} and \eqref{boundedness}.
\end{remark}

Our main result, Theorem \ref{DesiredTheorem}, has a direct correspondence with systems with the Laplace-Beltrami operator defined on Riemannian manifolds.

\begin{cor}\label{cor_Riem} Let $(\mathcal{M},g)$ be a $C^1$ Riemannian manifold, and consider $\{u_\beta\}_{\beta < 0}$  a family of positive solutions of the system
\begin{equation} \label{systemRiem}
-\Delta_g u=  f_i(x,u_{i,\beta}) + \beta |u_{i, \beta}|^{\gamma -1}u_{i, \beta} 
	\mathop{\sum_{j=1}^l}_{j\neq i} |u_{j, \beta}|^{\gamma + 1}\qquad \text{ in } \mathcal{M}
\end{equation}
for $i=1,\ldots, l$, under the assumption \textbf{(F)} for $f_i$, $N,\gamma\geq 1$. Assume moreover that $\frac{\gamma N}{\gamma + 1} < 2$, and that the sequence of solutions is uniformly bounded in $L^\infty$--norm \eqref{boundedness}.
Then, given $K \Subset \mathcal{M}$, there exists a constant $C>0$ such that 
	$$
	\sup_{\beta <0}
	\|\nabla u_{i,\beta}\|_{L^\infty(K)}\leq
	C\qquad \text{ for all $i \in \{1,...,l\}$.}
	$$
\end{cor}
\noindent Indeed, when using local coordinates in a small neighborhood of each point, \eqref{systemRiem} turns into \eqref{equation}, where $A(x)$ and $a(x)$ contain information about the metric $g$.

\section{Contradiction argument and blowup sequences}
\label{chapter:background}

Let us suppose, without loss of generality, that $B_3 \subset \Omega$. Within this section, we work with dimension $N\geq 1$, for $A$ satisfying assumptions \textbf{(A1)} and \textbf{(A1)}, $f$ satisfying \textbf{(F)} and $a$ satisfying 
\textbf{(a)}, we take $u_\beta$ to be a family of solutions of the system of equations \eqref{equation} that has a uniform $L^\infty(B_3)$ bound: for some $m>0$,
$$
	\sup_{\beta<0}  \|u_{\beta}\|_{L^\infty(B_3)} \leq m.
$$
Our goal is to show uniform Lipschitz bounds in $B_1$.
Assume, by contradiction, that there exists a sequence $\beta_n \rightarrow - \infty$ such that:
\begin{equation*}
\sup_{i = 1,...,l}\|\nabla u_{i,\beta_n}\|_{L^\infty(B_1)}
\to
\infty
\end{equation*}
In the spirit of \cite{HugoHolderVariable, SOAVE2016388, SoaveZilio, SphereDensityAlt},  since we want to localize the argument, we introduce a smooth cut-off function  $0 \leq \eta(x) \leq 1$  such that $\eta(x) = 1$ for $x \in B_1$ and $\eta(x) = 0$ for $x \in \mathbb{R}^N\setminus B_2$ and observe that  the contradiction assumption yields
\begin{equation}
	\label{eq:contradictionassumption}
	L_n 
	:= 
	\sup_{i=1,...,l} 
	\sup_{x \in \overline{B_2}} 
	|\nabla (\eta u_{i_n,\beta_n})| \rightarrow \infty
\end{equation}
as $n \rightarrow \infty$. For each $n \in \mathbb{N}$, there is a point $x_n \in \overline{B}_2$ such that $L_n = |\nabla (\eta u_{i,\beta_n})(x_n)|$ for some $i_n$. We can assume, without loss of generality (by possibly extracting another subsequence and relabelling components), that:
\begin{itemize}
\item $L_n=|\nabla (\eta u_{1,\beta_n})(x_n)|$, i.e.,  the maximum is always attained at the first coordinate;
	\item  $x_n \to x_\infty\in \overline{B}_2$.
	\end{itemize}

We define:
\begin{equation}
	\label{tildeU}
	\tilde{u}_{i,\beta_{n}}(x) := u_{i,\beta_n}(x_n+A(x_n)^{\frac{1}{2}}x).
\end{equation}
and consider the matrix function given by:
\begin{equation}
	\label{tildeMatrix}
	\tilde{A}_n(x) 
:= 
    A(x_n)^{-\frac{1}{2}}
	A(x_n + A(x_n)^{\frac{1}{2}}x)
	A(x_n)^{-\frac{1}{2}},\qquad \text{ which is such that } \qquad \tilde{A}_n(0)=Id
\end{equation}

\begin{lemma}
	\label{lemmaForMatrixTilde}
	We have:
	\begin{equation}
	\label{tildeEquation}
		-\div(\tilde{A}_n(x)\nabla \tilde{u}_{i,\beta_n})
	=
		f_i(x_n+A(x_n)^{\frac{1}{2}}x,\tilde{u}_{i,\beta_n}) 
	+ 
		a(x_n+A(x_n)^{\frac{1}{2}}x)
		\mathop{\sum_{j=1}^l}_{j\neq i}\beta_n 
		|\tilde{u}_{j, \beta_n}|^{\gamma + 1}
		|\tilde{u}_{i, \beta_n}|^{\gamma -1}\tilde{u}_{i, \beta_n},
	\end{equation}
	for $x \in A(x_n)^{-\frac{1}{2}}(B_3 - x_n)$,
	where $\tilde{A}_n$ is the matrix in \eqref{tildeMatrix}.
	Moreover:
	\begin{enumerate}
	\item  $B_{{1}/{M^\frac{1}{2}}} \subset A(x_n)^{-\frac{1}{2}}(B_3 - x_n)$, 
	\item $\langle \tilde{A}_n(x)\xi, \xi \rangle \geq \frac{\theta}{M} |\xi|^2$ for every $\xi \in \mathbb{R}^N$,
	\item  there exists $C = C(M, \theta)>0$ such that $\|D\tilde{A}_n\|_{L^\infty} \leq C$ and $\|\tilde{A}_n\|_{L^\infty} \leq C$ .
	\end{enumerate}
\end{lemma}

\begin{proof}
	
Equation \eqref{tildeEquation} follows from a straightforward computation. For (1),  we start by observing that, by \textbf{(A2)}, we have $|A(x_n)^\frac{1}{2}\xi |^2 =\langle A(x_n)\xi, \xi \rangle \leq M |\xi|^2$ for every $\xi\in \R^N$. Hence, since $x_n\in \overline{B}_2$, we have (1) and $|\xi|^2\leq M|A(x_n)^{-\frac{1}{2}}\xi|^2$. From this last fact and using also \textbf{(A1)}:
\[
		\langle \tilde{A}_n(x)\xi, \xi \rangle
		=
		\langle A(x_n + A(x_n)^{\frac{1}{2}}x)
		A(x_n)^{-\frac{1}{2}}\xi,A(x_n)^{-\frac{1}{2}} \xi \rangle \geq
		\theta|A(x_n)^{-\frac{1}{2}}\xi|^2 \geq 
		\frac{\theta}{M}|\xi|^2,
	\]
which is (2). Finally, by hypothesis \textbf{(A2)}, we have:
$$
  | \partial_{x_k}({A}_{ij}(x_n + A(x_n)^{\frac{1}{2}}x))|
=
|    \langle 
        \nabla A_{ij}(x_n + A(x_n)^{\frac{1}{2}}x),
        A(x_n)^{\frac{1}{2}}e_k     \rangle | \leq  \kappa \|DA\|_{L^\infty(\Omega)}
    \cdot
    \|A(x_n)^\frac{1}{2}\| \leq \kappa M^\frac{3}{2},
$$
where $\{e_j\}_{j=1,...,N}$ is the canonical basis for $\mathbb{R}^N$. This, combined with the fact that  $\|A(x_n)^{-\frac{1}{2}}\|_{L^\infty} \leq \theta^{-\frac{1}{2}}$ and $\tilde A_n(0)=Id$, implies (3).
\end{proof}

We present a result regarding the limit of $\tilde{u}_{\beta_n}$. This result will be a bit lateral for now, but will be necessary in the proof of Lemma \ref{radGoToZero} below.
\begin{lemma}
	\label{remarkOnStuff} The sequence $\{\tilde{u}_{\beta_n}\}$ is bounded in $C^{0,\alpha}(B_{1/(2M^\frac{1}{2})})$ and there exists  $\tilde{u}_\infty \in C^{0,\alpha}(B_{1/(2M^\frac{1}{2})})\cap H^1(B_{1/(2M^\frac{1}{2})})$ such that, up to a subsequence, $\tilde{u}_{\beta_n} \rightarrow \tilde{u}_\infty$ in $C^{0,\alpha}(B_{1/(2M^\frac{1}{2})})\cap H^1(B_{1/(2M^\frac{1}{2})})$ for every $\alpha\in (0,1)$. 	Moreover, $\tilde{u}_\infty(0) = 0$. 
\end{lemma}
\begin{proof}
	By results in the first sentence follow directly from Theorem C and reasoning exaclty as in \cite[Theorem 1.4]{uniformHolderBoundsHugoTerraciniNoris} for the strong convergence in $H^1$. 
	
Now, if $\tilde{u}_\infty(0)> 0$, then there exist $\epsilon,r C>0$, $i \in \{1,...,l\}$ such that  $\tilde{u}_{i,\beta_n}(x)\geq \epsilon$ for $x \in B_r$ and $n$ large enough.  We now have two cases. 

If $i\neq 1$, then by the equation of $\tilde{u}_{1,\beta_n}$ we have
	$$
	    -\div(\tilde{A}_n(x)\nabla \tilde{u}_{1,\beta_n})
	\leq
	    -
	    |M_n|
	    \delta
	    \epsilon^{\gamma}
	    \tilde{u}_{1,\beta_n}^{\gamma}
	    +
	    dm
	    \qquad
	    \forall x \in B_{r}
	$$
	and so, by Lemma \ref{estimateLemma}-(1), we conclude the existence of $c>0$ such that $|M_n|\tilde{u}_{1,\beta_n}^{\gamma}(x)<c$ for $x \in B_{r/2}$ and so (going back to the equation of  $\tilde{u}_{1,\beta_n}$), $|\div(\tilde{A}_n(x)\nabla \tilde{u}_{1,\beta_n})|$ is uniformly bounded in $n$. By elliptic regularity theory, we must have that, up to a subsequence, $\tilde{u}_{1,\beta_n} \rightarrow \tilde{u}_{1,\infty}$ in $C^1(B_{r/2})$, in contradiction with the fact that the gradient $|\nabla \tilde{u}_{1,\beta_n}(0)|$ blows up.
	
	On the other hand, if $i = 1$, then for all $j \neq 1$ we can apply the argument above to conclude that there exists $c>0$ such that $|M_n| \tilde{u}_{j,n}(x) < c$ for all $x \in B_{r/2}$, from which we conclude a uniform bound for $|\div(\tilde{A}_n(x)\nabla \tilde{u}_{1,\beta_n})|$. This, once again, leads to a contradiction.

\end{proof}

Consider now the blowup sequences given by:
\begin{align}
	&v_{i,n}(x) := \eta(x_n) \frac{\tilde{u}_{i,\beta_n}(r_nx)}{L_n r_n} \label{blowUp}\\
	&
	\overline{v}_{i,n}(x) 
	:= \frac{(\eta u_{i, \beta_n})(x_n + r_n A(x_n)^{\frac{1}{2}} x)}{L_n r_n},\nonumber
\end{align}
where we consider each function defined in:
\begin{equation}
	\label{Domain}
	\Omega_n := \frac{A(x_n)^{-\frac{1}{2}}(B_3 - x_n)}{r_n}, \quad \text{ where} \quad  0< r_n := \sum_{i=1}^l\frac{(\eta u_{i,\beta_n})(x_n)}{L_n}\to 0.
\end{equation} 
The fact that $r_n\to 0$ is a consequence of the bound $r_n\leq lm/L_n \to 0$, using the uniform boundedness from \eqref{boundedness} and the contradiction assumption \eqref{eq:contradictionassumption}. On the other hand, to show that $r_n>0$, notice that, since $|\nabla(\eta u_{1,\beta_n})(x_n)| = L_n>0$, we must have $(\eta u_{1,\beta_n})(x_n)>0$, otherwise around the point $x_n$ the function $(\eta u_{1,\beta_n})$ would take negative values, which is a contradiction.

With this choice of $r_n$, we have the normalization
\begin{equation}
\label{rnDefinition}
	\sum_{i=1}^l\overline{v}_{i,n}(0) 
= 
	\sum_{i=1}^l\frac{(\eta u_{i, \beta_n})(x_n)}{L_n r_n} 
= 
	1.
\end{equation}

Define
\begin{align*}
	&A_n(y) = \tilde{A}_n(r_n y) =  A(x_n)^{-\frac{1}{2}}
	A(x_n + r_nA(x_n)^{\frac{1}{2}}y)
	A(x_n)^{-\frac{1}{2}},\\
	&a_n(y) = a(x_n + r_nA(x_n)^{\frac{1}{2}}y),\\
	&f_{i,n}(y,t) = \frac{r_n \eta(x_n)}{L_n}f_i(x_n + r_nA(x_n)^{\frac{1}{2}} y, \frac{L_n r_n}{\eta(x_n)} t).
\end{align*}

\begin{lemma}
    \label{lemmaMatrixBlowLimit}
We have 	$A_n \in C^\infty(\Omega_n, \text{Sym}^{N\times N})$ and:
	\begin{align}
		\label{ellipticConstantForBlowUp}
		\begin{split}
			\|A_n(y) - Id\| \leq Cr_n |y|, \quad \|DA_n\|_{L^\infty(B_r)}\leq Cr_n, \quad
			\langle A_n(y)\xi,\xi \rangle\geq \frac{\theta}{M}|\xi|^2.
		\end{split}
	\end{align}
	Also $a_n(\cdot) \rightarrow a(x_\infty)$ locally uniformly in $\Omega_n$, $f_{i,n}(x,v_{i,n}(\cdot)) \rightarrow 0$ uniformly in $\Omega_n$,
 and $|f_{i,n}(x,v_{i,n}(x))|\leq dr_n^2|v_{i,n}(x)|$.
\end{lemma}
\begin{proof}The first three inequalities follow by Lemma \ref{lemmaForMatrixTilde} and by \textbf{(A2)}.

Let $K \Subset \Omega_n$ be a compact set. Using \textbf{(a)}, we obtain the existence of $C>0$ also such that:
$$
	|a_n(y)-a(x_n)| \leq C r_n|y|
$$
for all $y \in K$. Thus $a_n \rightarrow a(x_\infty) = \lim_n a(x_n)$ for each compact set in $\Omega_n$.

For the last information about $f_n$, using \eqref{boundForF} we obtain:
\begin{align*}
    |f_{i,n}(x,v_{i,n}(x))|
= &
    |\frac{r_n \eta(x_n)}{L_n}f(x_n + r_nA(x_n)^{\frac{1}{2}} x, u_{i,n}(x_n + r_nA(x_n)^{\frac{1}{2}} x))|\\
\leq &
    d\frac{r_n \eta(x_n)}{L_n}|u_{i,n}(x_n + r_nA(x_n)^{\frac{1}{2}} x)|
=  dr_n^2 |v_{i,n}(x)|; \qedhere
\end{align*}
the uniform convergence to 0 is now a consequence of the fact that $r_n\to 0$, $L_n\to 0$ and the uniform bounds for $u_\beta$.
\end{proof}

\begin{prop}
	\label{ListProp}
	The following are satisfied:
	\begin{enumerate}
		\item We have that $
		B_{1/(M^{\frac{1}{2}} r_n)} \subset \Omega_n,
		$ so $\Omega_n$ exhaust the whole $\mathbb{R}^N$ as $n\to \infty$.
		\item The sequence $v_{i,n}$ satisfies the equation:
		\begin{equation}\label{eq:system_rescaled}
		-\div(A_n(y)\nabla v_{i,n}) = f_{i,n}(y, v_{i,n}) + M_na_n(y)\mathop{\sum_{j=1}^l}_{j\neq i} |v_{j,n}|^{\gamma+1} |v_{i,n}|^{\gamma-1}v_{i,n}(y)
		\qquad
		\forall
		y \in \Omega_n
		\end{equation}
		with $M_n := \beta_n(\frac{ L_n}{\eta(x_n)})^{2\gamma} r_n^{2\gamma+2}$.
		\item The sequence $\overline{v}_n$ has a uniform Lipschitz bound:
		\begin{equation}\label{eq:boundedgradient}
		\sup_{i=1,...,l}
		\sup_{x \in \Omega_n}
		|\nabla \overline{v}_{i,n}|
		=
		\sup_{i=1,...,l}\, \mathop{\sup_{x,y \in \Omega_n}}_{x\neq y}\frac{|\overline{v}_{i,n}(x)-\overline{v}_{i,n}(y)|}{|x-y|} \leq M^{1/2},
		\end{equation}
		for some $C>0$. Also,
\begin{equation}\label{bounds_nabla_at_origin}
		|\nabla \overline{v}_{1,n}(0)|  \geq 		\theta^\frac{1}{2}\quad \text{ and } \lim \inf 		|\nabla v_{1,n}(0)| \geq \theta^{\frac{1}{2}}.
		\end{equation}
		\item There exists $v = (v_1,...,v_l) \in C(\mathbb{R}^N,\mathbb{R}^l)$ such that $v_n, \overline{v}_n  \rightarrow v$ in $L^\infty_{loc}(\mathbb{R}^N)$.
		\item Moreover, $v_{n} \rightarrow v$ in $H^1_{loc}(\mathbb{R}^N)$ and, for any $r$, there exists $C_r$ such that:
		\begin{equation}\label{eq:uniform_nonvariational}
		\int_{B_r} |M_n|a_n(y)
		\mathop{\sum_{j=1}^l}_{j\neq i} |v_{j,n}|^{\gamma+1}|v_{i,n}|^{\gamma-1}v_{i,n}dy \leq C_r
		\end{equation}
		for every $i \in \{1,...,l\}$.
		In particular, if $M_n \rightarrow -\infty$, then $v_{i,n}\cdot v_{j,n} \rightarrow 0$ whenever $j \neq i$.
	\end{enumerate}
\end{prop}
\begin{proof} 
	Items (1) and (2) follows directly from  the definitions and Lemma \ref{lemmaForMatrixTilde}.	 
	
\noindent	\textbf{Proof of (3).} By using \textbf{(A2)} and the fact that $\sup_{x \in B_2}|\nabla(\eta u_{i,\beta_n}(x))|\leq L_n$ for all $i \in \{1,...,l\}$ and $\eta=0$ on $\R^N\setminus B_2$:
	\begin{align*}
		|\nabla \overline{v}_{i,n}(y)|
		&=
		\frac{
			|A(x_n)^{\frac{1}{2}}\nabla(\eta u_{i,\beta_{n}})(x_n+r_nA_n(x_n)^{\frac{1}{2}}y)|
		}{
			L_n
		}\leq M^{1/2}
		\quad \quad \forall i = 1,...,l.
	\end{align*}
	On the other hand, we have:
	\begin{align*}
			|\nabla \overline{v}_{1,n}(0)|
			&=
			\frac{
				|A(x_n)^{\frac{1}{2}}\nabla(\eta u_{1,\beta_{n}})(x_n)|
			}{
				L_n
			}=
			\frac{1}{L_n}\langle A(x_n)^\frac{1}{2}\nabla(\eta u_{1,\beta_n} )(x_n),A(x_n)^\frac{1}{2}\nabla(\eta u_{1,\beta_n} )(x_n)  \rangle^{\frac{1}{2}}
			=\\
			&=
			\frac{1}{L_n}\langle A(x_n)\nabla(\eta u_{1,\beta_n} )(x_n),\nabla(\eta u_{1,\beta_n} )(x_n) \rangle^{\frac{1}{2}}
			\geq \theta^{\frac{1}{2}}\frac{
				|\nabla(\eta u_{1,\beta_n} )(x_n)|
			}
			{
				L_n
			} = \theta^{\frac{1}{2}},
	\end{align*}
	by using hypothesis \textbf{(A1)}.
	Moreover,
	$$
	\nabla \overline{v}_{1,n}(0) = \frac{u_{1,\beta_n}(x_n) A(x_n)^{\frac{1}{2}}\nabla \eta(x_n)}{L_n} + \frac{A(x_n)^{\frac{1}{2}}\nabla u_{1,\beta_n}(x_n) \eta(x_n)}{L_n} 
	= 
	O(\frac{1}{L_n}) + \nabla v_{1,n}(0)
	$$
	(using this time conditions \textbf{(A2)} and \eqref{boundedness}),
	concluding the result.
	
	\vspace{2mm}
	
\noindent \textbf{Proof of (4).} Since the sequence $\bar{v}_n$ is bounded at $0$ (recall \eqref{rnDefinition}) and has bounded gradient (recall \eqref{eq:boundedgradient}), by Ascoli-Arzela's Theorem one can find $v \in C(\mathbb{R}^N)$ such that $\bar v_n \rightarrow v$ uniformly over any compact set (up to a subsequence).
	
	Now, for any $K \subset \mathbb{R}^N$ compact and $y \in K$, we have:
	$$
	|\overline{v}_{i,n}(y) - v_{i,n}(y)| 
	\leq 
	\sup_{y \in K}
	\frac{
		u_{i,\beta_n}(x_n+r_nA(x_n)^{\frac{1}{2}}y)
	}{
		L_n 
	}\frac{
	|\eta(x_n) - \eta(x_n+r_nA(x_n)^{\frac{1}{2}}y)| }{r_n} \leq \frac{C}{L_n}
	\rightarrow 
	0
	$$
we also conclude that $v_{i,n} \rightarrow v_i$ uniformly over any compact set in $\mathbb{R}^N$, for all $i \in \{1,...l\}$.
	
	\vspace{2mm}
	
\noindent \textbf{Proof of (5).} For $r>0$ fixed, test the equation of $v_{i,n}$ against a function $\phi \in C_c^\infty(B_{2r})$ satisfying $0 \leq \phi\leq 1$ and $\phi(x) = 1$ for $x \in B_r$. Thus, for large $n$, there exists $C_r$ depending only on $r$ such that
	\begin{align*}
		&\mathop{\sum_{j=1}^l}_{j\neq i}
		\int_{B_r} a_n(y)|M_n| |v_{j,n}|^{\gamma+1}|v_{i,n}|^{\gamma-1}v_{i,n}dy
		\leq 
		\mathop{\sum_{j=1}^l}_{j\neq i}
		-\int_{B_{2r}} 
		a_n(y)
		M_n |v_{j,n}|^{\gamma+1}|v_{i,n}|^{\gamma-1}v_{i,n}\phi\\
		&= \large|\int_{B_{2r}} f_{i,n}(y,v_{i,n})\phi + \langle 
			A_n(y)\nabla v_{i,n},
			\nabla \phi
		\rangle
		\large|=
		\large|
		\int_{B_{2r}} f_{i,n}(y,v_{i,n})\phi - 
			v_{i,n}
			\div(
			A_n(y)
			\nabla \phi
			)
		dy
		\large|\\
		&\leq  
		|B_{2r}|\cdot
		\|v_{i,n}\|_{L^\infty(B_{2r})}
		\large(d
		+
		C\|DA_n\|_{L^\infty(B_{2r})}
		\|\phi\|_{C^1(B_{2r})}+
		C\|A_n\|_{L^\infty(B_{2r})}
		\|\phi\|_{C^2(B_{2r})}
		\large) \leq C_r,		
	\end{align*}
where we have used Lemma \ref{lemmaMatrixBlowLimit}, the fact that  $v_{i,n}$ is uniformly bounded over compact sets (item (4)), and the fact that $\|DA_n\|_{L^\infty(B_{2r})}$ and $\|A_n\|_{L^\infty(B_{2r})}$ are uniformly bounded. This yields \eqref{eq:uniform_nonvariational}.

To prove the that $v_n\to v$ in $H^1_{loc}(\R^N)$, we test this time the equation of $v_{i,n}$ against $v_{i,n}\phi$. Using the ellipticity constant for $A_n$ given by \eqref{ellipticConstantForBlowUp}, we obtain:
	\begin{align}
		\frac{\theta}{M} &\int_{B_r}|\nabla v_{i,n}(x)|^2dx 
		\leq
		\int_{B_r}
		\langle A_n(x) \nabla v_{i,n}, \nabla v_{i,n}  \rangle 
		\leq
		\left|\int_{B_{2r}}\phi(x)\langle A_n(x) \nabla v_{i,n}, \nabla v_{i,n}\rangle dx\right| \label{boundOnDerivativeProp2}\\
		=& 
\left| \int_{B_{2r}} - \langle A_n(x)\nabla v_{i,n}, \nabla\phi		\rangle v_{i,n} 		f_{i,n}(x,v_{i,n}) v_{i,n}\phi + a_n(x)M_n\sum_{j\neq i} |v_{j,n}|^{\gamma+1}|v_{j,n}|^{\gamma-1}		v_{i,n}^2\phi				\right|\nonumber \\
		=& \left| \int_{B_{2r}} \frac{1}{2}  v_{i,n} \div\left(		A_n(x)\nabla\phi\right)		v_{i,n}+		f_{i,n}(x,v_{i,n}) v_{i,n}\phi + a_n(x)M_n\mathop{\sum_{j=1}^l}_{j\neq i} |v_{j,n}|^{\gamma+1}|v_{j,n}|^{\gamma-1}v_{i,n}^2\phi		 \right|,\nonumber
	\end{align}
	and this is uniformly bounded: the first two terms since $v_{i,n}$ and  $\div\left(
	A_n(x)\nabla\phi\right)$ are uniformly bounded in $B_{2r}$, and the last one by \eqref{eq:uniform_nonvariational}. In particular, $v_n$ is bounded in $H^1_{loc}(\R^N)$ and $v_n \rightharpoonup v_i$ weakly. Now given $K>0$, we consider the set:
$E_K = \large\{ s \in [0,r]:	\sup_{i, n}	\int_{\partial B_s}	|\nabla v_{i,n}|^2	d\sigma(y)	>	K	\large\}	$.
	Since there exists $C>0$ such that for all $n$:
	\begin{align*}
		C 
		\geq&
		\int_{B_r} |\nabla v_{i,n}|^2dx
		=
		\int_{0}^r 
		\int_{\partial B_s}
		|\nabla v_{i,n}|^2 
		d\sigma(x)
		dr,
	\end{align*}
	then this implies that $|E_K| \leq \frac{C}{K}$. Thus by taking $K$ large enough such that $|E_K| < \epsilon$, then we can choose a slightly smaller radius $r' \in [r-\epsilon, r]$ such that:
	\begin{align*}
		\int_{\partial B_{r'}} |\nabla v_{i,n}|^2dx
		\leq
		K
	\end{align*}
	for all $n \in \mathbb{N}$.
	For the rest of the proof, we will still call $r$ this slightly smaller radius.	Now we test the equation for $v_{i,n}$ against $v_{i,n}-v_i$ in $B_r$, obtaining
	\begin{align*}
		&\Big|\int_{B_r}\langle A_n(y) \nabla v_{i,n}, \nabla (v_{i,n}-v_{i}) \rangle \left|		=	\Big|\int_{\partial B_r} \langle A_n(y) \nu_y, \nabla v_{i,n} \rangle (v_{i,n} - v_{i}) d\sigma \right.\\  
		&\qquad \left. + \int_{B_r}				f_{i,n}(y,v_{i,n})(v_{i,n}-v_{i}) + M_n a_n(y)\sum_{j\neq i}|v_j|^{\gamma+1} |v_i|^\gamma (v_{i,n}-v_{i})\, dy \right|\\		
		\leq& \left(		M		\int_{\partial B_r}		|\nabla v_{i,n}| 		d\sigma	+ 		\int_{B_r} 				|f_{i,n}(y,v_{i,n})| 		+ 		|M_n| a_n(y) 		\sum_{j\neq i}		|v_{i,n}|^{\gamma+1}|v_{i,n}|^\gamma\, dy			\right)		\|v_{i,n}-v_{i}\|_{L^{\infty}(B_r)} \rightarrow 0
	\end{align*}
	since all the terms are bounded and $\|v_{i,n}-v_{i}\|_{L^\infty(B_r)} \rightarrow 0$.
 On the other hand, by the weak covergence of $v_{i,n}$ we have:
	$$\int_{B_r}\langle \nabla v_{i}, \nabla(v_i-v_{i,n}) \rangle dy \rightarrow 0,\qquad \text{ and so }\qquad \int_{B_r}\langle A_n(y) \nabla v_i, \nabla (v_i-v_{i,n}) \rangle dy \rightarrow 0.$$
	We conclude:
	$$
	\theta\int_{B_r} \large|
	\nabla (v_i - v_{i,n})|^2 dy 
	\leq 
	\int_{B_r} 
	\langle A_n(y)\nabla (v_i-v_{i,n}), \nabla (v_i-v_{i,n})\rangle dy  
	\rightarrow 
	0
	$$
therefore $v_{i,n}$ converges strongly in $H^1(B_r)$ to $v_i$. This, combined with \eqref{eq:uniform_nonvariational} and the lower bound of the function $a_n$ (cf.  \textbf{(a)}), 
$$
\sum_{j\neq i}^l\int_{B_r} \delta |v_{i}|^{\gamma+1}|v_{i}|^\gamma dx	\leq \lim \sum_{j\neq i}^l\int_{B_r} a_n(x) |v_{i,n}|^{\gamma+1}|v_{i,n}|^\gamma(x)dx \leq \lim \frac{C_r}{|M_n| }=0.
	$$
In conclusion, $v_i v_j = 0$ whenever $i \neq j$.
\end{proof}

\begin{lemma}
	\label{nonTrivial}
	Let   $v=(v_1,\ldots, v_l)$ be the limit of the  sequences $(v_n)$ and $(\bar v_n)$, provided by Proposition \ref{ListProp}-(4). Then the first component, $v_1$, is nonconstant. % namely $v_1\not\equiv 0$.
\end{lemma}

\begin{proof} We split the proof in two cases, according to the asymptotic behaviour of $M_n$.

\smallbreak

\noindent	Case 1. If $M_n$ is bounded, then the right hand side of \eqref{eq:system_rescaled} is also bounded. Then, by elliptic regularity theory, that $v_{1,n}\to v_1$ in $C^{1,\alpha}_{loc}$, for every $\alpha\in (0,1)$. Thus, by \eqref{bounds_nabla_at_origin}: 
	$$
	0 < \theta^{\frac{1}{2}}\leq \lim \inf|\nabla v_{1,n}(0)| = |\nabla v_1(0)|,
	$$ 
	and so $v_1$ is nonconstant.

Case 2. If $|M_n| \rightarrow \infty$, then we know the limit $v=(v_1,\ldots, v_k)$ has segregated components by Proposition \ref{ListProp}-\emph{6}, that is, $v_iv_j = 0$ whenever $i\neq j$. Since $ \bar v_n\to v$ uniformly in $B_r$, passing to the limit the normalization condition \eqref{rnDefinition} we have

	$$\sum_{j=1}^l {v}_j(0) = 1$$
Therefore we either have $v_1(0) = 1$ or $v_1(0)= 0$. 
	\smallbreak 
	First we suppose $v_1(0) = 0$ (and we check that this leads to a contradiction). Then there exists $h \in \{2,...,l\}$ such that $v_h(0) = 1$. By the uniform convergence $v_n\to v$ in $B_r$, we must have:
	\begin{equation}
	    \label{lowBoundSomewhere}
	    v_{h,n}(0)\geq \frac{7}{8} \qquad \text{ for $n$ large enough.}
	\end{equation}
	
	Using Proposition \ref{ListProp}-(3), there exists $C>0$ such that $\|\nabla \overline{v}_n\|_{L^\infty(B_r)} \leq C$. Therefore, since $v_n-\bar v_n\to 0$ locally uniformaly,  for $x \in B_{\frac{1}{2C}}$ we have
	$$|v_{h,n}(x) - v_{h,n}(0)| \leq |v_{h,n}(x) - \overline{v}_{h,n}(x)| + |\overline{v}_{h,n}(x) - \overline{v}_{h,n}(0)| \leq o(1) + C|x| \leq o(1)+ \frac{1}{2}$$
Therefore
\[
v_{h,n}(x) \geq \frac{1}{8}\qquad \text{ for $x \in B_{\frac{1}{2C}}$,\ $n$ large enough, }
\]
For every $\epsilon>0$, by Lemma \ref{lemmaMatrixBlowLimit} we also have
	$$
	\sup_{x\in B_{\frac{1}{2}}}|f_{i,n}(x,v_{i,n}(x))| < \epsilon \qquad \text{ for $n$ large};
	$$ 
combining this with $a_n(y)\geq \delta>0$ and inequality \eqref{lowBoundSomewhere}, and remembering that $M_n<0$, we are able to conclude that $v_{1,n}$ satisfies the inequalities:
	\begin{equation}
		\label{equationForLemmaBound}
		-\div\left(
		    A_n(x)\nabla v_{1,n}
		 \right) \leq 
		-
		\delta
		\left(
		  \frac{7}{8}
		\right)^{\gamma+1} |M_n|v_{1,n}^\gamma + \epsilon,
	\end{equation}
	\begin{equation}
		\label{equationForGoodBound}
		|\div(A_n(x)\nabla v_{1,n})| \leq \delta
		\left(
		  \frac{7}{8}
		\right)^{\gamma+1} |M_n|v_{1,n}^\gamma + \epsilon.
	\end{equation}
Using equation \eqref{equationForLemmaBound} together with Lemma \ref{estimateLemma}-(1), we obtain the existence of $c>0$ such that $|M_n|v_{1,n}^\gamma(x)< c$ for $x \in B_{\frac{1}{4C}}$. Plugging this information in \eqref{equationForGoodBound} , we conclude that  $|\div(A_n(x)\nabla v_{1,n}(x))|<C$ for all $n \in \mathbb{N}$. From the uniform ellipticity of $A_n(x)$ and by elliptic regularity theory, we conclude that (up to a subsequence) $v_{1,n} \rightarrow v_1$ in $C^1(B_{\frac{1}{4C}})$, and once again by \eqref{bounds_nabla_at_origin} we have: 
	$$
	|\nabla v_1(0)| = \lim_n |\nabla v_{1,n}(0)| \geq \theta^\frac{1}{2}>0,
	$$ 
	which is a contradiction since $v_1$ is positive and $v_1(0) = 0$. Thus we must have $v_1(0) = 1$.

	Since $v_1(0) = 1$, we can reason similarly to the previous paragraph, showing via Lemma \ref{estimateLemma}-(1) that $|M_n| v_{j,n}^\gamma(x)< C$ for $x \in B_{\frac{1}{4C}}$ and  $j \neq 1$. From this we conclude the boundedness of $|\div(A_n(x)\nabla v_{1,n})|$, so we can take the convergence $v_{n,1}\to v_1$ in $C^1(B_{\frac{1}{4C}})$ and
	$$
	|\nabla v_{1}(0)|=\lim_n|\nabla v_{1,n}(0)| \geq \theta^\frac{1}{2}>0,
	$$ concluding the fact that $v_1$ is nonconstant.
\end{proof}

Now we make a concluding proposition of this section, which summarizes what is known about the limiting profiles $v$.

\begin{prop}
	\label{concludeProp}
		Let   $v=(v_1,\ldots, v_l)\in H^1_{loc}(\R^N)\cap C(\R^N)$ be the limit of the  sequences $(v_n)$ and $(\bar v_n)$. Then
	\begin{equation}\label{eq:Lipschitzbound_limit}
	\max_{i=1,...,l}\mathop{\sup_{x\neq y}}_{x,y\in\mathbb{R}^N}\frac{|v_i(x)-v_i(y)|}{|x-y|}\leq C,
	\end{equation}
	and $v_1$ is nonconstant. Actually,  $v$ may only have at most another nontrivial component, say $v_2$. These two components satisfy
	\begin{equation*}
	v_1(0)+v_2(0)=1,
	\qquad
	|\nabla v_1(0)|\geq \theta^\frac{1}{2}>0,
	\end{equation*}
	and furthermore
	\begin{enumerate}
\item If $M_n \rightarrow -\infty$, then both $v_1$ and $v_2$ are subharmonic in $\mathbb{R}^N$ and:
	\begin{equation}
	\label{unboundedEquation}
	\begin{cases}
		-\Delta v_1 = 0 & \text{in } \, \{v_1>0\}\\
		-\Delta v_2 = 0 &  \text{in } \, \{v_2>0\}\\
		v_1\cdot v_2 = 0 & \text{in } \, \mathbb{R}^N\\
		v_1, v_2 \geq 0 & \text{in } \, \mathbb{R}^N
	\end{cases}
	\end{equation}
\item If $M_n$ is bounded then there exists $M_\infty <0$ such that up to a subsequence $M_n \rightarrow M_\infty < 0$ and:
	\begin{equation}
	\label{boundedEquation}
	\begin{cases}
		-\Delta v_1 = M_\infty v_1^\gamma v_2^{\gamma+1} & \text{in } \, \mathbb{R}^N\\
		-\Delta v_2 = M_\infty v_2^\gamma v_1^{\gamma+1} &  \text{in } \, \mathbb{R}^N\\
		v_1, v_2 \geq 0 & \text{in } \, \mathbb{R}^N
	\end{cases}
	\end{equation}
	\end{enumerate}
\end{prop}

\begin{proof}
	The first observation is just a consequence of Proposition \ref{ListProp}. To show that the limit can only have at most two nontrivial components and that these satisfy either \eqref{unboundedEquation} or \eqref{boundedEquation}, we divide the discussion in two cases.

\noindent \textbf{Case 1.} Suppose that $M_n \rightarrow -\infty$. Then, by Proposition \ref{ListProp}-(5), we know that $v_{i,n}\cdot v_{j,n} \rightarrow 0$ locally uniformly for all $j \neq i$, and so also $v_i\cdot v_j=0$ whenever $j \neq i$. 

Fix $i$ and let $\phi \in C_c^\infty(\{v_i>0\})$,  $K:= \text{supp}(\phi)$ and $\delta>0$ such that  $v_i|_{K}>\delta$. Exactly as we did in the proof of Case 2 in  Lemma \ref{nonTrivial}, we can  use Lemma \ref{estimateLemma}-(1) and show that we have that $|M_n\|v_{j,n}|^\gamma \leq C$ in $K$, for every each $j \neq i$  (just use the uniform boundedness of $v_{n}$ over this compact set and take a cover by balls). In particular, $|M_n\|v_{j,n}|^{\gamma+1}|v_{i,n}(x)|^\gamma \rightarrow 0$ uniformly in $K$

Thus, using the equation for $v_n$ in \eqref{eq:system_rescaled} and the fact that $A_n(y) \to  Id$  and $f_{i,n}(y,v_{i,n}(y)) 
\to 0$  over compact sets (recall Lemma \ref{lemmaMatrixBlowLimit}), we obtain:
	\begin{align*}
		\int_{\{v_i>0\}} \langle \nabla \phi,\nabla v_{i}\rangle dy
		&=
		\lim_n
		\int_{K}
		\langle A_n(y)\nabla \phi, \nabla v_{i,n}\rangle dy\\
		&= 
		\lim_n 	\int_{K}
		\left(
		f_{i,n}(y,v_{i,n})\phi(y)
		+
		M_na_n(y)
		\sum_{j\neq i}^l|v_{j,n}|^{\gamma+1}|v_{i,n}|^\gamma \phi(y)
		\right)dy=0,
	\end{align*}	
In conclusion, we have that:
	$$
	-\Delta v_i = 0\qquad \text{	in $\{v_i>0\}$.}
	$$
Since $v_i \in C(\mathbb{R}^N)$ are zero outside this set we conclude that the $v_i$ are subharmonic in $\mathbb{R}^N$. Using Lemma \ref{subHarmonicNonTrivial}, since we have the bound \eqref{eq:Lipschitzbound_limit}, we conclude that this system has at most two nontrivial components.

\noindent \textbf{Case 2.} In case $M_n \rightarrow M_\infty$ then using the fact that $f_{i,n}(y,v_{i,n}(y)) \rightarrow 0$ and that $a_n(x) \rightarrow a(x_\infty)>0$ uniformly over compact sets we conclude that:
	$$
	-\Delta v_i = M_\infty a(x_\infty)\sum_{j\neq i} v_{i}^{\gamma}v_{j}^{\gamma+1},
	$$
	since, by elliptic regularity, $v_{i,n}$ converges in $C^{1,\alpha}_{loc}(\mathbb{R}^N)$.  If $M_\infty = 0$, this would imply that the functions $v_i$ are harmonic and nonnegative which is a contradiction, thus $M_\infty < 0$. Using Lemma \ref{equationNonTrivial}, we conclude that $v$ only has two nontrivial components.
	
\end{proof}

Proposition \ref{concludeProp}, by itself, does not give any contradiction related to the limiting profiles $v$. We need to add more information about these limits; for this, in Section \ref{chapter:implementation} we explore Almgren's monotonicity formulas (a consequence of the variational structure of the approximating system) and in Sections \ref{chapter:resultsChap4} we show a Alt-Caffarelli-Friedman type monotonicity formula.
%%%%%%%%%%%%

%%%%%%%%%%%%%

\section{Almgren's Monotonicity Formula}
\label{chapter:implementation}

In this section we will deduce an  Almgren's Monotonicity formula. The formula will be stated for vector functions $u_\beta = (u_{1,\beta},...,u_{l,\beta})$, whose components are positive solutions of the system: 
\begin{equation}
	\label{AlmgrenEq}
	-\div(A(x)\nabla u_{i,\beta})= f_i(x,u_{i,\beta})+a(x)\beta 
	\mathop{\sum^l_{i,j=1}}_{j\neq i} u_{j,\beta}^{\gamma+1}u_{i,\beta}^\gamma
	\qquad
	i=1,...,l
\end{equation}
satisfying conditions \textbf{(A1)}, \textbf{(A2)}, \textbf{(F)} and \eqref{boundedness} where $\beta <0$. We also assume that the matrix $A$ satisfies:
\begin{equation*}
	A(0)=Id.
\end{equation*}
Thus, around zero, the operator $v \mapsto \div(A(x)\nabla v)$ is just a perturbation of the usual Laplacian. The main purpose of this section is to prove Theorem \ref{almgrenMonotonicity} below.
%where $A(0) = Id$. Thus, around zero, the operator $v \mapsto \div(A(x)\nabla v)$ is just a perturbation of the usual Laplacian. In this work we will be mainly following the work of \cite{HugoHolderVariable}.
 
 \smallbreak

We define:
$$
\mu(x) 
:= 
\langle A(x)\frac{x}{|x|}, \frac{x}{|x|} \rangle \geq \theta
\qquad
\forall x \in \mathbb{R}^N\setminus \{0\},
$$
The following lemma is taken from \cite{HugoHolderVariable} (which, in turn, is based on \cite{GAROFALO2014682}). 

\begin{lemma}[{{\cite[Lemma C.2]{HugoHolderVariable}}}]
	\label{listLemma}
	There exists a constant $C$ and radius $0<\tilde{r}< d(0, \partial \Omega)$, depending only on the dimension $N$ and $\theta, M$ (bounds from conditions \textbf{(A1)} and \textbf{(A2)}) such that, for $|x| < \tilde{r}$, we have:
	\begin{enumerate}
		\item $\|A(x) - Id\| \leq C|x|,$ 
		\item $|\mu(x)-1| \leq C|x|,$
		\item $|\frac{1}{\mu(x)} - 1| \leq C|x|,$. 
		\item $|\frac{1}{\mu^2(x)}-1| \leq \frac{C}{1-C|x|^2}|x|,$
		\item $|\nabla \mu(x)| \leq C,$
		\item $|\div(A(x)\nabla |x|)-\frac{N-1}{|x|}|\leq C,$ 
		\item $|\div(\frac{A(x)x}{\mu(x)}) - N| \leq C|x|.$ 
	\end{enumerate}
\end{lemma}

 Define the following: 
 \begin{align}
 	E_{i,\beta}(u_\beta,r)&= \frac{1}{r^{N-2}}\int_{B_r}\langle A(x)\nabla u_{i,\beta}, \nabla u_{i,\beta} \rangle 
		- 
		f_i(x,u_{i,\beta})u_{i,\beta} 
		- 
		\mathop{\sum_{j=1}^l}_{j\neq i} a(x)\beta |u_{j,\beta}|^{\gamma +1} |u_{i,\beta}|^{\gamma+1}
		\Big)dx \nonumber \\
		&= \frac{1}{r^{N-2}} \int_{\partial B_r} u_{i, \beta}\langle  A(x) \nabla u_{i,\beta}, \nu\rangle d\sigma(x)\ \text{ for each $i=1,\ldots, l$,  and}		\label{ETerm}\\
 	E_\beta(u_\beta,r)&:=\sum_{i=1}^l E_{i,\beta}(u_\beta,r). \nonumber
 \end{align}
where, to obtain \eqref{ETerm}, we have tested the $i$-th equation in \eqref{AlmgrenEq} by $u_{i,\beta}$ and integrated by parts. We also define
\begin{align*}
	H_{i,\beta}(u_\beta,r) := \frac{1}{r^{N-1}}\int_{\partial B_r}\mu(x)|u_{i,\beta}|^2d\sigma(x), \ i=1,\ldots, l,\quad \text{ and } \quad	H_\beta(u_\beta,r) := \sum_{i=1}^l H_{i,\beta}(u_\beta,r),
\end{align*}
and, whenever $H_\beta(u_\beta,r)\neq 0$, consider the \emph{Almgren's quotient}:
\begin{align}
\label{Neq}
	N_\beta(u_\beta,r):= \frac{
		E_\beta(u_\beta,r)
	}{
		H_\beta(u_\beta,r)
	}.
\end{align}

\begin{lemma}
	\label{derivOfH}
	There exists $\overline{r},C>0$ depending only on the dimension $N$, and the constants $M$ and $\theta$ from conditions \textbf{(A1)} and \textbf{(A2)}, such that:
\[
	\Big|
	H_{i,\beta}'(u_\beta,r)-\frac{2}{r}E_{i,\beta}(u_\beta,r)
	\Big| 
	\leq 
	CH_{i,\beta}(u_\beta,r),\quad \text{ and }\quad  r\mapsto H_{i,\beta}(u_\beta,r)e^{Cr} \text{ is monotone nondecreasing}
\]
	for all $r \in ]0,\overline{r}[$, $\beta <0$. In particular, by summing up in $i$,
	$$\Big|
	H_{\beta}'(u_\beta,r)-\frac{2}{r}E_{\beta}(u_\beta,r)
	\Big| 
	\leq 
	CH_{\beta}(u_\beta,r),\quad \text{ and }\quad  r\mapsto H_{\beta}(u_\beta,r)e^{Cr} \text{ is monotone nondecreasing}
	$$\end{lemma}
\begin{proof}
Recalling \eqref{ETerm}, the proof of the bounds follows exactly as the one of  \cite[Lemma C.5]{HugoHolderVariable}. 	Regarding the monotonicity,  we compute the derivative of $H_{i,\beta}(u_\beta, r)$:
	\begin{align*}
		H_{i,\beta}'(u_\beta,r) 		&\geq		\frac{2}{r}E_\beta(u_\beta,r) -CH_{i,\beta}(u_\beta,r) 		\\
		&\geq \frac{2}{r^{N-1}}\int_{B_r}\left(	 \langle A(x) \nabla u_{i,\beta} , \nabla u_{i,\beta}\rangle - f_i(x,u_{i,\beta})u_{i,\beta} \right) dx	- CH_{i,\beta}(u_\beta,r) \\
		&\geq \int_{B_r}\Bigg(\frac{2\theta}{r^{N-1}}|\nabla u_{i,\beta}|^2 	- \frac{2d r}{r^{N}}u_{i,\beta}^2\Bigg)dx- CH_{i,\beta}(u_\beta,r)\\
		&= 2\theta \int_{B_r}\Bigg(\frac{1}{r^{N-1}} |\nabla u_{i,\beta}|^2 - \frac{d r}{\theta r^{N}}u_{i,\beta}^2\Bigg)dx - CH_\beta(u_\beta,r).
	\end{align*}
By using Poincar\'e's inequality (Lemma \ref{PoincareIneq}) we know that, if $dr/\theta<N-1$, then:
	$$\int_{B_r} 
	\left(
    	\frac{1}{r^{N-1}}|\nabla u_{i,\beta}|^2
    	- 
    	\frac
    	{d r}
    	{\theta r^{N}}
    	u_{i,\beta}^2
	\right)
	dx
	\geq 
	- 
	\frac{1}{r^{N-1}}\int_{\partial B_r} u_{i,\beta}^2d\sigma(x) 
	\geq 
	-
	\frac{1}{\theta r^{N-1}}
	\int_{\partial B_r}
	\mu(x)u_{i,\beta}^2
	d\sigma(x)$$
	since $\mu(x)\geq \theta$. With this, we conclude:	
	$$H_{i,\beta}'(u_\beta,r)\geq 2\theta\Big(-\frac{1}{\theta} H_{i,\beta}(u_\beta,r)\Big) - CH_{i,\beta}(u_\beta,r)\geq -\tilde{C}H_{i,\beta}(u_\beta,r),$$
	from which the monotonicity of $\mapsto H_{i,\beta}(u_\beta,r)e^{\tilde Cr}$ easily follows.
\end{proof}

We now define the following vector field:

$$Z(x) = \frac{A(x)x}{\mu(x)}$$
Observe that $Z(x)\sim x$ for $x\sim 0$, since $A(0) = Id$ and $\mu(x) \rightarrow 1$. Using \textbf{(A1)}, \textbf{(A2)} and Lemma \ref{listLemma}-\emph{7}, we have the existence of $\overline{r}$ and $C>0$ such that: 
\begin{equation}\label{eq:bounds_Z}
    |Z(x)|\leq    C|x|, \qquad |\div(Z(x)-x)|\leq C|x|, \qquad \forall x:  |x|\leq \overline{r}.
\end{equation}

The following lemma is written in the Einstein notation, so summations are suppressed. This means that any repeated indices over any pair of variables having that index are supposed to be summed over.

\begin{lemma}
	\label{PozohaevLemma}
	For $r>0$ such that $B_r \subset \Omega$, if $A(x) = A = (a_{ij})$ we have that:
	\begin{align*}&
		r\int_{\partial B_r} \langle A(x)\nabla u_{i,\beta}, \nabla u_{i,\beta} \rangle 
		= 
		\int_{B_r} 
		\div(Z)
		\langle A\nabla u_{i,\beta}, \nabla u_{i,\beta} \rangle
		+ 
		2\int_{B_r} 
		f_{i}(x,u_{i,\beta})
		\langle\nabla u_{i,\beta}, Z \rangle\\
		+&
		2\int_{\partial B_r}
		\langle Z, \nabla u_{i,\beta} \rangle \langle A \nabla u_{i,\beta} , \nu\rangle 
		+ 
		\int_{B_r}
		\langle Z, \nabla a_{hl}\rangle \frac{\partial u_{i,\beta}}{\partial x_h}\frac{\partial u_{i,\beta}}{\partial x_l}
		- 
		2\int_{B_r} 
		a_{hl}\frac{\partial Z_j}{\partial x_h} \frac{\partial u_{i,\beta}}{\partial x_j}\frac{\partial u_{i,\beta}}{\partial x_l}\\
		+&
		2\sum_{j<i} \int_{B_r}\left(
		-\frac{\div(Z)}{\gamma+1} 
	    -\frac{\langle Z, \nabla a(x) \rangle}{a(x)(\gamma+1)}
		\right)
		a(x)\beta|u_{j,\beta}|^{\gamma+1}|u_{i,\beta}|^{\gamma+1}
		+ 
		2
		\sum_{j<i}
		\int_{\partial B_r}\frac{\langle Z, \nu \rangle}{\gamma +1} a(x)\beta |u_{j,\beta}|^{\gamma+1}|u_{i,\beta}|^{\gamma+1}.
	\end{align*}
\end{lemma}
\begin{proof}
This proof follows exactly as the one of \cite[Lemma C.6]{HugoHolderVariable} (without considering the singular limit right at the end of the proof).
\end{proof}

Now define the quantity:
$$\tilde{E}_\beta(u_\beta,r) := \sum_{i=1}^l\frac{1}{r^{N-2}}\int_{B_r}\langle A(x)\nabla u_{i,\beta}, \nabla u_{i,\beta} \rangle.$$
whose derivative we compute in the following lemma.
\begin{lemma}
	\label{estimateOnVariable}
	We have:
	\begin{align}
		\tilde{E}_\beta'(u_\beta,r) =&\sum_{i=1}^l\Bigg( \frac{2}{r^{N-2}}		\int_{\partial B_r} 		\frac{\langle A(x) \nabla u_{i,\beta} , \nu_x \rangle^2}{\mu(x)}  d\sigma(x)		+ 		\frac{2}{r^{N-1}}\int_{B_r} f_{i}(x,u_{i,\beta})\langle\nabla u_{i,\beta}, Z \rangle \nonumber \\
		&+ 		\frac{1}{r^{N-1}}\int_{B_r}\langle Z, \nabla a_{hl} \rangle 		\frac{\partial u_{i,\beta}}{\partial x_h}\frac{\partial u_{i,\beta}}{\partial x_l}		- 		\frac{2}{r^{N-1}}\int_{B_r} a_{hl}\frac{\partial (Z_j-x_j)}{\partial x_h} \frac{\partial u_{i,\beta}}{\partial x_j}\frac{\partial u_{i,\beta}}{\partial x_l} \nonumber \\		
		&+		\frac{1}{r^{N-1}}\int_{B_r}\div(Z-x)\langle A \nabla u_{i,\beta} , \nabla u_{i,\beta}\rangle 		\Bigg) \nonumber \\
		&+		\mathop{\sum_{i,j=1}^l}_{j<i}		\frac{2}{r^{N-1}(\gamma+1)}\int_{\partial B_r}\langle Z(x), \nu_x \rangle{a(x)\beta}|u_{j,\beta}|^{\gamma+1}|u_{i,\beta}|^{\gamma+1} d\sigma(x) \nonumber \\
		&+\frac{2}{r^{N-1}} 		\mathop{\sum_{i,j=1} ^l}_{j<i}		\int_{B_r}\left(-\frac{\div(Z)}{\gamma+1} 		-		\frac{\langle Z(x), \nabla a(x) \rangle}{a(x)(\gamma+1)}		\right)
		a(x)\beta|u_{j,\beta}|^{\gamma+1}|u_{i,\beta}|^{\gamma+1}. 	\label{KineticEquation}
	\end{align}
	Moreover, there also exist $C$ and $\overline{r}$ such that:
	\begin{align}
			\Big|\sum_{i=1}^l\Big(
			\frac{1}{r^{N-1}}
			\int_{B_r}
			\langle Z, \nabla a_{hl}\frac{\partial u_{i,\beta}}{\partial x_h}\frac{\partial u_{i,\beta}}{\partial x_l} \rangle 
			- 
			\frac{2}{r^{N-1}}&\int_{B_r} 
			a_{hl}\frac{\partial (Z_j-x_j)}{\partial x_h} \frac{\partial u_{i,\beta}}{\partial x_j}
			\frac{\partial u_{i,\beta}}{\partial x_l} \nonumber \\
			+
			\frac{1}{r^{N-1}}
			&\int_{B_r}\div(Z-x)
			\langle A \nabla u_{i,\beta} , \nabla u_{i,\beta}\rangle
			\Big)
			\Big| 
			\leq 
			C\tilde{E}(u_\beta,r) 		\label{weirdBound}
	\end{align}
	for all $r \in ]0,\overline{r}[$ and $\beta < 0$.
\end{lemma}
\begin{proof} We have
\begin{align}
\tilde{E}_\beta'(u_\beta, r)
&=
\frac{2-N}{r}\tilde {E}_\beta(u_\beta, r) + \frac{1}{r^{N-2}} \left(\int_{B_r} \langle A\nabla u_{i,\beta}, \nabla u_{i,\beta}\rangle\right)' \nonumber \\
&=
-\frac{(N-2)}{r^{N-1}}\int_{B_r} \langle A\nabla u_{i,\beta},\nabla u_{i,\beta}\rangle
+
\frac{1}{r^{N-2}}\left(\int_{B_r} \langle A\nabla u_{i,\beta}, \nabla u_{i,\beta}\rangle\right)'.\label{eq:original_derivative}
\end{align}
In order to compute the derivative of the last term, we define:
\begin{align*}
    Q_\beta(u_\beta, r)
:=&
    \mathop{\sum_{i,j=1}^l}_{j<i}
		\frac{2}{r^{N-1}(\gamma+1)}
		\int_{\partial B_r}
		\langle 
		    Z, \nu 
		\rangle
		{a(x)\beta}|u_{j,\beta}|^{\gamma+1}|u_{i,\beta}|^{\gamma+1} d\sigma(x)+\\
		&+\frac{2}{r^{N-1}}
		\mathop{\sum_{i,j=1}^l}_{j<i}
		\int_{B_r}\left(-\frac{\div(Z)}{\gamma+1} 
		-
		\frac{\langle Z, \nabla a(x) \rangle}{a(x)(\gamma+1)}
		\right)
		a(x)\beta|u_{j,\beta}|^{\gamma+1}|u_{i,\beta}|^{\gamma+1}.
\end{align*}
From Lemma \ref{PozohaevLemma}, $\div(Z) =N+\div(Z-x)$ and using: 
\[
 \frac{2}{r}\int_{\partial B_r} \langle Z,\nabla u_{i,\beta}\rangle \langle A\nabla u_{i,\beta},\nu\rangle=2\int_{\partial B_r}\frac{\langle A\nabla u_{i,\beta},\nu\rangle^2}{\mu},\quad  \frac{2}{r}\int_{B_r}
    a_{hl}
    \frac{\partial x_j}{\partial x_h}
    \frac{\partial u_{i,\beta}}{\partial x_j}
    \frac{\partial u_{i,\beta}}{\partial x_l}= \frac{2}{r}    \int_{B_r}    \langle A \nabla u_{i,\beta}, \nabla u_{i,\beta} \rangle
\]
and , we have
\begin{align*}
\frac{d}{dr} & \int_{B_r}  \langle A\nabla u_{i,\beta}, \nabla u_{i,\beta}\rangle=    \int_{\partial B_r} \langle A(x)\nabla u_{i,\beta}, \nabla u_{i,\beta}\rangle \\
    =& \frac{1}{r}\int_{B_r}	\div(Z)\langle A\nabla u_{i,\beta}, \nabla u_{i,\beta} \rangle+ \frac{2}{r}\int_{B_r} f_{i}(x,u_{i,\beta})\langle\nabla u_{i,\beta}, Z \rangle+ \frac{2}{r}\int_{\partial B_r}\langle Z, \nabla u_{i,\beta} \rangle \langle A \nabla u_{i,\beta} , \nu\rangle\\
		& 		+		\frac{1}{r}		\int_{B_r}		\langle Z, \nabla a_{hl}\rangle \frac{\partial u_{i,\beta}}{\partial x_h}\frac{\partial u_{i,\beta}}{\partial x_l}- \frac{2}{r}\int_{B_r} a_{hl}\frac{\partial Z_j}{\partial x_h} \frac{\partial u_{i,\beta}}{\partial x_j}\frac{\partial u_{i,\beta}}{\partial x_l} + 		Q_\beta(u_\beta,r)r^{N-2}+		2\int_{\partial B_r}\frac{\langle A\nabla u_{i,\beta},\nu\rangle^2}{\mu}	\\    
		=&\frac{1}{r}\int_{B_r} 	\div(x + Z -x)\langle A(x)\nabla u_{i,\beta}, \nabla u_{i,\beta} \rangle		+ 		\frac{2}{r}\int_{B_r} 		f_{i}(x,u_{i,\beta})		\langle\nabla u_{i,\beta}, Z \rangle\\	
			&	+ 		\frac{1}{r}		\int_{B_r}		\langle Z, \nabla a_{hl}\rangle \frac{\partial u_{i,\beta}}{\partial x_h}\frac{\partial u_{i,\beta}}{\partial x_l}	- 		\frac{2}{r}\int_{B_r} 		a_{hl}\frac{\partial (Z_j-x_j+x_j)}{\partial x_h} \frac{\partial u_{i,\beta}}{\partial x_j}\frac{\partial u_{i,\beta}}{\partial x_l}		+		Q_\beta(u_\beta,r)r^{N-2}	\\
         =&
            \frac{N-2}{r}\int_{B_r} \langle A\nabla u_{i,\beta},\nabla u_{i,\beta}\rangle +\frac{2}{r}\int_{B_r} f_i(x,u_{i,\beta})\langle \nabla u_{i,\beta},Z\rangle 	  \\ 
		&+2\int_{\partial B_r}\frac{\langle A\nabla u_{i,\beta},\nu\rangle^2}{\mu}+\frac{1}{r}\int_{B_r} \div (Z-x) \langle A\nabla u_{i,\beta},\nabla u_{i,\beta}\rangle +\frac{1}{r}\int_{B_r} \langle Z,\nabla a_{hl}\rangle\frac{\partial u_{i,\beta}}{\partial x_h} \frac{\partial u_{i,\beta}}{\partial x_l}\\
		&-\frac{2}{r}\int_{B_r}a_{hl}\frac{\partial (Z_j-x_j)}{\partial x_h}\frac{\partial u_{i,\beta}}{\partial x_j} \frac{\partial u_{i,\beta}}{\partial x_l} + Q_\beta(u_\beta, r)r^{N-2}.
\end{align*}
Going back to \eqref{eq:original_derivative}, we conclude that identity \eqref{KineticEquation} is true.	

 It remains to prove \eqref{weirdBound}. We bound each term individually. By  \textbf{(A2)} and \eqref{eq:bounds_Z},
$$
\left|
    \frac{1}{r^{N-1}}
    \int_{B_r}
    \langle
    Z,\nabla a_{h,l}
    \rangle
    \frac{\partial u_{i,\beta}}{\partial x_h}
    \frac{\partial u_{i,\beta}}{\partial x_l}
\right|
\leq
    \frac{C}{r^{N-2}}
    \int_{B_r}
        |\nabla u_{i,\beta}|^2
\leq
    \frac{C}{\theta r^{N-2}}
    \tilde{E}_\beta(u_\beta, r).
$$
Moreover, $\div(Z(x)-N)| \leq C|x|$ and
\begin{align*}
    |\frac{\partial}{\partial x_k}
    (Z_j(x)- x_j)|
&=
    |\frac{\partial}{\partial x_k}
    \left(
        \sum_{h=1}^N
        \frac{a_{jh}(x)x_h}{\mu(x)}
        -
        x_j
    \right)|\\
&=| \sum_{h=1}^N\left( \frac{\partial a_{jh}(x)}{\partial x_k} \frac{x_h}{\mu(x)}    + \frac{a_{jh}(x)\delta_{hk}}{\mu(x)} + \frac{a_{jh}(x)x_h}{\mu^2(x)}  \frac{\partial \mu(x)}{\partial x_l} \right)  -  \delta_{jk}|\\
&\leq C|x|+   | \frac{a_{jk}(x)}{\mu(x)}  -\delta_{jk} | \leq C|x| + |\frac{a_{jk}(x) - \delta_{jk}}{\mu(x)} |   +  |  \frac{\delta_{jk}}{\mu(x)}     -   \delta_{jk}  |\leq C'|x|,
\end{align*} 
where we used Lemma \ref{listLemma}-\emph{1.,3.,5.} and  $\mu(x)\geq \theta >0$. With this, we conclude the desired bound:
\begin{multline*}
\Big| 
    \frac{1}{r^{N-1}} 
    \int_{B_r} 
    \langle Z,\nabla a_{hl}\rangle 
    \frac{\partial u_{i,\beta}}{\partial x_h}
    \frac{\partial u_{i,\beta}}{\partial x_l}
+
    \frac{1}{r^{N-1}}
    \int_{B_r} 
    \div (Z-x) 
    \langle A\nabla u_{i,\beta},\nabla u_{i,\beta}\rangle\\
-
    \frac{2}{r^{N-1}}  
    \int_{B_r} a_{hl} \frac{\partial u_{i,\beta}}{\partial x_l}
    \frac{\partial (Z_j-x_j)}{\partial x_h} 
    \frac{\partial u_{i,\beta}}{\partial x_j}
\Big|
\leq 
    C\frac{1}{r^{N-2}}
    \int_{B_r}
    \langle A\nabla u_{i,\beta},\nabla u_{i,\beta}\rangle
=
    C\tilde{E}_\beta(u_{\beta},r),
\end{multline*}which completes the proof.
\end{proof}

\begin{thm}
	\label{almgrenMonotonicity}
Let $u_\beta$ be a positive solution of \eqref{AlmgrenEq}, under $\beta <0$ \textbf{(A1)}, \textbf{(A2)}, \textbf{(F)} and \eqref{boundedness}. Assume moreover that 
\begin{equation*}
	A(0)=Id.
\end{equation*}
If $N\in \mathbb{N}$, $\gamma\geq 1$ and $\frac{\gamma N}{\gamma + 1} < 2$, then there exist constants $\overline{r}$ and $C>0$, such that:
	\begin{equation*}	    \label{NisIncreasingStuff1}
	    N_\beta(u_{\beta},r)+1\geq 0,\qquad 	    H_\beta(u_\beta,r) > 0,\qquad	    N_\beta'(u_{\beta},r) \geq -C(N(u_{\beta},r)+1)
	\end{equation*}
In particular,
		\begin{equation}
	\label{Nmonotonicity}
	(N_\beta(u_\beta,r)+1)e^{Cr} 
	\end{equation}
		is monotone nondecreasing for every $r \in
	]0,\overline{r}[$, $\beta<0$. 
	
The constants $C$ and $\bar r$	depend only on the ellipticity constant $\theta>0$ of $A$, the upper bound $M>0$ of $\|DA\|_\infty$, the dimension $N$, and the uniform bound $m>0$ from \eqref{boundedness} and the constant $d$ from \eqref{boundForF}.
\end{thm}
\begin{proof} \textbf{Step 1.} We show  $E_\beta(u_\beta,r) + H_\beta(u_\beta,r) \geq 0$, which is equivalent to $N_\beta(u_{\beta},r)+1\geq 0$.

Indeed, since $a(x)>0$, $\beta < 0$, \eqref{boundForF} and (which implies, in particular, that $\mu(y)\geq \theta$) we have
	\begin{align*}
		E_\beta(&u_\beta,r) + H_\beta(u_\beta,r) =
		\sum_{i=1}^l\Bigg[\frac{1}{r^{N-2}}\int_{B_r} \left(\langle A(x) \nabla u_{i,\beta} , \nabla u_{i,\beta} \rangle - f_i(x,u_{i,\beta})u_{i,\beta} - a(x)\beta \sum_{j\neq i} u_{j,\beta}^{\gamma+1}u_{i,\beta}^{\gamma+1}\right)\\
		&\qquad\qquad+ \frac{1}{r^{N-1}}		\int_{\partial B_r} \mu(x) u_{i,\beta}^2d\sigma(x)\Bigg] \\
		\geq&\sum_{i=1}^l \Bigg[\int_{B_r}\left(\frac{1}{r^{N-2}} \langle A(x) \nabla u_{i,\beta} \nabla u_{i,\beta} \rangle -\frac{dr^2}{r^N} u_{i,\beta}^2\right)+ \frac{1}{r^{N-1}}\int_{\partial B_r} \mu(x) u_{i,\beta}^2d\sigma\Bigg] \\
	\geq& \sum_{i=1}^l\Bigg[\int_{B_r}\left(   \frac{\theta}{r^{N-2}}	|\nabla u_{i,\beta}|^2 -\frac{dr^2}{r^N}u_{i,\beta}^2\right)	+ \frac{\theta}{r^{N-1}}\int_{\partial B_r}  u_{i,\beta}^2d\sigma
		\Bigg] \geq \sum_{i=1}^l \left [ \int_{B_r} \frac{\theta(N-1)-dr^2}{r^N} u_{i,\beta}^2 \right],
	\end{align*}
where we have used Poincar\'e's inequality (Lemma \ref{PoincareIneq}). The claim now follows by choosing $r$ small enough so that $dr^2 < \theta (N-1)$.
	
\noindent	\textbf{Step 2.} We show equation $N_\beta'(u_{\beta},r) \geq -C(N(u_{\beta},r)+1)$ whenever $r$ is small $H_\beta(u_\beta,r) \neq 0$ (which, in particular, shows \eqref{Nmonotonicity} for such $r$'s). First of all, by Lemma \ref{derivOfH},
\[
	H_\beta'(u_\beta,r) = \frac{2E_\beta(u_\beta,r)}{r} + O(1)H_\beta(u_\beta,r),
	\]
As for $E_\beta(u_\beta,r)$, recalling that
	\begin{align*}
	    E_\beta(u_\beta,r)  &=  \tilde{E}_\beta(u_\beta,r)   - \frac{1}{r^{N-2}}\int_{B_r}f_i(x,u_{i,\beta})u_{i,\beta} - \frac{1}{r^{N-2}}\int_{B_r}2\mathop{\sum_{j=1}^l}_{j<i} a(x)\beta |u_{j,\beta}|^{\gamma +1} |u_{i,\beta}|^{\gamma+1}dx
	\end{align*}
we have, by  Lemma \ref{estimateOnVariable},
	\begin{align}
		E_\beta'(u_\beta,r) =& \tilde{E}_\beta'(u_\beta,r) + \sum_{i=1}^l\Bigg[\frac{(N-2)}{r^{N-1}}\int_{B_r} f_i(x,u_{i,\beta})u_{i,\beta} dx - \frac{1}{r^{N-2}}\int_{\partial B_r} f_i(x,u_{i,\beta})u_{i,\beta}d\sigma(x)\nonumber\\
		&-\sum_{j<i}\Bigg(\frac{(4-2N)}{r^{N-1}}\int_{B_r}a(x)\beta|u_{j,\beta}|^{\gamma+1}|u_{i,\beta}|^{\gamma +1 } dx- \frac{2}{r^{N-2}}\int_{\partial B_r} a(x)\beta|u_{j,\beta}|^{\gamma+1}|u_{i,\beta}|^{\gamma +1 }d\sigma(x)\Bigg)\Bigg] 	\nonumber\\
		=&
		O(1)\tilde E_\beta(u_\beta,r) + R_\beta(u_\beta, r)+\sum_{i=1}^l \Bigg[\frac{2}{r^{N-1}} \int_{\partial B_r}\frac{\langle A(x)\nabla u_{i,\beta}, \nu_x\rangle^2}{\mu(x)}d\sigma(x) \nonumber \\
		&-\sum_{j<i}\Bigg(\frac{1}{r^{N-1}}\int_{B_r}\left(4-2N +\frac{2\div(Z)}{\gamma+1} - 2\frac{\langle Z(x), \nabla a(x)\rangle}{a(x)(\gamma+1)}\right)a(x)\beta|u_{j,\beta}|^{\gamma+1}|u_{i,\beta}|^{\gamma +1 }dx \nonumber\\
		&- \frac{1}{r^{N-2}}\int_{\partial B_r}\left(2+\frac{2\langle Z(x), \nu_x \rangle}{r(\gamma+1)}\right) a(x)\beta|u_{j,\beta}|^{r(\gamma+1)}|u_{i,\beta}|^{\gamma +1 }d\sigma(x)\Bigg)\Bigg],
		\label{derivativeOfEeq123}
	\end{align}
where	
	\begin{align*}
R_\beta(u_\beta,r) :=& \sum_{i=1}^l\Bigg[\frac{2}{r^{N-1}} \int_{B_r}			\left(f_i(x,u_{i,\beta})\langle Z(x),\nabla u_{i,\beta} \rangle+ \frac{(N-2)}{r^{N-1}}f_i(x,u_{i,\beta})u_{i,\beta}\right)\\
		&- \frac{1}{r^{N-2}}\int_{\partial B_r}f_i(x,u_{i,\beta})u_{i,\beta}d\sigma\Bigg].
	\end{align*}
and $O(1)$ is a bounded function that comes from \eqref{weirdBound}.

%%%%%%%%%%%

Now, we check that the terms in \eqref{derivativeOfEeq123} with $\beta$ are all nonnegative. Indeed, using the fact that $Z(0) = 0$, $\div(Z) = N + O(r)$ and condition \textbf{(a)}, we obtain $\frac{\langle Z(x), \nabla a(x) \rangle}{a(x)(\gamma+1)} = O(r)$, and so, for $r$ small, we have
	\begin{equation}\label{eq:RESTRICTION}
	4-2N 
	+
	\frac{2\div(Z)}{\gamma+1} 
	+ 
	\frac{2\langle Z(x), \nabla a(x)\rangle}{a(x)(\gamma+1)} 
	= 
	2(2-\frac{\gamma N}{\gamma+1}) 
	+ 
	O(r) 
	> 
	0
	\end{equation}
since, by assumptin, $\frac{\gamma N}{\gamma + 1} < 2$.  Also, on $\partial B_r$:
	$$
	    2-\frac{\langle Z(x), \nu_x\rangle}{r(\gamma+1)}
	=
	    2 
	    - 
	    \frac
	    {\langle A(x)\frac{x}{|x|}, \frac{x}{|x|}\rangle} 
	    {\mu(x)(\gamma + 1)}
    =
        2 - \frac{1}{\gamma+1}> 0.$$
Therefore, since $\beta < 0$,  we conclude from \eqref{derivativeOfEeq123} that:
	\begin{equation}
	\label{KineticBoundFromBellow}
	E_\beta'(u_\beta,r)
	\geq 
	\frac{2}{r^{N-1}} 
	\sum_{i=1}^l 
	\int_{\partial B_r}
	\frac{\langle A(x)\nabla u_{i,\beta}, \nu_x \rangle^2}{\mu(x)}d\sigma(x)+ O(1)\tilde{E}_\beta(u_\beta,r) + R_\beta(u_\beta, r).
	\end{equation}
	Next, we estimate $O(1)\tilde{E}_\beta(u_\beta,r)+R_\beta(u_\beta, r)$. Using  \eqref{boundForF} (which implies that $|f(x,u_{i,\beta})|\leq d|u_{i,\beta}|$), the bound $|Z(x)|\leq C|x|$ for $|x|\leq \overline{r}$, the ellipticity condition \textbf{(A1)} and  choosing $r<\overline{r}<1$ (which implies $r^{N}<r^{N-1}$), then:
	\begin{align}
		|O(1)\tilde{E}_\beta(u_\beta,r)&+R_\beta(u_\beta, r)|
		\leq O(1)\sum_{i=1}^l\left[\int_{B_r}\left(\frac{1}{r^{N-2}}\langle A(x)\nabla u_{i,\beta} , \nabla u_{i,\beta} \rangle + \frac{1}{r^{N}}  u_{i,\beta}^2 \right)dx+ \frac{1}{r^{N-1}} \int_{\partial B_r} u_{i,\beta}^2 \right]\nonumber \\
	&		\leq O(1) \left[ E_\beta(u_\beta,r)	+ \sum_{i=1}^l  \int_{B_r}f_i(x,u_{i,\beta})u_{i,\beta} dx +  \frac{1}{r^{N-1}}\sum_{i=1}^l  \int_{\partial B_r} u_{i,\beta}^2  \right] \nonumber\\
		&		\leq O(1) \left[ E_\beta(u_\beta,r)	+ H_\beta(u_\beta,r)+   \frac{d r^2}{r^{N}}\sum_{i=1}^l\int_{B_r}u_{i,\beta}^2dx.\right] \label{remainderFinalBound}.
	\end{align}
	On the other hand, using Poincar\'e's inequality (Lemma \ref{PoincareIneq}) and reasoning as above,
	\begin{align*}
		\frac{1}{r^{N}}\sum_{i=1}^l 
		\int_{B_r}
		u_{i,\beta}^2(x)
		dx
		&\leq 
		\sum_{i=1}^l\frac{1}{N-1}\left(
		\frac{1}{r^{N-2}}
		\int_{B_r}
		|\nabla u_{i,\beta}(x)|^2
		dx 
		+ 
		\frac{1}{r^{N-1}}
		\int_{\partial B_r}
		u_{i,\beta}^2(x)d\sigma(x)
		\right)\\
		&\leq
		    O(1)\Big(E_\beta(u_\beta,r) 
		+ 
		    H_\beta(u_\beta,r)
		+
    		\frac{r^2d}{r^N}
    		\int_{\partial B_r}
    		u_{i,\beta}^2
    		dx
		\Big).
	\end{align*}
	Thus for $r$ small enough we conclude that:
	\begin{equation}
	    \label{BoundOfMass}
	    \frac{1}{r^{N}}\sum_{i=1}^l 
		\int_{B_r}
		u_{i,\beta}^2(x)
		dx
	\leq
	    O(1)
	    \Big(E_\beta(u_\beta,r) 
		+ 
		    H_\beta(u_\beta,r)
		\Big).
	\end{equation}
In conclusion, by combining \eqref{KineticBoundFromBellow}, \eqref{remainderFinalBound} and \eqref{BoundOfMass} we have:
	$$
	E_\beta'(u_\beta,r) \geq  \frac{2}{r^{N-1}} \sum_{i=1}^l \int_{\partial B_r}\frac{\langle A(x)\nabla u_{i,\beta}, \nu_x \rangle^2}{\mu(x)}d\sigma(x)+ O(1)\Big(E_\beta(u_\beta,r) + H_\beta(u_\beta,r)\Big).
	$$
	Using Holder's inequality:
	\begin{align*}
		N_\beta'(u_\beta,r)  &=\frac{E_\beta'(u_\beta,r)H_\beta(u_\beta,r) - E_\beta(u_\beta,r)H_\beta'(u_\beta,r)}{H_\beta^2(u_\beta,r)}\\
		&\geq
		\frac{2}{H_\beta^2(u_\beta,r)r^{2N-3}}
		\Bigg[
		\Big(
		\sum_{i=1}^l
		\int_{\partial B_r}
		\frac{\langle A(x) \nabla u_{i,\beta} , \nu \rangle^2}{\mu(x)}d\sigma
		\Big) 
		\Big(
		\sum_{i=1}^l
		\int_{\partial B_r }
		\mu(x)
		u_{i,\beta}^2
		d\sigma
		\Big)\\ 
		&- 
		\Big(
		\sum_{i=1}^l 
		\int_{\partial B_r} 
		u_{i,\beta} 
		\langle 
		A(x) \nabla u_{i,\beta} , \nu \rangle
		d\sigma
		\Big)^2
		\Bigg]\\
		&
		+
		\frac{1}{H_\beta^2(u_\beta,r)}\Bigg[
		O(1)
		H_\beta(u_\beta,r)
		\Big(
		E_\beta(u_\beta,r) + H_\beta(u_\beta,r)
		\Big) 
		+ 
		O(1)E_\beta(u_\beta,r)H_\beta(u_\beta,r)
		\Bigg] \\
		&\geq
		\frac{1}{H_\beta^2(u_\beta,r)}
		\Bigg[
		O(1)H_\beta(u_\beta,r)
		\Big(
		E_\beta(u_\beta,r) + H_\beta(u_\beta,r)
		\Big) 
		+	 
		O(1)E_\beta(u_\beta,r)H_\beta(u_\beta,r)
		\Bigg] \\
		&\geq
		-C\Big(\frac{E_\beta(u_\beta,r)}{H_\beta(u_\beta,r)} + 1\Big) 
		=
		-C\Big(N_\beta(u_\beta,r)+1\Big),
	\end{align*}
    where in the last inequalities we used the fact (proved in Step 1.) that $E_\beta(u_\beta,r) + H_\beta(u_\beta,r)$ is positive. Thus we conclude $N_\beta'(u_\beta,r)\geq -C\left(N_\beta(u_\beta,r)+1\right)$ whenever $r$ is small and $H_\beta(u_\beta,r) \neq 0$. 
	
\noindent	\textbf{Step 3.} There exists $\overline{r}>0$ small enough such that $H_\beta(u_\beta, r) \neq 0$ for $r \in ]0,\overline{r}[$. Indeed, we have $H_\beta'(u_\beta,r)=a_\beta(r)H_\beta(u_\beta,r)$, where (by  Lemma \ref{derivOfH}) $a_\beta(r)=\frac{2}{r}N_\beta(u_\beta,r)+O(1)$. Then, by the existence and uniqueness theorem for this ODE, and since $u_\beta>0$, we have $H_\beta(u_\beta,r)>0$ for sufficiently small $r>0$.
\end{proof}

\begin{remark}
The restriction $\frac{\gamma N}{\gamma+1}<2$ in Theorem \ref{DesiredTheorem} comes only from the proof of the previous theorem, namely from the necessity of having the inequality
\[
4-2N 
	+
	\frac{2\div(Z)}{\gamma+1} 
	+ 
	\frac{2\langle Z(x), \nabla a(x)\rangle}{a(x)(\gamma+1)} 
	= 
	2(2-\frac{\gamma N}{\gamma+1}) 
	+ 
	O(r) 
	> 
	0
	\]
	for small $r>0$.	
\end{remark}

We conclude this section with the following result.

\begin{lemma}
	\label{doublingLemma}
Under the assumptions of Theorem \ref{almgrenMonotonicity}, there exists $C>0$
	\begin{enumerate}
\item If there exists $\tilde{r}$ and $R$ such that $N_\beta(u_\beta,r)\leq \lambda$ for all $0 \leq \tilde{r} \leq r \leq R \leq \overline{r}$, then:
	$$
	r \mapsto \frac{H_\beta(u_\beta,r)}{r^{2\lambda}}e^{-Cr}
	$$
	is monotone nonincreasing for $\tilde{r} \leq r\leq R$.
	\item If there exists $\tilde{r}$ and $R$ such that $N_\beta(u_\beta,r)\geq \gamma$  for all $0 \leq \tilde{r} \leq r \leq R \leq \overline{r}$, then:
	$$
	r \mapsto \frac{H_\beta(u_\beta,r)}{r^{2\gamma}}e^{Cr}. 
	$$
	is monotone nondecreasing
	\end{enumerate}
\end{lemma}
\begin{proof}
(1) We know by Lemma \ref{derivOfH} that there exists $C$ such that:
	\begin{align*}
		\frac{d}{dr}\log\left(
		H_\beta(u_\beta,r)
		\right)
		\leq
		\frac{2N_\beta(u_\beta,r)}{r} + C \leq \frac{2\lambda}{r} + C
	\end{align*}
	from this the result follows. The proof of  (2) is similar.
\end{proof}

%%%%%%%%%%%%%%%%%%%%%%

%%%%%%%%%%%%%%%%%%%%%%

\section{An Alt-Caffarelli-Friedman type monotonicity formula ($N\geq 3$)}
\label{chapter:resultsChap4}
The purpose of this section is to prove an Alt-Caffarelli-Friedman type monotonicity formula, see Theorem \ref{AltCaffMonotonicity} below. For simplicity, we focus on systems of type \eqref{equation} with $\gamma = 1$; however, it should be clear how to adapt our proofs to the general case $\gamma\geq 1$, see Remark \ref{rem_finalSection4} for the details.
Let $u_n = (u_{1,n}, u_{2,n},...,u_{l,n})$  be a nonnegative solution of the system:
\begin{equation}
	\label{altCafEq}
	-\div(A_n(y)\nabla u_{i,n}) = f_{i,n}(y,u_{i,n}) + M_n a_n(y)\sum_{j\neq i} |u_{j,n}|^{2}u_{i,n}
	\qquad
	y \in \Omega_n
\end{equation}
where $M_n < 0$. We take $N\geq 3$ through this section, see Remark \ref{rem:dimension} below on why this is not restrictive.  Before stating the monotonicity formula, we need some preparations.

\noindent \textbf{Notation. } Within this section, we let $
    \nabla u_{i,n}(y)$ denote the usual gradient in $\mathbb{R}^N$, while the gradient on a sphere is denoted by
$$
    \nabla_{\theta}u_{i,n}(y) := \pi_{T_y (\partial B_{|y|})} ( \nabla u_{i,n} (y) ),
$$ 
where $\pi_{T_y (\partial B_{|y|})}$ is the projection of $\nabla u_{i,n}(y)$ onto $T_y\partial(B_{|y|}(0))$. Given $y \neq 0$, the vector $\nu_y := \frac{y}{|y|}$ is the exterior normal of the sphere $\partial(B_{|y|}(0))$ at $y$.
We also define: 
$$\mu_n(y) := \langle A_n(y)\nu_y,\nu_y\rangle =\left\langle A_n(y) \frac{y}{|y|}, \frac{y}{|y|} \right\rangle \qquad \forall y \neq 0.$$

Fix the following objects:
\begin{itemize}
	\item $c_n$ a sequence such that $c_n \rightarrow 0$,
	\item $R_n>1$ a sequence of radius such that $B_{R_n}\subset \Omega_n$,
	\item $\epsilon_n > 0$ a sequence of positive numbers,
	\item $\lambda, w > 0$ positive numbers,
\end{itemize}
and assume the following conditions:
\begin{itemize}
	\item[($h_0$)]$\sup_{y \in B_r(0)}\|A_n(y)-Id\| \leq c_nr$ and $\sup_{y \in B_r(0)}\|DA_n(y)\| \leq c_n$ for all $r \in ]1,R_n[$;\\
	\item[($h_1$)] $\epsilon_n R_n^2 \leq (\frac{N-2}{2})^2-\delta$ for some $\delta>0$;\\
	\item[($h_2$)] $|f_{i,n}(y,u_{i,n})| \leq \frac{1}{2}\epsilon_n  \mu_n(y) |u_{i,n}|$;\\
	\item[($h_3$)] 
	$
\displaystyle	\frac{1}{\lambda}
\leq 
    \frac
	{
    	\int_{\partial B_r}
    	\mu_n(y)
    	u_{1,n}^2 d\sigma(y)
	}
	{
	    \int_{\partial B_r} \mu_n(y)
	    u_{2,n}^2d\sigma(y)
	} 
	\leq \lambda
	$ and $ \displaystyle \frac{
		1}{r^{N-1}}\int_{\partial B_r}\mu_n(y)u_{i,n}^2d\sigma(y)\geq w$ for all $r \in ]1,R_n[$ and $i=1,2$;\\
	\item[($h_4$)] $c_nR_n \rightarrow 0$;\\
	\item[($h_5$)] There exists $C>0$ such that, for all $s,r \in [0,R_n]$ satisfying $s\leq r$ and $i\in \{1,...,l\}$, we have \[
	\frac{1}{s^{N-1}}\int_{\partial B_s} u_{i,n}^2(y)d\sigma(y) \leq \frac{C}{r^{N-1}}\int_{\partial B_r}u_{i,n}^2(y)d\sigma(y).\]
\end{itemize}

\begin{lemma}		\label{divClaim}
Under ($h0$)--($h5$), there exists $\alpha > 0$, independent of $n$ and $r$, such that:
		\begin{equation*}
    	    \label{alphaEquationFTW}	\int_{B_r}u_{1,n}^2\div(A_n(y)\nabla|y|^{2-N})dy 
    		\leq 
    		\frac
    		{\alpha c_n(N-2)}{r^{N-2}}
    		\int_{\partial B_r}\mu_n(y)u_{1,n}^2d\sigma(y).
		\end{equation*}
\end{lemma}
\begin{proof}
We know that $\Delta |y|^{2-N} = -C_N\delta$. Thus, we have:
		\begin{align*}
			\int_{B_r} u_{1,n}^2\div(A_n(y)\nabla |y|^{2-N})dy 
			&= 
			-C_Nu_{1,n}^2(0) 
			+ 
			\int_{B_r} u_{1,n}^2\div\big((A_n(y)-Id\big)\nabla |y|^{2-N})dy
			\\
			&\leq 
			\int_{B_r}
			u_{1,n}^2
		\div\big((A_n(y)-Id)\nabla|y|^{2-N}\big)
			dy.
		\end{align*}
To estimate the integral above, we observe that there exists $\tilde{C}>0$ such that:
	    \begin{align*}
		|\div (	&(A_n(y)-Id) \nabla |y|^{2-N})	| 	\leq\left| \sum_{i=1}^N	\frac{\partial}{\partial y_i}\left((A_n(y)-Id) \nabla |y|^{2-N}\right)_i\right| \\
		&=\left| \sum_{i=1}^N\left(   \frac{\partial}{\partial y_i}   (A_n(y)-Id)   \nabla |y|^{2-N}\right)_i+  \left( (A_n(y)-Id) \frac{\partial}{\partial y_i}  \nabla |y|^{2-N}  \right)_i 	\right| \leq \frac{\tilde{C}c_n	}{|y|^{N-1}}.
		\end{align*}
		since $\|A_n(y)-Id\|_{L^\infty(B_r)}\leq c_n|y|$ and $\|DA_n(y)\|_{L^\infty(B_r)} \leq c_n$, by ($h0$). Thus, we obtain
		\begin{align*}
			\left|	\int_{B_r}	u_{1,n}^2\div((A_n(y)-Id)\nabla|y|^{2-N})	dy \right| &\leq 	\tilde{C}c_n	\int_{B_r}\frac{u_{1,n}^2}{|y|^{N-1}}dy = 
			\tilde{C}c_n\int_0^r \int_{\partial B_s} \frac{u_{1,n}^2}{s^{N-1}}d\sigma(y)ds\\
			&\leq \tilde{C} c_n
			\int_{0}^r 
			\int_{\partial B_s}
			\frac{u_{1,n}^2}{r^{N-1}}
			d\sigma(y) 
			ds
			\leq
			\tilde C C c_nr 
			\int_{\partial B_r} \frac{u_{1,n}^2}{r^{N-1}}d\sigma(y).
		\end{align*}
	where we have used hypothesis ($h_5$). Now, since $|\mu_n(y) -1| \leq c_n$, then $1 \leq 2\mu_n(y)$ and so:
		$$
		    \int_{B_r}u_{1,n}^2\div(A_n(y)\nabla|y|^{2-N})dy 
		\leq 
    		2\tilde C C (c_nr)\int_{\partial B_r}\mu_n(y)\frac{u_{1,n}^2}{r^{N-1}}d\sigma(y),
		$$
	    and we can choose $\alpha = \frac{2\tilde{C} C}{N-2}$.
\end{proof}

We consider, for the first two components $(u_{1,n},u_{2,n})$ of the solution of \eqref{altCafEq}, the expressions:
\begin{align*}
J_{1,n}(r)&:= \int_{B_r}\left(\langle A_n(y)\nabla u_{1,n}, \nabla u_{1,n}  \rangle - M_n a_n(y)  |u_{1,n}|^2|u_{2,n}|^2 - u_{1,n}f_{1,n}(y,u_{1,n})\right)|y|^{2-N} dy,\\
J_{2,n}(r)&:= \int_{B_r}\left(\langle A_n(y)\nabla u_{2,n}, \nabla u_{2,n} \rangle - M_n a_n(y)  |u_{1,n}|^{2}|u_{2,n}|^2- u_{2,n}f_{1,n}(y,u_{2,n})\right)|y|^{2-N}dy.
\end{align*}
For $\alpha>0$ as in Lemma \ref{divClaim}, we define
\begin{align}
\Lambda_{1,n}(r)&:= \frac{\displaystyle r^2\int_{\partial B_r} \langle A_n(y) \nabla_{\theta} u_{1,n}, \nabla_{\theta} u_{1,n} \rangle - \frac{\displaystyle \langle A_n \nabla_{\theta}u_{1,n}, \nu_y \rangle^2}{\mu_n(y)} - M_na_n(y)|u_{2,n}|^{2} |u_{1,n}|^2	-u_{1,n}f_{1,n}(y,u_{1,n})}{\displaystyle \int_{\partial B_r}(1+\alpha rc_n)u_{1,n}^2(y)\mu_n(y)d\sigma(y)} \label{LambdaAltCaf1}\\
\Lambda_{2,n}(r)&:= \frac{\displaystyle r^2\int_{\partial B_r} \langle A_n(y) \nabla_{\theta} u_{2,n}, \nabla_{\theta} u_{2,n} \rangle - \frac{\langle A_n \nabla_{\theta}u_{2,n}, \nu_y \rangle^2}{\mu_n(y)} - M_na_n(y)|u_{1,n}|^{2} |u_{2,n}|^2
	-u_{2,n}f_{2,n}(y,u_{2,n})}{\displaystyle \int_{\partial B_r}(1+\alpha rc_n)u_{2,n}^2(y)\mu_n(y)d\sigma(y)}.\nonumber 
\end{align}
In addition to ($h0$)--($h5$), we assume
\begin{itemize}
	\item[($h_6$)] $J_{i,n}(r)>0$ and $\Lambda_{i,n}(r)>0$ for every $r \in ]1, R_n[$;\\
\end{itemize}

The main objective of this section is to prove the following theorem:

\begin{thm}
	\label{AltCaffMonotonicity}
	Let $u_n=(u_{1,n},...,u_{l,n})$ be a nonnegative solution of equation \eqref{altCafEq}, and that $A_n$ satisfy \textbf{(A1)} and \textbf{(A2)} for some $\theta>0$, $M>0$. Assume that ($h_0$)-($h_6$) hold true. Then, for any $0<\eta<\frac{1}{4}$, there exists a positive constant $C= C((R_n),(c_n), (\epsilon_n), \lambda, w, N,\eta)$, such that:
	$$r \mapsto \frac{J_{1,n}(r)J_{2,n}(r)}{r^4}
	e^{
		-C|M_n|^{-\eta}
		r^{-2\eta} 
		+ 
		C\epsilon_nr^2
		+
		C
		c_n
		r
	} 
	$$
	is monotone nondecreasing for $r \in ]1,R_n[$.
\end{thm}

\begin{remark}\label{eq:differencesfromLaplace}
This result is inspired by \cite[Theorem 3.14]{SoaveZilio}, which deals with a system with the Laplace operator. Comparing the hypothesis of this reference with ours, besides the technical changes, the main difference is condition ($h_5$). This condition was essential in the proof of Lemma \ref{divClaim}; this key bound is straightforward in case the operator ir the Laplacian, since $-\Delta(\frac{1}{|y|^{N-2}}) = C\delta$; in our case, we need ($h_5$) when approximating  the fundamental solution of the operator: $-\div(A_n(x)\nabla (\cdot))$ by the fundamental solution of the Laplacian operator, recalling also that $A_n(y) \sim Id$ for $y$ close to zero. Later on we will apply Theorem \ref{AltCaffMonotonicity} to the blowup sequence $(v_n)$ introduced in Section \ref{chapter:background}, and condition ($h5$) will be a consequence of the Almgren's-monotonicity formula, see Section \ref{sec:conditions} below for the details.
\end{remark}

\begin{remark}
 In condition ($h4$) we assume that $r \in ]1,R_n[$, but we notice that the lower bound $1$ can be replaced with any other positive.
\end{remark}

The rest of the section is devoted to the proof of Theorem \ref{AltCaffMonotonicity}, which we divide into several lemmas. Before we start, we define:
\begin{equation}
\label{GammaCaffareliDef}
\gamma(t):= \sqrt{(\frac{N-2}{2})^2 + t} - \frac{N-2}{2},
\end{equation}
a natural quantity within this context (see \cite{FriedmanSphere}), which satisfies $\gamma(t)^2 + (N-2)\gamma(t) = t$. The following lemma clarifies the definition of $\Lambda_{i,n}(r)$, its relation with $J_{i,n}(r)$ and the need of Lemma \ref{divClaim}.

\begin{lemma}
	\label{derivativeLemma}
	Let $u = (u_{1,n},...,u_{l,n})$ be a positive solution of \eqref{altCafEq}, assume hypothesis ($h_0$)-($h_6$), and \textbf{(A1)}. Then, for $i=1,2$,
	\begin{equation}
	    \label{AltCaffJequation}
    	J_{i,n}(r) \leq \frac{r}{2\gamma(\Lambda_i(r))} \int_{\partial B_r }\left(\langle A_n(y)\nabla u_{i,n}, \nabla u_{i,n}  \rangle 
    	- 
    	M_n a_n(y)  |u_{1,n}|^{2}|u_{2,n}|^2 - u_{i,n}f_{i,n}(y,u_{i,n})\right)|y|^{2-N}.
	\end{equation}
\end{lemma}
\begin{proof} We prove the statement for $i=1$.

\noindent    \textbf{Step 1.}
    Check that
    \begin{equation}\label{AltCaffJequation_aux}
    J_{1,n}(r) = \int_{B_r}\frac{1}{2} u_{1,n}^2\div( A_n(y) \nabla |y|^{2-N} ) + \int_{\partial B_r} \left(\frac{1}{r^{N-2}} u_{1,n}\langle A_n(y)\nabla u_{1,n}, \nu_y \rangle+ 
		\frac{N-2}{2 r^{N-1}}	 \mu_n(y)  u_{1,n}^2\right).
		\end{equation}
	First, we test the equation \eqref{altCafEq} for $u_{1,n}$ by $u_{1,n}(y)|y|^{2-N}$, obtaining
	\begin{align*}
	    \int_{B_r}
	    & \langle
	        A_n(y)
	        \nabla
	        u_{1,n}
	        ,
	        \nabla
	        u_{1,n}
	    \rangle
	    |y|^{2-N}
	    dy
	=
	   	-\int_{B_r}
		u_{1,n}(y)
		\langle A_n(y)\nabla u_{1,n},   \nabla |y|^{2-N}\rangle dy \\
	&+
	\int_{B_r}
	\left(
		\sum_{i\neq 1}^l
		M_n a_n(y) |u_{i,n}|^{2}|u_{1,n}|^2 
	+ 
		u_{1,n}f_{1,n}(y,u_{1,n})
	\right)
		|y|^{2-N}dy +
	    \frac{1}{r^{N-2}}
		\int_{\partial B_r} 
		u_{1,n}
		\langle A_n(y)\nabla u_{1,n},  \nu_y \rangle .
	\end{align*}
Using the above equation and the fact that $M_n<0$, $a(y)>0$ and the fact that $\nabla (|y|^{2-N}) = \frac{(2-N) y}{r^N} = \frac{(2 - N)}{r^{N-1}} \nu_y$ for $y \in \partial B_r$ , we conclude that:	
	\begin{align*}
		J_{1,n}(r) &\leq \int_{B_r} \left(\langle A_n(y)\nabla u_{1,n}, \nabla u_{1,n}  \rangle - \sum_{i \neq 1}^lM_n a_n(y) |u_{i,n}|^{2}|u_{1,n}|^2 - u_{1,n}f_{1,n}(y,u_{1,n})\right)|y|^{2-N}dy \\
		&= -\int_{B_r}u_{1,n}\langle A_n(y)\nabla u_{1,n},   \nabla |y|^{2-N}\rangle dy + \frac{1}{r^{N-2}}\int_{\partial B_r} u_{1,n}\langle A_n(y)\nabla u_{1,n},  \nu_y \rangle d\sigma(y)\\
		&= -\frac{1}{2}\int_{B_r}\langle \nabla( u_{1,n}^2),  A_n(y) \nabla |y|^{2-N}\rangle dy+ \frac{1}{r^{N-2}}\int_{\partial B_r} u_{1,n}\langle A_n(y)\nabla u_{1,n},  \nu_y \rangle d\sigma(y)\\
		&= \int_{B_r}\frac{1}{2} u_{1,n}^2\div( A_n(y) \nabla |y|^{2-N} ) + \frac{1}{r^{N-2}}\int_{\partial B_r} u_{1,n}\langle A_n(y)\nabla u_{1,n}, \nu_y \rangle d\sigma(y)\\
		&\quad \quad+ 
		\frac{N-2}{2 r^{N-1}}
		\int_{\partial B_r}
		\langle  \nu_y, A_n(y) \nu_y \rangle 
		u_{1,n}^2
		d\sigma(y) 
		\\
		&= 
		\int_{B_r}\frac{1}{2} 
		u_{1,n}^2 
		\div( A_n(y) \nabla |y|^{2-N} ) 
		dy
		+ 
		\frac{1}{r^{N-2}}
		\int_{\partial B_r} 
		u_{1,n}
		\langle A_n(y)\nabla u_{1,n}, \nu_y \rangle 
		d\sigma(y),
	\end{align*}
which implies \eqref{AltCaffJequation_aux}, by recalling that $\mu_n(y) = \langle A_n(y) \nu_y, \nu_y \rangle$.
		
\noindent		\textbf{Step 2.} Conclusion of inequality \eqref{AltCaffJequation}.
		
		Using Lemma \ref{divClaim}  in inequality \eqref{AltCaffJequation_aux}, we obtain
	\begin{equation}
		\label{randomEqJ}
		J_{1,n}(r) 
		\leq 
		\frac{1}{r^{N-2}}
		\int_{\partial B_r} 
		u_{1,n}
		\langle A_n(y)\nabla u_{1,n},  \nu_y \rangle 
		d\sigma (y)
		+ 
		\frac{N-2}{2 r^{N-1}}
		\int_{\partial B_r}
		(1+\alpha c_n r) 
		\mu_n(y) 
		u_{1,n}^2 
		d\sigma(y).
	\end{equation}
	Now, by Young's inequality, one obtains:
	\begin{align*}
		&\int_{\partial B_r} 
		u_{1,n}
		\langle A_n(y)\nabla u_{1,n},  \nu_y \rangle
		d\sigma(y) 
		\\
		\leq&
		\frac{\gamma(\Lambda_{1,n}(r))}{2r}\int_{\partial B_r}u_{1,n}^2\mu_n(y)d\sigma(y)+\frac{r}{2\gamma(\Lambda_{1,n}(r))}\int_{\partial B_r}\frac{\langle A_n(y)\nabla u_{1,n},  \nu_y \rangle^2}{\mu_n(y)}d\sigma(y)
		\\
		\leq& 
		\frac{\gamma(\Lambda_{1,n}(r))}{2r}
		\int_{\partial B_r}
		(1+\alpha c_nr)
		u_{1,n}^2
		\mu_n(y)
		d\sigma(y) 
		+ 
		\frac{r}{2\gamma(\Lambda_{1,n}(r))}
		\int_{\partial B_r}
		\frac{
			\langle A_n(y)\nabla u_{1,n},  \nu_y \rangle^2
		}{
			\mu_n(y)
		}
		d\sigma(y).
	\end{align*}
	Applying this inequality to equation \eqref{randomEqJ}:
	\begin{align*}
	J_{1,n}(r) 
	\leq&\frac{1}{2r^{N-1}\gamma(\Lambda_{1,n}(r))} \Bigg[ \left(\gamma(\Lambda_{1,n}(r))^2  + (N-2)\gamma(\Lambda_{1,n}(r)) \right) \int_{\partial B_r}(1+\alpha c_nr)u_{1,n}^2\mu_n(y)d\sigma(y)\\ 
	&+ r^2\int_{\partial B_r}\frac{\langle A_n(y)\nabla u_{1,n},  \nu_y \rangle^2}{\mu_n(y)}d\sigma(y)\Bigg].
	\end{align*}
	Using the fact that $\gamma(\Lambda_{1,n}(r))^2  + (N-2)\gamma(\Lambda_{1,n}(r)) = \Lambda_{1,n}(r)$ then by the definition given in \eqref{LambdaAltCaf1} we obtain:
	\begin{align}
			J_{1,n}(r) 
			\leq& 
			\frac{r^2}{2r^{N-1}
				\gamma(\Lambda_{1,n}(r))}
			\Bigg(
			\int_{\partial B_r} 
			\left(
			\langle A_n(y) \nabla_{\theta} u_{1,n}, \nabla_{\theta} u_{1,n} \rangle 
			- 
			\frac{\langle A_n(y) \nabla_{\theta}u_{1,n}, \nu_y \rangle^2}{\mu_n(y)}
			+
			\frac{
				\langle A_n(y)\nabla u_{1,n},  \nu_y \rangle^2
			}{
				\mu_n(y)
			}
			\right)d\sigma(y)
			\nonumber \\
			&\quad\quad+
			\int_{\partial B_r}\left(
			-
			M_na_n(y)|u_{1,n}|^{2} |u_{2,n}|^{2}
			-u_{1,n}f_{1,n}(y,u_{1,n})
			\right)
			d\sigma(y)
			\Bigg). 		\label{anothRandomEqJ}
	\end{align}
	Now, to compute $\frac{\langle A_n(y)\nabla u_{1,n}(y), \nu_y \rangle^2}{\mu_n(y)}$, we use the following auxiliary equations:
	$$
	    \nabla u_{1,n}
	  =
	    \nabla_{\theta} u_{1,n} 
	  +
	    (\partial_{\nu} u_{1,n}) \nu_y,
	    \qquad
	    \langle 
	        A_n(y)\nabla u_{1,n}, \nu_y \rangle
	  =
	    \langle 
	        A_n(y) \nabla_\theta u_{1,n}, \nu_y
	    \rangle
	  +
	    \mu_n(y)(\partial_\nu u_{1,n}).
	$$
	With them, we obtain:
	$$
	    \frac{\langle A_n(y)\nabla u_{1,n}, \nu_y \rangle^2}{\mu_n(y)}
	 =
	    \frac{
	        \langle 
	        A_n(y)
	        \nabla_\theta 
	        u_{1,n}, 
	        \nu_y
	        \rangle^2
	    }
	    {
	        \mu_n(y)
	    }
	 +
	    2(\partial_\nu u_{1,n})
	    \langle 
	        A_n(y) \nabla_\theta u_{1,n},
	        \nu_y
	    \rangle
	 +
	    (\partial_\nu u_{1,n})^2 \mu_n(y).
	$$ 
	We also notice:
	\begin{align*}
		\langle A_n(y)\nabla u_{1,n}, \nabla u_{1,n} \rangle 
		=& 
		\langle A_n(y)\nabla_{\theta} u_{1,n}, \nabla_{\theta} u_{1,n} \rangle+ 
		2\langle A_n(y)\nabla_{\theta}u_{1,n},  \nu_y \rangle (\partial_\nu u_{1,n}) + (\partial_\nu u_{1,n})^2\mu_n(y),
	\end{align*}
	and from this we conclude that:
	$$ 
	\langle A_n(y)\nabla u_{1,n}, \nabla u_{1,n} \rangle
	=
	\langle A_n(y) \nabla_{\theta} u_{1,n}, \nabla_{\theta} u_{1,n} \rangle
	- 
	\frac{\langle A_n(y) \nabla_{\theta}u_{1,n}, \nu_y \rangle^2}{\mu_n(y)} 
	+ 
	\frac{\langle A_n(y)\nabla u_{1,n},  \nu_y \rangle^2}{\mu_n(y)}.$$
	Thus, applying this equality to equation \eqref{anothRandomEqJ}, we conclude \eqref{AltCaffJequation}, as wanted.
\end{proof}

Before we proceed, it is important at this point to simplify the notation of $\Lambda_{i,n}(r)$. We  rewrite one of the terms as follows:
\begin{align*}
	\Big\langle 
	A_n(y) \nabla_{\theta} u_{1,n}, \nabla_{\theta} u_{1,n} \Big\rangle &- \frac{\langle A_n(y) \nabla_{\theta}u_{1,n}, \nu_y \rangle^2}{\mu_n(y)}\\
	=& \Big\langle
	A_n(y) \nabla_{\theta} u_{1,n}
	, 
	\nabla_{\theta} u_{1,n} 
	\Big\rangle 
	- 
	\Big\langle \frac{\langle A_n(y) \nu_y, \nabla_{\theta}u_{1,n} \rangle}{\mu_n(y)}A_n(y) \nu_y, \nabla_{\theta}u_{1,n} 
	\Big\rangle \\
	=& \Big\langle A_n(y) \nabla_{\theta} u_{1,n} -\frac{\langle A_n(y) \nu_y, \nabla_{\theta}u_{1,n} \rangle}{\mu_n(y)}A_n(y) \nu_y, \nabla_{\theta} u_{1,n} \Big\rangle.
\end{align*}
\begin{definition}\label{definition:operatorB}
For $y \neq 0$, we define the operator $B_n(y)$ by:
$$
B_n(y) v 
:= 
\left(
A_n(y) v 
- 
\frac{\langle A_n(y) \nu_y, v \rangle}{\mu_n(y)}A_n(y) \nu_y\right)
$$
and write, for $i,j\in \{1,2\}$, $i\neq j$,
	\begin{equation}\label{eq:equivalent_LAMBDA}
	\Lambda_{i,n}(r) 
	= 
	\frac{r^2\int_{\partial B_r} \langle B_n(y) \nabla_{\theta} u_{i,n}, \nabla_{\theta} u_{i,n} \rangle 
	- 
	M_na_n(y)|u_{i,n}|^{2} |u_{j,n}|^2
		+
		u_{i,n}f_{i,n}(y,u_{i,n})d\sigma(y)}{\int_{\partial B_r}(1+\alpha(rc_n))\mu_n(y)u_{i,n}^2d\sigma(y)}
	\end{equation}

\end{definition}
\noindent It is straightforward to see that $B_n(y)$ is a symmetric operator. Moreover, for all $\xi \in \mathbb{R}^N$, we have that $B_n(y)\xi \in T_y (\partial B_{|y|})$. This is the case since:
\begin{align*}
    \langle B_n(y) \xi, \nu_y \rangle
=
    \langle 
    A_n(y) \xi 
- 
    \frac{\langle A_n(y)\nu_y, \xi \rangle}{\mu_n(y)}
     A_n(y) \nu_y, \nu_y
    \rangle =
    \langle A_n(y) \xi, \nu_y \rangle
-
    \frac{\langle A_n(y) \nu_y, \xi \rangle \mu_n(y)}{\mu_n(y)}
=
    0.
\end{align*}

The following result shows what is the ellipticity constant of $B_n$, and provides a bound for its distance to the identity operator.
\begin{lemma}
    \label{MatrixBoundsOnB}
    Suppose that \textbf{(A1)} is satisfied for the sequence $A_n$:
    \begin{equation}
    \label{ellipticityCaf1}
    \langle A_n(y) \xi, \xi \rangle
\geq
    \theta |\xi|^2
    \qquad
    \forall \xi \in \mathbb{R}^N,
    \end{equation}
and ($h0$) and ($h4$) hold true.    Then there exist $C,\tilde{\theta}>0$ depending only on $\theta$, $M$ and $N$ such that, for all $y \in B_{R_n}$,
    \begin{equation}
    \label{ineq2AltCafMatrixB}
        \langle B_n(y)\xi, \xi \rangle
    \geq
        \tilde{\theta}|\xi|^2
        \quad
        \forall
        \xi \in 
        T_y 
        \left(
            \partial B_{|y|}
        \right),
\quad \text{     and } \quad 
        \|    Id|_{T_{y}\partial B_{|y|}} -  B_n(y) \| \leq C  \|Id_{\mathbb{R}^N}-A_n(y)\|.
    \end{equation}
\end{lemma}
\begin{proof}
    Given $\xi \in T_y (\partial B_{|y|})$, we obtain:
    \begin{align*}
        |(B_n(y) - &Id_{T_y (\partial B_{|y|})} ) \xi|
    =
        |
        (
            A_n(y)\xi 
        - 
            \frac
            {
            \langle 
                A_n(y) \nu_y, \xi 
            \rangle
            }
            {
            \mu_n(y)
            } 
        )
        -
        \xi
        |
    =
        |
            (A_n(y) - Id)\xi 
        - 
            \frac
            {
            \langle 
                \nu_y, A_n(y)\xi 
            \rangle
            }
            {
            \mu_n(y)
            } 
        |\\
    &\leq
        |(A_n(y)-Id)\xi|
    +
        |\frac{\langle \nu_y, A_n(y)\xi \rangle}{\mu_n(y)}|
    =
        |(A_n(y)-Id)\xi|
    +
        |\frac{\langle \nu_y, (A_n(y)-Id)\xi \rangle}{\mu_n(y)}|\\
    &\leq
           \sqrt{N} \|A_n(y)-Id\|\cdot|\xi|
        +
            \frac{\sqrt{N}}{\mu_n(y)}
            \|A_n(y)-Id\|\cdot|\xi|
    \leq
        \sqrt{N}
        (1 + \frac{1}{\theta})
        \|A_n(y)-Id\|
        \cdot 
        |\xi|,
    \end{align*}
    where we used $\langle \xi, \nu_y \rangle = 0$, and the ellipticity constant of \eqref{ellipticityCaf1} to obtain $\frac{1}{\mu_n(y)} \leq \frac{1}{\theta}$. Taking $C =\sqrt{N}(1 + \frac{1}{\theta})$ concludes the second statement in \eqref{ineq2AltCafMatrixB}. Regarding the first statement in \eqref{ineq2AltCafMatrixB}, given $\xi \in T_y(\partial B_{|y|})$:
    \begin{align*}
        \langle B_n(y) \xi, \xi \rangle&=  \Big\langle   A_n(y)\xi  -  \frac{    \langle  A_n(y)\nu_y , \xi  \rangle  A_n(y)\nu_y } {   \mu_n(y)  } ,  \xi \Big\rangle= \langle A_n(y)\xi,\xi \rangle -  \frac{\langle A_n(y)\nu_y,\xi\rangle^2}{\mu_n(y)}\\
        &= \langle A_n(y)\xi,\xi \rangle - \frac{\langle (A_n(y)-Id)\nu_y,\xi\rangle^2}{\mu_n(y)} \geq \theta |\xi|^2 -\frac{\langle (A_n(y)-Id)\nu_y,\xi\rangle^2}{\mu_n(y)} \\
        &\geq \theta |\xi|^2 - \frac{\sqrt{N}}{\theta} \|A_n(y)-Id\|\cdot |\xi|^2\geq      \theta |\xi|^2 -   \frac{\sqrt{N}}{\theta}  c_nR_n |\xi|^2\geq   \frac{\theta}{2}|\xi|^2,
    \end{align*}
    where we used $(h_0)$ and $(h_5)$ in the last inequalities.
\end{proof}

Given $\lambda_n > 0$ and a sequence $\tilde{c}_n \rightarrow 0$, define, for each $n$ the subspace of $(H^1(\partial B_1))^2$ given by:
\begin{equation*}
	H_{\lambda_n,\tilde{c}_n} 
	= 
	\Big\{
	(u,v) \in (H^1(\partial B_1))^2: \int_{\partial B_1}(1+\alpha \tilde{c}_n)\mu_n(y)u^2dy = 1 ;\,\, \int_{\partial B_1}(1+\alpha \tilde{c}_n)\mu_n(y)v^2 dy = \lambda_n
	\Big\}.
\end{equation*}

A fundamental result in the proof of the classical Alt-Caffarelli-Friedman's monotonicity formula  \cite{FriedmanSphere} is the following Friedman-Hayman inequality \cite{FriedlanHayman}:
\begin{equation}\label{eq:FH_ineq}
	\gamma\left(
	\frac{\int_{\partial B_1} |\nabla_\theta f|^2
	d\sigma(y)}
	{\int_{\partial B_1}f^2d\sigma(y)}
	\right) 
	+ 
	\gamma\left(
	\frac{\int_{\partial B_1} |\nabla_\theta g|^2
	d\sigma(y)}{\int_{\partial B_1} g^2d\sigma(y)}
	\right)\geq 2,\ \ \text{ for every $f,g\in H^1(\partial B_1)$ with $fg\equiv 0$,}
\end{equation}
where $\gamma$ is defined in \eqref{GammaCaffareliDef}. Inspired by  \cite[Lemma 4.2]{SphereDensityAlt} and \cite[Lemma 3.10]{SoaveZilio}, we prove the following result.

\begin{lemma}
	\label{sphereLemma}
	Fix $\overline{\lambda}>1$ and let $\delta$, $\tilde{c}$, $\tilde{M}$, $\tilde{\theta}$ be positive constants. Then, for every $0 <\eta < \frac{1}{4}$, there exists $C = C(N, \overline{\lambda}, \tilde{c}, \delta, \eta, \alpha, \tilde{\theta}, \tilde{M})>0$ such that, whenever:
	\begin{equation}\label{MatrixSphereBoundsCaffareli0}
	\frac{1}{\overline{\lambda}} < \lambda_n < \overline{\lambda}, \qquad 0
	\leq 
	\epsilon_n <  (\frac{N-2}{2})^2 - \delta,
	\qquad
    \tilde{c}_n
    \to 0,
    \
    \tilde{c}_n \leq \tilde{c},
    \qquad
    k_n\geq0,
	\end{equation}
	and $\tilde{B}_n$ is a sequence of symmetric operators such that
	\begin{equation}
	\label{MatrixSphereBoundsCaffareli}
	\tilde{\theta}
	|\xi|^2\leq \langle \tilde{B}_n(y)\xi,\xi \rangle
	\quad
	\forall 
	\xi \in T_y (\partial B_1),
	\qquad
	\sup_{y \in \partial B_1}
	\|D\tilde{B}_n(y)\| \leq \tilde{M},
	\qquad
	    \sup_{y \in \partial B_1}
	    \|
	        \tilde{B}_n(y) - Id|_{T_y \partial B_1}
	    \|
	\leq
	\tilde{M}
	\tilde{c}_n,
	\end{equation}
	we have:
	\begin{multline*}
	\min_{(u,v)\in H_{\lambda_n,\tilde{c}_n}} \gamma\left(
	\int_{\partial B_1} \langle \tilde{B}_n(y)\nabla_{\theta} u, \nabla_{\theta} u \rangle + k_nu^2v^2-\epsilon_n
	\right)
	+\gamma\left(
	\frac{\int_{\partial B_1} \langle \tilde{B}_n(y)\nabla_{\theta} v, \nabla_{\theta} v \rangle + k_nu^2v^2-\epsilon_n \lambda_n}{\lambda_n}
	\right)\\
	\geq 
	2 -C(\epsilon_n + k_{n}^{-\eta} + \tilde{c}_n).
	\end{multline*}
\end{lemma}

This proof is substantially harder than the one in the case where the operator is the Laplacian. Indeed, in order to obtain \eqref{eq:FH_ineq}, or to obtain \cite[Lemma 4.2]{SphereDensityAlt}, \cite[Lemma 3.10]{SoaveZilio}, an important part of the argument is to symmetrize the solutions of the underlying  minimization problem; this, in particular, allows to conclude directly that the sequence of minimizers is Lipschitz continuous, and that its level sets are circles.  The proof in our case is harder and required new ideas: since we are dealing with an operator with variable coefficients, we cannot use a symmetrization argument; instead, we rely on Theorem C, from which we obtain uniform H\"older bounds.

We now prove Lemma \ref{sphereLemma},  after which we are able to conclude the proof of Theorem \ref{AltCaffMonotonicity}.

\begin{proof}[Proof of Lemma \ref{sphereLemma}] 	
	We see that the minimization problem in the theorem is equivalent to minimizing in $H_{1,\tilde c_n}$ (by replacing $v$ with $\frac{v}{\sqrt{\lambda_n}}$):
	\begin{equation}
	\label{GammaMinimization}
	\min_{H_{1,\tilde{c}_n}} \gamma\left(
	\int_{\partial B_1} \langle \tilde{B}_n(y)\nabla_{\theta} u, \nabla_{\theta} u \rangle + k_n \lambda_n u^2v^2-\epsilon_n
	\right) 
	+ \gamma\left(
	{\int_{\partial B_1} \langle \tilde{B}_n(y)\nabla_{\theta} v, \nabla_{\theta} v \rangle + k_nu^2v^2-\epsilon_n}
	\right).
	\end{equation}
The direct method of Calculus of Variations yields the existence of a minimizer  $(u_n, v_n)$ of \eqref{GammaMinimization} satisfying:
	\begin{equation*}\label{GammaMinimization_norm}
	\int_{\partial B_1}(1+\alpha \tilde{c}_n)\mu_n(y)u_n^2d\sigma(y) = 1,
	\qquad
	\int_{\partial B_1}(1+\alpha \tilde{c}_n)\mu_n(y)v_n^2d\sigma(y)
	=
	1.
	\end{equation*}
By eventually replacing $(u,v)$ by $(|u|,|v|)$,  we may assume without loss of generality that the minimizer $(u_n,v_n)$ is nonnegative.   (notice that $ \int_{\partial B_1}\langle \tilde B_n(y)\nabla_\theta u, \nabla_\theta u \rangle d\sigma(y) =  \int_{\partial B_1}\langle \tilde{B}_n(y) \nabla_\theta |u|, \nabla_\theta |u| \rangle d\sigma(y)$). Let:
\begin{align}
	x_n &= \int_{\partial B_1} 
	\left(
	    \langle \tilde{B}_n(y)\nabla_{\theta} u_n, \nabla_{\theta} u_n \rangle + k_n \lambda_n u_n^2v_n^2 
	\right)d\sigma(y)
	-\epsilon_n, \label{def_of_x_n}\\
	y_n &= {\int_{\partial B_1} 
	\left(
	\langle \tilde{B}_n(y)\nabla_{\theta} v_n, \nabla_{\theta} v_n \rangle + k_nu_n^2v_n^2
	\right)
	d\sigma(y)-\epsilon_n}.\nonumber
	\end{align}
	There exist two Lagrange multipliers $\sigma_{1,n}$ and $\sigma_{2,n}$ such that:
	\begin{equation}
		\label{SphereEq}
		\begin{cases}
			-\div_{\partial B_1}(\tilde{B}_n(y)\nabla_{\theta} u_n) = -k_n(\lambda_n + \frac{\gamma'(y_n)}{\gamma'(x_n)})u_nv_n^2 + \frac{\sigma_{1,n}}{\gamma'(x_n)}(1+ \alpha \tilde{c}_n)\mu_n(y) u_n\\
			-\div_{\partial B_1}(\tilde{B}_n(y)\nabla_{\theta} v_n) = -k_n(1 + \frac{\lambda_n\gamma'(x_n)}{\gamma'(y_n)})v_nu_n^2 + \frac{\sigma_{2,n}}{\gamma'(y_n)}(1+\alpha\tilde{c}_n)\mu_n(y)v_n
		\end{cases}
	\end{equation}
	We divide the proof in several steps.

\noindent	\textbf{Step 1.} There exists $C = C(N, \tilde{M}, \tilde{\theta}, \delta, \overline{\lambda}, \tilde{c})>0$ such that $\|u_n\|_{H^1(\partial B_1)},\|v_n\|_{H^1(\partial B_1)}\leq C$, $\frac{1}{C}\leq \gamma'(x_n),\gamma'(y_n)\leq C$ and $0 \leq \sigma_{1,n}, \sigma_{2,n}<C$.
	
	\vspace{2mm}
	
	Fix $\phi = \phi^+ - \phi^- \in H^1(\partial B_1)$ such that $\|\phi^+\|_{H^1(\partial B_1)}^2 = \|\phi^-\|_{H^1(\partial B_1)}^2 = 1$. Then, since $(u_n, v_n)$ is the minimizer for \eqref{GammaMinimization}, and $\phi^+\cdot \phi^- = 0$, we have:
	\begin{multline*}
	    \gamma\left(
	\int_{\partial B_1} 
	\left(
    	\langle 
    	\tilde{B}_n(y)\nabla_{\theta} u_n, 
    	\nabla_{\theta} u_n 
    	\rangle 
    	+ 
    	k_n \lambda_n u_n^2v_n^2
	\right)
	-
	\epsilon_n
	\right) 
	+ \gamma\left(
	{
    	\int_{\partial B_1}
    	\left(\langle \tilde{B}_n(y)\nabla_{\theta} v_n, \nabla_{\theta} v_n \rangle + k_nu_n^2v_n^2\right)
    	-\epsilon_n
	}
	\right)\\
	\leq
	\gamma \Big(
	       \frac{
	        \int_{\partial B_1}
	    \langle
	        \tilde{B}_n(y)
	        \nabla_{\theta}
	        \phi^+
	    ,
	        \nabla_\theta
	        \phi^+
	    \rangle
	-
	    \epsilon_n
	    (1+\alpha \tilde{c}_n)
	    \mu_n(y)
	    (\phi^+)^2
	    }
	    {
	        \int_{\partial B_1}
    	    (1+\alpha \tilde{c}_n)\mu_n(y)
    	    (\phi^+)^2
	    }
	    \Big)
	    \\
	    +
	    \gamma \Big(
	       \frac{
	        \int_{\partial B_1}
	    \langle
	        \tilde{B}_n(y)
	        \nabla_{\theta}
	        \phi^-
	    ,
	        \nabla_\theta
	        \phi^-
	    \rangle
	-
	    \epsilon_n
	    (1+\alpha \tilde{c}_n)
	    \mu_n(y)
	    (\phi^-)^2
	    }
	    {
	        \int_{\partial B_1}
    	    (1+\alpha \tilde{c}_n)
    	    \mu_n(y)
	       (\phi^-)^2
        }
	    \Big).
	\end{multline*}
	Therefore, we need a uniform bound on the right-hand-side of the previous inequality; we just bound the term involving $\phi^+$, since the computations for the other term are analogous. We use equations \eqref{MatrixSphereBoundsCaffareli0} and \eqref{MatrixSphereBoundsCaffareli} to conclude
	\begin{align*}
	    \int_{\partial B_1}
	    \left(
	    \langle
	        \tilde{B}_n(y)
	        \nabla_{\theta}
	        \phi^+
	    ,
	        \nabla_\theta
	        \phi^+
	    \rangle
	-
	    \epsilon_n
	    (1+\alpha \tilde{c})
	    \mu_n(y)
	    \phi^+
	    \right)
	&\leq
        \|\tilde{B}_n(y)\|
    \cdot
        \|\phi^+\|^2_{H^1(\partial B_1)}
  \leq 
        (1+\tilde{M}\tilde c),\\
	    \int_{\partial B_1}
	    (1+\alpha \tilde{c}_n)\mu_n(y)
	    (\phi^+)^2
	&\geq
	    \int_{\partial B_1}
	    \mu_n(y)(\phi^+)^2
	  	\geq
	    \tilde{\theta}\int_{\partial B_1}(\phi^+)^2
	>
	    0.
	\end{align*}
	Thus, since $\gamma$ is monotone, there exists $C = C(N, \tilde{c}, \tilde{\theta}, \tilde{M}) = 2\gamma (\frac{(1+\tilde{M}\tilde c)}{\tilde \theta}) >0$ such that 
	  $  \gamma(x_n)+\gamma(y_n)
	\leq
	    C
	$
	for all $n \in \mathbb{N}$. 
	Since $\epsilon_n < (\frac{N-2}{2})^2 - \delta$, and $x_n,y_n\geq -\epsilon_n$, then:
	$
	    \delta
	    -
	    (\frac{N-2}{2})^2
	    \leq 
	    x_n, y_n.
	$
	Thus, due to the expression of $\gamma$ in \eqref{GammaCaffareliDef}, there exists $C = C(N, \tilde{c}, \tilde{\theta}, \delta)$ large enough such that:
	$$
	\frac{1}{C}
	\leq
	\gamma'(x_n)\leq C,
	\quad
	\frac{1}{C}
	\leq
	\gamma'(y_n)
	\leq
	C, \quad
	    \int_{\partial B_1}
	    \Big(
	    \langle
	        \tilde{B}_n(y)
	        \nabla_\theta
	        u_n
	    ,
	        \nabla_\theta
	        u_n
	    \rangle
	 +
        \lambda_n 
        k_n
        u_n^2
        v_n^2
        \Big)
        d\sigma(y)
    \leq
        C.
	$$
In particular, this implies that each one of the two terms in the last inequality are bounded; using the uniform ellipticity of $\tilde{B}_n(y)$, we obtain that $u_n$ and $v_n$ are uniformly bounded in $H^1(\partial B_1)$ with a bound depending again on $\tilde{\theta}$. Finally, as for the Lagrage's multiplier's, there exists $C = C(N,\tilde{c}, \tilde{\theta}, \delta, \overline{\lambda}, \tilde{M})$ such that
	$$
    0
    \leq
    \sigma_{1,n}
	= 
	\gamma'(x_n)
	\int_{\partial B_1} 
	\Big(
	\langle \tilde{B}_n(y)\nabla_{\theta}u_n, \nabla_{\theta}u_n \rangle +  k_n(\lambda_n+\frac{\gamma'(y_n)}{\gamma'(x_n)})u_n^2v_n^2 
	\Big)
	d\sigma(y)
	\leq 
	C,
	$$
and similarly for $\sigma_{2,n}$.
	
\noindent	\textbf{Step 2.} We check that there exists $C = C(N, \tilde{c}, \tilde{\theta}, 
	\delta, \overline{\lambda}, \alpha, \tilde{M})$ such that
	$\|u_n\|_{L^\infty(\partial B_1)}, \|v_n\|_{L^\infty(\partial B_1)} \leq C$ and moreover that, 
given $0 <\beta < 1$, there exists $D = D(N, \tilde{c}, \tilde{\theta}, 
	\delta, \overline{\lambda}, \beta, \tilde M, \alpha)$ such that $(u_nv_n)^{\frac{1}{2}+\frac{1}{2\beta}}(x) \leq D k_n^{-\frac{1}{2}}$ for all $x \in \partial B_1$. In particular:
	\begin{align*}
		&x \in \{v_n-u_n\leq 0\}
		\implies
		v_n(z)\leq D^\frac{1}{2}k_n^{-\frac{\beta}{2(\beta+1)}},\qquad x \in \{v_n-u_n\geq 0\}
		\implies
		u_n(z)\leq D^\frac{1}{2}k_n^{-\frac{\beta}{2(\beta+1)}}.
	\end{align*}
	
	\vspace{2mm}
	
	To first show the uniform $L^\infty(\partial B_1)$ bound, we notice that the functions $u_n,v_n$ are nonnegative, and $k_n\geq 0$, so by equation \eqref{SphereEq} we have
	$$-\div_{\partial B_1}
	\left(\tilde{B_n}(y)\nabla_{\theta} u_n\right) 
	\leq 
	\frac{C\sigma_{1,n}}{\gamma'(x_n)}u_n, \quad
	-\div_{\partial B_1}\left(\tilde{B_n}(y)\nabla_{\theta} v_n\right) 
	\leq 
	\frac{C\sigma_{2,n}}{\gamma'(y_n)}v_n.
	$$
	By Step 1, the sequences $\frac{\sigma_{1,n}}{\gamma'(x_n)}$, $\frac{\sigma_{2,n}}{\gamma'(y_n)}$ are uniformly bounded and the functions $u_n,v_n$ are uniformly bounded in $H^1(\partial B_1)$. Then, by a Brezis-Kato-type argument (see for instance \cite[Appendix B.2 B.3]{Struwe}), using the uniform ellipticity of $\tilde{B}_n$, we obtain a uniform $L^\infty(\partial B_1)$ bound on $u_n$ and $v_n$.
	
	Next, to prove the second part of Step 2, we suppose by contradiction that there exists a sequence of points $z_n \in \partial B_1$ such that:
	\begin{equation}
		\label{blowUpContradiction}
		k_n^{\frac{1}{2}}
		\left(
		v_n^{1+\frac{1}{\beta}}(z_n)
		u_n^{1+\frac{1}{\beta}}(z_n)
		\right)^{\frac{1}{2}}
		\rightarrow 
		\infty.
	\end{equation}
	By the uniform boundedness of $(u_n,v_n)$ we have, since $\beta>0$,
	\begin{equation}
	    \label{kGoToInfinity}
	    k_n \rightarrow \infty
	,\qquad
		k_n
		\left(
		v_n(z_n)
		u_n(z_n)
		\right)
		\rightarrow 
		\infty.
	\end{equation}	
	For each point $z_n \in \partial B_1$, we consider the parametrization of the sphere $\partial B_1$, $\phi_n: \mathbb{R}^{N-1} \rightarrow \partial B_1/\{-z_n\}$ given by the stereographic projection from that point, thus $\phi_n(0) = z_n$.
	Now fix the sequence $a_n$ for which there exists $C$ such that:
	$$
	\frac{1}{C}
	\leq
	a_n := \sqrt{
	        \frac{
	            \lambda_n
	       +
	            \frac{
	               \gamma'(y_n)
	            }{
	                \gamma'(x_n)
	            }
	       }{
	            1 
	       + 
	            \lambda_n
	            \frac{
	                \gamma'(x_n)
	            }{
	                \gamma'(y_n)
	            }
	       }
	       }
	   =
	   \sqrt
	   {
	   \frac{
	    \gamma'(y_n)
	   }{
	    \gamma'(x_n)
	   }
	   }
	   \leq
	   C
	$$
	and take the functions $\tilde{v}_n, \tilde{u}_n: \mathbb{R}^{N-1} \rightarrow \mathbb{R}$:
	$$\tilde{u}_n(z) = u_n(\phi_n(z)), \qquad 
	\tilde{v}_n(z) = a_nv_n(\phi_n(z)).$$
	This change of variables leads to the equation:
	\begin{equation}
		\label{SemiPlaneSemiSphereEq}
		\begin{cases}
			-\div_{\partial B_1}\left(
			\tilde{B}_n(y)\nabla_\theta u_n
			\right)_{y = \phi_n(z)}
			= 
			-k_n(1 + \frac{\lambda_n \gamma'(x_n)}{\gamma'(y_n)})\tilde{u}_n\tilde{v}_n^2 + \frac{\sigma_{1,n}(1+\alpha\tilde{c}_n)}{\gamma'(x_n)}\mu_n(\phi_n(z))\tilde{u}_n
			\\
			-a_n\div_{\partial B_1}\left(
		    	\tilde{B}_n(y)\nabla_\theta v_n
			\right)_{y=\phi_n(z)}
			= 
			-k_n(1 + \frac{\lambda_n\gamma'(x_n)}{\gamma'(y_n)})\tilde{v}_n\tilde{u}_n^2 
			+ 
			\frac{\sigma_{2,n}(1+\alpha\tilde{c}_n)}{\gamma'(y_n)}\mu_n(\phi_n(z))\tilde{v}_n
		\end{cases}
	\end{equation}
By Proposition \ref{DivergenceSphereProp}, in appendix, we know that:
	$$
	    \div_{\partial B_1}
	    \left(
			(\tilde{B}_n(y)\nabla_\theta u_n 
		\right)_{y=\phi_n(z)}
	=
	    (1+|z|^2)^{N-1}
	    \div_{\mathbb{R}^{N-1}}
	    \left(
	        \frac{1}{4(1+|z|^2)^{N-3}}
	        M_n(z)
	        \nabla_{\mathbb{R}^{N-1}}
	        \tilde{u}_n
	    \right)
	$$
	where $M_n(z) = (d\phi_n)^{-1}_{\phi(z)}\tilde{B}_n(\phi_n(z))(d\phi_n)_z$.
	For simplicity of notation, we define $\tilde{M}_n(z) := \frac{1}{4(1+|z|^2)^{N-3}}M_n(z)$ and $g(z) = (1+|z|^2)^{N-1}$. This allows us to rewrite equation \eqref{SemiPlaneSemiSphereEq} as:
	\begin{equation}
		\label{planeEq}
		\begin{cases}
			-\div_{\mathbb{R}^{N-1}}
			\left(
			\tilde{M}_n(z)
			\nabla_{\mathbb{R}^{N-1}}
			\tilde{u}
			\right) 
			= 
			-\frac{k_n}{g(z)}
			(1 + \frac{\lambda_n \gamma'(x_n)}{\gamma'(y_n)})\tilde{u}_n\tilde{v}_n^2 
			+ \frac
			{\sigma_{1,n}(1+\alpha\tilde{c}_n)}
			{g(z)\gamma'(x_n)}
			\mu_n(\phi_n(z))\tilde{u}_n
			\\
			-\div_{\mathbb{R}^{N-1}}\left(
			\tilde{M}_n(z)\nabla_{\mathbb{R}^{N-1}}\tilde{v}_n 
			\right) 
			= 
			-
			\frac{k_n}{g(z)}
			(1 + \frac{\lambda_n\gamma'(x_n)}{\gamma'(y_n)})\tilde{v}_n\tilde{u}_n^2 
			+ 
			\frac
			{
			    \sigma_{2,n}(1+\alpha\tilde{c}_n)
			}{
			    g(z)\gamma'(y_n)
			}
			\mu_n(\phi_n(z))\tilde{v}_n.
		\end{cases}
	\end{equation}
	By assumption, $\tilde{B}_n(y)$ has  ellipticity constant $\tilde{\theta}$, and $\|D\tilde{B}_n\| \leq \tilde{M}$.
Moreover, by Proposition \ref{DivergenceSphereProp} we know that, given the compact set $K = B_2(0) \subset \mathbb{R}^{N-1}$, there exists a constant $\tilde{C} = C(\tilde{M}, K)$ such that:
	$$
	    \langle 
	    \tilde{M}_n(y) \xi, \xi
	    \rangle_{\mathbb{R}^{N-1}}
	\geq
	    \frac{1}{4\cdot 5^{N-3}}
	    \theta
	    \langle 
	    \xi, \xi
	    \rangle_{\mathbb{R}^{N-1}}
	   \quad
	   \forall \xi,y
	   \in 
	   \mathbb{R}^{N-1},
	   \qquad
	   \|D \tilde{M}_n\|
	   \leq
	   \tilde{C},
	$$
	and
	$$
	    \frac
	    {1}
	    {5^{N-1}}
	\leq 
	   \frac{1}{g(z)}
	\leq
	    1 \quad
        \forall
        z \in \mathbb{R}^{N-1}.
	$$
	Since $(\tilde{u}_n, \tilde{v}_n)$ are uniformly bounded in $L^\infty$, satisfies the system \eqref{planeEq}, and  $\tilde{M}_n$ are uniformly elliptic over $B_2(0) \subset \mathbb{R}^{N-1}$, we are under the assumptions of Theorem C. Therefore, for each $0<\beta<1$, there exists a constant $C_\beta = C(\beta, B_2(0), \theta, M, N, \tilde{\theta}, \tilde{M})$ such that 
	\begin{equation*}
	    \label{AltCaffHolderBound1}
	    \|\tilde{u}_n\|_{C^{0,\beta}(B_1(0))}, \|\tilde{v}_n\|_{C^{0,\beta}(B_1(0))} \leq C_\beta.
	\end{equation*}
	
	    Define:
	\begin{equation*}
		t_n := \tilde{u}_n(0) + \tilde{v}_n(0);
	\end{equation*} 
we claim that both $\tilde{u}_n(0) \rightarrow 0$ and $\tilde{v}_n(0) \rightarrow 0$. We suppose, in view of a contradiction, that $\tilde{v}_n(0)\geq \overline{\delta}>0$ for all $n$. Then, by uniform convergence and boundedness of H\"older  norms, there exists a radius $R>0$ small enough such that 
	$$
	\inf_{x\in B_{2R}(0)}\tilde{v}_n(x) 
	\geq 
	\frac{\delta}{2}\qquad \text{ for all $n$}.
	$$
	From this, we conclude the differential inequality
	\begin{align*}
	    -\div_{\mathbb{R}^{N-1}}
	    (
	        \tilde{M}_n(z) 
	        \nabla \tilde{u}_n
	    )
	&=
	    -\frac{k_n}{g(z)}
			(1 + \frac{\lambda_n \gamma'(x_n)}{\gamma'(y_n)})\tilde{u}_n\tilde{v}_n^2 
			+ \frac
			{\sigma_{1,n}(1+\alpha\tilde{c}_n)}
			{g(z)\gamma'(x_n)}
			\mu_n(\phi_n(z))\tilde{u}_n
			\\
	&\leq
	    \big(
	    -
	        \frac{k_n}{g(z)}
	        \frac{\delta^2}{4}
	    +
			\frac{\sigma_{1,n}(1+\alpha\tilde{c}_n)}
			{\gamma'(x_n)}
			\mu_n(\phi_n(z))
		\big)
		\tilde{u}_n\\
	& \leq
	    \big(
	        -\frac{k_n}{5^{N-1}}
	        \frac{\delta^2}{4}
	        +
	        C^2(1+\alpha \tilde c)(1+\tilde c)^\frac{1}{2}
	    \big)
	    \tilde{u}_n \qquad \forall z \in B_{2R}(0)\subset \mathbb{R}^{N-1}.
	\end{align*}
From \eqref{kGoToInfinity} and since $k_n \rightarrow \infty$, for $n$ large enough we have: 
	$$
	-\frac{k_n}{5^{N-1}}
	        \frac{\delta^2}{4}
	        +
	        C^2(1+\alpha \tilde c)(1+\tilde c)^\frac{1}{2} \leq -\frac{k_n}
	        {5^{N-1}}
	        \frac{\delta^2}{8},
	 $$
	 and so
	 \begin{equation}
	    \label{divIneqAltCaff1}
	     -\div_{\mathbb{R}^{N-1}}
	    (
	        \tilde{M}_n(x) 
	        \nabla \tilde{u}_n
	    )
	 \leq
         -\frac{k_n}{5^{N-1}}
         \frac{\delta^2}{8}
         \tilde{u}_n.
	 \end{equation}
	Thus, by Lemma \ref{estimateLemma}-(2) and inequality \eqref{divIneqAltCaff1}, there exists $C$ and $c_2 = c_2(\beta, B_1(0), \theta, \tilde{M}, N, \tilde{\theta}, \tilde{c})$ such that
	$$
	\sup_{x \in B_{R}(0)}\tilde{u}_n(x)
	\leq
	Ce^{-c_2R\sqrt{k_n\delta^2}},
	$$
	and since $k_n \rightarrow \infty$ we have 
	$0\leq 
	k_n\tilde{u}_n(0) \leq Ck_ne^{-c_2R\sqrt{k_n\delta^2}}  \rightarrow 0.
	$.
	This contradicts the fact that:
	$
	C k_n\tilde{u}_n(0)
	\geq
	k_n 
	\tilde{v}_n(0)
	\tilde{u}_n(0)
	\to \infty$
	coming from \eqref{blowUpContradiction} and the uniform boundedness $\|\tilde{v}_n\|_{L^\infty(\partial B_1)} \leq C$.
	Thus we have that $\tilde{v}_n(0) \rightarrow 0$, and similarly $\tilde{u}_n(0) \rightarrow 0$, thus $t_n = \tilde{v}_n(0) + \tilde{u}_n(0) \rightarrow 0$, as claimed.
	
	Now define the functions $(\overline{u}_n, \overline{v}_n)$ and the matrix $\overline{M}_n(z)$ by:
	\begin{equation*}
	\label{definitionOfOverlines}
	\overline{u}_n(z) = \frac{1}{t_n}\tilde{u}_n(t_n^{\frac{1}{\beta}}z), \qquad
	\overline{v}_n(z) 
	= 
	\frac{1}{t_n}\tilde{v}_n(t_n^{\frac{1}{\beta}}z),
	\qquad
	\overline{M}_n(z)
	=
	\tilde{M}_n(t_n^\frac{1}{\beta}z)
	.
	\end{equation*}
	From \eqref{planeEq}, we have:
	\begin{equation}
		\label{overlineEqAltCaff}
		\begin{cases}
			-\div_{\mathbb{R}^{N-1}}
			\left(
			\overline{M}_n(z)
			\nabla_{\mathbb{R}^{N-1}}
			\overline{u}_n
			\right) 
			= 
			-t_n^{2+\frac{2}{\beta}}
			\frac{k_n}{g(t_n^\frac{1}{\beta}z)}
			(1 + \frac{\lambda_n \gamma'(x_n)}{\gamma'(y_n)})\overline{u}_n\overline{v}_n^2 
			+ 
			t_n^\frac{2}{\beta}
			\frac	{\sigma_{1,n}(1+\alpha\tilde{c}_n)}
			{g(t_n^\frac{1}{\beta}z)\gamma'(x_n)}
			\mu_n(\phi_n(t_n^\frac{1}{\beta}z))\overline{u}_n
			\\
			-\div_{\mathbb{R}^{N-1}}
			\left(
			\overline{M}_n(z)
			\nabla_{\mathbb{R}^{N-1}}
			\overline{v}_n
			\right) 
			= 
			-t_n^{2+\frac{2}{\beta}}
			\frac{k_n}{g(t_n^\frac{1}{\beta}z)}
			(1 + \frac{\lambda_n \gamma'(x_n)}{\gamma'(y_n)})\overline{v}_n\overline{u}_n^2 
			+ 
			t_n^\frac{2}{\beta}
			\frac	{\sigma_{2,n}(1+\alpha\tilde{c}_n)}
			{g(t_n^\frac{1}{\beta}z)\gamma'(x_n)}
			\mu_n(\phi_n(t_n^\frac{1}{\beta}z))\overline{v}_n.
		\end{cases}
	\end{equation}
	Moreover, the functions $\overline{u}_n, \overline{v}_n$ are $\beta$-H\"older, with constant
	$C_\beta$ in the set $B_{t_n^{-1/\beta}}(0)$, we have
	\begin{equation}
	    \label{normalizationAtZero}
	    \overline{u}_n(0)+\overline{v}_n(0) = 1,
	\end{equation}
Since $t_n \rightarrow 0$, ovserve that $B_1(0) \subset B_{t_n^{-1/\beta}}$. The uniform H\"older  bounds and the boundedness of the functions at $0$ by \eqref{normalizationAtZero} imply the existence of $C_\infty = 1 + C_\beta$ such that:
	\begin{equation*}
    	\label{LinfinityBoundAltCaffareli}
    	\|\overline{u}_n\|_{L^\infty(B_1(0))} \leq C_\infty,\qquad \|\overline{v}_n\|_{L^\infty(B_1(0))} \leq C_\infty,
	\end{equation*}
and by Ascoli-Arzel\'a's Theorem there exists $(\overline{u}_\infty, \overline{v}_\infty) \in C^0(B_1(0))$ such that, up to a subsequence, $(\overline{u}_n,\overline{v}_n) \rightarrow (\overline{u}_\infty,\overline{v}_\infty)$ in $C^0(B_1(0))$.
	Since $\overline{u}_\infty(0)+\overline{v}_\infty(0) = 1$, we may assume without loss of generality that $\overline{u}_\infty(0) \geq  \frac{1}{2}$. Then, there exists a $0 <\delta < 1$ and $\overline{n}$ large enough such that:
	$$\overline{u}_n(x)\geq \frac{1}{4} 
	\qquad 
	\forall n > \overline{n}, \ x \in B_{2\delta}(0).$$
Notice also that $\tilde{u}_n(0)\tilde{v}_n(0) \leq \tilde{u}_n^2(0)+\tilde{v}_n^2(0) + 2\tilde{u}_n(0) \tilde{v}_n(0) = t_n^2$, and so:
	$$k_n\tilde{u}_n^{1+\frac{1}{\beta}}(0)\tilde{v}_n^{1+\frac{1}{\beta}}(0) \leq k_n t_n^{2+\frac{2}{\beta}}.$$
	Since, by the contradiction hypothesis \eqref{blowUpContradiction}, we have   $k_n\tilde{u}_n^{1+\frac{1}{\beta}}(0)\tilde{v}_n^{1+\frac{1}{\beta}}(0) \rightarrow \infty$, then $k_nt_n^{2+\frac{2}{\beta}} \rightarrow \infty$.

	By using equation \eqref{overlineEqAltCaff}, that $k_nt_n^{2+\frac{2}{\beta}} \rightarrow \infty$ and $t_n^\frac{2}{\beta}
			\frac	{\sigma_{1,n}(1+\alpha\tilde{c}_n)}
			{g(t_n^\frac{1}{\beta}z)\gamma'(x_n)}
			\mu_n(\phi_n(t_n^\frac{1}{\beta}z))$ is bounded in $B_1$, we conclude there exists a constant $C>0$ such that:
	\begin{align}
		-
		\div_{\mathbb{R}^{N-1}}&
		(
		\overline{M}_n(z)
		\nabla_{\theta} \overline{v}_n
		) 
		= 
		\overline{v}_n\left(
	    -t_n^{2+\frac{2}{\beta}}
			\frac{k_n}{g(t_n^\frac{1}{\beta}z)}
			(1 + \frac{\lambda_n \gamma'(x_n)}{\gamma'(y_n)})\overline{u}_n^2 
			+
			t_n^\frac{2}{\beta}
			\frac	{\sigma_{2,n}(1+\alpha\tilde{c}_n)}
			{g(t_n^\frac{1}{\beta}z)\gamma'(x_n)}
			\mu_n(\phi_n(t_n^\frac{1}{\beta}z))
		\right) 
		\nonumber \\
		&\leq 
		\overline{v}_n\left(
		-t_n^{2+\frac{2}{\beta}}
		\frac{k_n}{5^{N-1}}
		(\frac{1}{4})^2 
		+ 
		\frac{t_n^{\frac{2}{\beta}}
			\sigma_{2,n}(1+\alpha \tilde c)
			(1+\tilde c)^\frac{1}{2}}{\gamma'(y_n)}
		\right) \leq 
		-Ct_n^{2+\frac{2}{\beta}}k_n\overline{v}_n \qquad \forall z \in B_{2\delta}.
		\label{overlineInequalityAltCaff}
	\end{align}
Using once again Lemma \ref{estimateLemma}-(2) and inequality \eqref{overlineInequalityAltCaff},  there exist constants $C_1$ and $C_2$ depending on the ellipticity constant $\tilde \theta$, and the bounds on the norms of $\overline{M}_n$ such that:
	\begin{equation}
	\label{overlineInequalityCoolLemaAltCaff}
	\overline{v}_n(0) \leq C_1C_\infty e^{-C_2\delta t_n^{1+\frac{1}{\beta}}k_n^{\frac{1}{2}}}.
	\end{equation}
	Multiplying inequality \eqref{overlineInequalityCoolLemaAltCaff} by $\overline{u}_n(0)$ and bounding it by $C_\infty$, we obtain:
	\begin{equation}
	    \label{overlineInequalityCoolLemaAltCaff2}
    	\overline{u}_n(0)\overline{v}_n(0) \leq C_1C_\infty^2 e^{-C_2\delta t_n^{1+\frac{1}{\beta}}k_n^{\frac{1}{2}}}.
	\end{equation}
	Raising both sides of \eqref{overlineInequalityCoolLemaAltCaff2} to the power $\frac{1}{2}+\frac{1}{2\beta}$ and taking $\tilde{C}_1 = C_1^{\frac{1}{2}+\frac{1}{2\beta}}C_\infty^{1+\frac{1}{\beta}}$ and $\tilde{C}_2 = C_2(\frac{1}{2}+\frac{1}{2\beta})$ we get:
	\begin{equation}
	\label{overlineInequalityCoolLemaAltCaff3}
	\overline{u}_n^{\frac{1}{2}+\frac{1}{2\beta}}(0)\overline{v}_n^{\frac{1}{2}+\frac{1}{2\beta}}(0) \leq \tilde{C}_1 e^{-\tilde{C}_2t_n^{1+\frac{1}{\beta}}k_n^{\frac{1}{2}}}.
	\end{equation}
	Multiplying \eqref{overlineInequalityCoolLemaAltCaff3} by $k_n^{\frac{1}{2}}t_n^{1+\frac{1}{\beta}}$ we get:
	\begin{equation}
	\label{finalIneqAltCaff5}
	k_n^{\frac{1}{2}}\tilde{u}_n^{\frac{1}{2}+\frac{1}{2\beta}}(0)\tilde{v}_n^{\frac{1}{2}+\frac{1}{2\beta}}(0)
	=
	k_n^\frac{1}{2}
	t_n^{1+\frac{1}{\beta}}\overline{u}_n^{\frac{1}{2}
		+
		\frac{1}{2\beta}}(0)\overline{v}_n^{\frac{1}{2}+\frac{1}{2\beta}}(0) 
	\leq 
	k_n^{\frac{1}{2}}t_n^{1+\frac{1}{\beta}}\tilde{C}_1 e^{-\tilde{C}_2\delta t_n^{1+\frac{1}{\beta}}k_n^{\frac{1}{2}}}
	\end{equation}
	Since 
	$
	t_n^{1+\frac{1}{\beta}}k_n^{\frac{1}{2}} \rightarrow \infty
	$ 
	we know that 
	$
	k_n^{\frac{1}{2}} t_n^{1+\frac{1}{\beta}}\tilde{C}_1 e^{-\tilde{C}_2\delta t_n^{1+\frac{1}{\beta}}k_n^{\frac{1}{2}}} \rightarrow 0
	$ 
	and, by the contradiction hypothesis \eqref{blowUpContradiction} we have 
	$
	k_n^{\frac{1}{2}}\tilde{u}_n^{\frac{1}{2}+\frac{1}{2\beta}}(0)\tilde{v}_n^{\frac{1}{2}+\frac{1}{2\beta}}(0) \rightarrow \infty
	$ 
	in contradiction with inequality \eqref{finalIneqAltCaff5}
	concluding the proof of this step. Thus, there exists $D = D(N, \tilde{c}, \tilde{\theta}, 
	\delta, \overline{\lambda}, \beta)>0$ such that 
	$
	u_nv_n \leq Dk_n^{-\frac{\beta}{\beta+1}}.
	$
	In particular, we have:
	\begin{equation*}
		\label{boundOnEquality}
		x \in \{v_n-u_n\leq 0\}
		\implies
		v_n(x)\leq D^\frac{1}{2}k_n^{-\frac{\beta}{2(\beta+1)}},\quad x \in \{v_n-u_n\geq 0\}
		\implies
		u_n(x)\leq D^\frac{1}{2}k_n^{-\frac{\beta}{2(\beta+1)}}.
	\end{equation*}
	
\noindent	\textbf{Step 3.} We claim that there exists $C = C(N, \overline{\lambda}, \tilde{c}, \delta, \beta, \alpha, \tilde{\theta}, \tilde{M})$ such that
	\begin{equation}
	\label{squareGradientEstimate}
	\int_{\{v_n> u_n\}}|\nabla_\theta u_n|^2d\sigma(y) \leq Ck_n^{-\frac{\beta}{2(\beta+1)}}, \qquad \int_{\{u_n> v_n\}}|\nabla_\theta v_n|^2
	d\sigma(y) 
	\leq 
	Ck_n^{-\frac{\beta}{2(\beta+1)}}
	\end{equation}
	and
	\begin{equation}
	\label{innerProductEstimate}
	\int_{\{u_n> v_n\}}
	\langle 
	    B_n(y) \nabla_\theta u_n,
	    \nabla_\theta v_n
	\rangle
	d\sigma(y)
	\leq 
	Ck_n^{-\frac{\beta}{2(\beta+1)}}
	,\qquad
	\int_{\{v_n> u_n\}}
	\langle 
	    B_n(y) \nabla_\theta u_n,
	    \nabla_\theta v_n
	\rangle
	d\sigma(y)
	\leq 
	Ck_n^{-\frac{\beta}{2(\beta+1)}}.
	\end{equation}

	\vspace{2mm}
	
	To show this, we fix from now on $n \in \mathbb{N}$, and we consider
	$
	\epsilon>0
	$ 
	such that if
	$
	u_n(x)-v_n(x) = \epsilon.
	$
	Then:
	$$
	v_n(x)\leq Ck_n^{-\frac{\beta}{2(\beta+1)}},\qquad
	u_n(x) \leq Ck_n^{-\frac{\beta}{2(\beta+1)}}.
	$$
	This is possible because, if $u_n(x)-v_n(x) = \epsilon$, then by Step 2 $v_n(x) \leq D^\frac{1}{2}k_n^{-\frac{\beta}{2(\beta+1)}}$ and
	$
	    u_n^2(x)-\epsilon u_n(x)
	\leq
	    Dk_n^{-\frac{\beta}{(\beta+1)}}.
	$
	Thus, by taking $\epsilon \leq \frac{Dk_n^{-\frac{\beta}{\beta+1}}}{2(1+\|u_n\|_{L^\infty(\partial B_1)})}$, we obtain $u_n(x) \leq 2^\frac{1}{2}D^\frac{1}{2}k_n^{-\frac{\beta}{2(\beta+1)}}$. Similarly, we obtain $v_n(x) \leq 2^\frac{1}{2}D^\frac{1}{2}k_n^{-\frac{\beta}{2(\beta+1)}}$ for $x \in \{u_n-v_n = \epsilon\}$. In particular, we conclude the following statements:
	\begin{align}
	    \label{setImplicationsAltCaff}
	    &x \in \{u_n-v_n\leq \epsilon\}
	    \Longrightarrow
	    u_n(x)
	    \leq 
	    Ck_n^{-\frac{\beta}{2(\beta+1)}},\qquad 
	    x \in \{u_n-v_n\geq\epsilon\}
	    \Longrightarrow
	    v_n(x)
	    \leq
	    Ck_n^{-\frac{\beta}{2(\beta+1)}}
	\end{align}

	By Morse-Sard's theorem (see for instance the version in \cite[Lemma 2.96]{AmbrosioFusco}) we can also suppose that 
	$
	\epsilon
	$ 
	is such that the set
	$
	\{u_n-v_n=\epsilon\}
	$ is an $N-2$ dimensional submanifold in the sphere $\partial B_1$. Now we integrate equation \eqref{SphereEq} for $u_n$ in the subset 
	$
	\{u_n-v_n \geq \epsilon\}
	$, and using the divergence theorem one obtains:
	\begin{align}
	-\int_{\{u_n-v_n=\epsilon \}}
	&\langle\tilde{B}_n(y)  \nabla_\theta u_n,  (\nu_{\epsilon}^1)_{y}\rangle_{\partial B_1}
	d\mathcal{H}^{N-2} \nonumber \\
&	=
	\int_{\{u_n-v_n\geq \epsilon\}}
	\left(
	-k_n(
	\lambda_n 
	+ 
	\frac{\gamma'(y_n)}{\gamma'(x_n)})
	u_nv_n^2 
	+ 
	\frac{\sigma_{1,n}(1+\alpha \tilde{c}_n)}{\gamma'(x_n)}\mu_n(y)
	u_n
	\right)
	d\sigma(y)	, \label{unEquationAltCaffSphere}
	\end{align}
	where 
	$
	(\nu_{\epsilon}^1)_y \in T_y \partial B_1
	$ 
	is the exterior normal to the set 
	$
	\{u_n-v_n\geq \epsilon\}
	$. Integrating equation \eqref{SphereEq} in all of $\partial B_1$, we obtain: 
	\begin{equation*}
	    \int_{\partial B_1}
	    k_n(\lambda_n + \frac{\gamma'(y_n)}{\gamma'(x_n)})u_nv_n^2 
	    d\sigma(y)
	=
	    \int_{\partial B_1}
	    \frac{\sigma_{1,n}(1+\alpha \tilde{c}_n)}{\gamma'(x_n)}\mu_n(y)
	    u_n
	    d\sigma(y).
	\end{equation*}
	Thus, by Step 1, we conclude that there exists $C>0$ such that:
	\begin{equation}
		\label{boundOnInteraction}
			\int_{\partial B_1}
			k_n v_n^2u_nd\sigma(y)
			\leq
			C\qquad
			\int_{\partial B_1}
			k_n u_n^2v_n
			d\sigma(y)
			\leq
			C\qquad \text{	for all $n \in \mathbb{N}$.}
	\end{equation}
	With \eqref{boundOnInteraction}, we conclude that the right-hand-side of \eqref{unEquationAltCaffSphere} is uniformly bounded in $n$ from above and below thus there exists $C>0$ such that:
	\begin{equation}
		\label{boundOnBoundaryIntegral1}
		\Big|
		\int_{\{u_n-v_n=\epsilon \}}
		\langle\tilde{B}_n(y)\nabla_\theta u_n,  (\nu_{\epsilon}^1)_y\rangle 
		d\mathcal{H}^{N-2}
		\Big| 
		\leq 
		C.
	\end{equation}
	We can do the same with $v_n$:
	\begin{equation}
		\label{boundOnBoundaryIntegral2}
		\Big|
		\int_{\{u_n-v_n=\epsilon\}}
		\langle \tilde{B}_n(y)\nabla_\theta v_n, (\nu_{\epsilon}^1)_y\rangle 
		d\mathcal{H}^{N-2}
		\Big|
		\leq 
		C.
	\end{equation}

	Now we multiply the equation \eqref{SphereEq} by $u_n$ and integrate in the set
	$
	\{u_n-v_n \leq \epsilon\}
	$. Then, by \eqref{setImplicationsAltCaff} for $x \in \{u_n-v_n\leq \epsilon\}$, we have $u_n(x) \leq Ck_n^{-\frac{\beta}{2(\beta+1)}}$,
	and so:
	\begin{multline*}
	\Big|
	\int_{\{u_n-v_n\leq \epsilon\}} 
	\langle \tilde{B}_n(y)\nabla_\theta u_n, \nabla_\theta u_n \rangle
	d\sigma(y)
	- 
	\int_{\{u_n-v_n=\epsilon\}} 
	\langle \tilde{B}_n(y) \nabla_\theta u_n,  (\nu_{\epsilon}^1)_y \rangle 
	u_n(y)
	d\mathcal{H}^{N-2}
	\Big| \\
	= \Big|
	\int_{\{u_n-v_n\leq \epsilon\}}
	\left(
    	-k_n(\lambda_n + \frac{\gamma'(y_n)}{\gamma'(x_n)})u_n^2v_n^2 + \frac{\sigma_{1,n}}{\gamma'(x_n)}u_n^2
	\right)
	d\sigma(y)
	\Big|\leq Ck_n^{-\frac{\beta}{2(\beta+1)}}
	\end{multline*}
	(since $-(\nu_{\epsilon}^1)$ is the exterior normal to $\{u_n-v_n \leq \epsilon\}$), the right hand side is bounded by 
	$
	Ck_n^{-\frac{\beta}{2(\beta+1)}}
	$ 
	by using equation \eqref{setImplicationsAltCaff} and \eqref{boundOnInteraction}. We can also do the same for $v_n$ by integrating in $\{u_n-v_n\geq \epsilon\}$ and we obtain the following bounds:
	$$
	\Big|
	\int_{\{u_n-v_n\leq \epsilon\}} 
	\langle \tilde{B}_n(y)\nabla_\theta u_n, \nabla_\theta u_n \rangle
	d\sigma(y)
	- 
	\int_{\{u_n-v_n=\epsilon\}} 
	\langle \tilde{B}_n(y) \nabla_\theta u_n,  (\nu_{\epsilon}^1)_y \rangle
	u_n 
	d\mathcal{H}^{N-2}
	\Big| 
	\leq 
	Ck_n^{-\frac{\beta}{2(\beta+1)}},
	$$
	$$
	\Big|
	\int_{\{u_n-v_n\geq \epsilon\}} 
	\langle \tilde{B}_n(y)\nabla_\theta v_n, \nabla_\theta v_n \rangle 
	d\sigma(y)
	+ 
	\int_{\{u_n-v_n=\epsilon\}} 
	\langle \tilde{B}_n(y) \nabla_\theta v_n,  (\nu_{\epsilon}^1)_y \rangle 
	v_n 
	d\mathcal{H}^{N-2}
	\Big| \leq Ck_n^{-\frac{\beta}{2(\beta+1)}}.
	$$
	The bound for $v_n$ has an inverted sign for the integral in $\{u_n-v_n=\epsilon\}$ since the exterior normal to $\{u_n-v_n\geq \epsilon\}$ is simply $-(\nu_{\epsilon}^1)$, because $(\nu_{\epsilon}^1)$ is the exterior normal to $\{u_n-v_n\leq \epsilon\}$.
	
	Summing up both equations we obtain that:
	\begin{align}
	&\Big|
	\int_{\{u_n-v_n\leq \epsilon\}} \langle \tilde{B}_n(y)\nabla_{\theta} u_n, \nabla_{\theta} u_n \rangle 
	d\sigma(y)
	+ \int_{\{u_n-v_n\geq \epsilon\}} \langle \tilde{B}_n(y)\nabla_{\theta} v_n, \nabla_{\theta} v_n \rangle 
    d\sigma(y)
	\nonumber \\
	&+ 
	\int_{\{u_n-v_n=\epsilon\}}
	\left(
	\langle \tilde{B}_n(y) (\nabla_{\theta} v_n - \nabla_{\theta} u_n),  (\nu_{\epsilon}^1)_y \rangle u_n - \epsilon\langle \nabla \tilde{B}_n(y)\nabla_\theta v_n,  (\nu_{\epsilon}^1)_y \rangle
	\right)
	d\mathcal{H}^{N-2}
	\Big| 
	\leq 
	Ck_n^{-\frac{\beta}{2(\beta+1)}}. \label{essentialEquationaltCaff3}
	\end{align}
	Now we make the observation that
	$$
	\nabla_{\theta}\left(
	v_n-u_n
	\right) 
	= 
	\big|\nabla_\theta \left(
	v_n-u_n
	\right)\big| 
	(\nu^1_{\epsilon})_y,
	$$
	since 
	$
	\nu_{\epsilon}^1
	$
	is the normal exterior to the level set
	$
	\{u_n-v_n\geq\epsilon\}
	$
	. We conclude:
	$$ 
	\langle \tilde{B}_n(y) \nabla_{\theta} (v_n - u_n),  (\nu_{\epsilon}^1)_y \rangle 
	= 
	\big|
	\nabla_\theta \big(v_n-u_n\big) 
	\big|
	\langle\tilde{B}_n(y) (\nu^1_{\epsilon})_y, (\nu^1_{\epsilon})_y \rangle 
	\geq 
	0.
	$$
	Thus, using the fact that the integrand below has a sign, we know that:
	\begin{align}
	&\Big|
	\int_{\{u_n-v_n=\epsilon\}} \langle \tilde{B}_n(y) (\nabla_{\theta} v_n - \nabla_{\theta} u_n), (\nu_{\epsilon}^1)_y \rangle u_nd\mathcal{H}^{N-2}
	\Big| 
\nonumber \\
	&\leq 
	\|u_n\|_{L^\infty(\{u_n-v_n=\epsilon\})}
	\Big|
	\int_{\{u_n-v_n=\epsilon\}} \langle \tilde{B}_n(y) (\nabla_{\theta} v_n - \nabla_{\theta} u_n),  (\nu_{\epsilon}^1)_y \rangle d\mathcal{H}^{N-2}
	\Big| 
	\leq 
	Ck_n^{-\frac{\beta}{2(\beta+1)}}
	\label{diferenceLineIntegral1}\end{align}
	where we have used  equations \eqref{setImplicationsAltCaff}, \eqref{boundOnBoundaryIntegral1} and \eqref{boundOnBoundaryIntegral2}.	Similarly, we have that:
	\begin{equation}
	\label{lineIntegral2}
	\Big|
	\int_{\{u_n-v_n=\epsilon\}} \epsilon\langle  \tilde{B}_n(y)\nabla_\theta v_n,(\nu_{\epsilon}^1)_y \rangle d\mathcal{H}^{N-2}
	\Big| 
	\leq 
	C\epsilon.
	\end{equation}
	From equations \eqref{diferenceLineIntegral1}, \eqref{lineIntegral2} and \eqref{essentialEquationaltCaff3} we obtain:
	$$
	\int_{\{u_n-v_n\geq \epsilon\}} \langle \tilde{B}_n(y)\nabla_{\theta} u_n, \nabla_{\theta} u_n \rangle d\sigma(y) + \int_{\{u_n-v_n\leq \epsilon\}} \langle \tilde{B}_n(y)\nabla_{\theta} v_n, \nabla_{\theta} v_n \rangle d\sigma(y)
	\leq 
	Ck_n^{-\frac{\beta}{2(\beta+1)}} + C\epsilon.
	$$
	By making $\epsilon$ go to zero we conclude that:
	$$
	\int_{\{u_n-v_n>0\}} \langle \tilde{B}_n(y)\nabla_{\theta}u_n, \nabla_{\theta}u_n  \rangle d\sigma(y)
	+ 
	\int_{\{u_n-v_n<0\}} \langle \tilde{B}_n(y)\nabla_{\theta}v_n, \nabla_{\theta}v_n\rangle d\sigma(y) 
	\leq 
	Ck_n^{-\frac{\beta}{2(\beta+1)}}.
	$$
	This shows \eqref{squareGradientEstimate}.
	
	Finally, by multiplying the equation for $v_n$ in \eqref{SphereEq} by $\min\{u_n,v_n\}$ and integrating in the entire sphere, we obtain an estimate of the form:
	$$
	    \Big|
	        \int_{\{u_n<v_n\}}
	        \langle 
    	        B_n(y)
    	        \nabla_\theta
    	        u_n,
    	        \nabla_\theta
    	        u_n
	        \rangle
	        d\sigma(y)
	    +
	        \int_{\{v_n\leq u_n\}}
	        \langle
	            B_n(y)
	            \nabla_\theta
	            u_n,
	            \nabla_\theta
	            v_n
	        \rangle
	        d\sigma(y)
	    \Big|
	\leq
	    Ck_n^{-\frac{\beta}{2(\beta+1)}},
	$$
	thus using the first estimate given by \eqref{squareGradientEstimate} we can find $C = C(N, \overline{\lambda}, \tilde{c}, \delta, \beta, \alpha, \tilde{\theta}, \tilde{M})$ such that
	$$
	\Big|
	    \int_{\{v_n\leq u_n\}}
	        \langle
	            B_n(y)
	            \nabla_\theta
	            u_n,
	            \nabla_\theta
	            v_n
	        \rangle
	        d\sigma(y)
	\Big|
	\leq
	    Ck_n^{-\frac{\beta}{2(\beta+1)}},
	$$
which is the first estimate in \eqref{innerProductEstimate}. To obtain the second one we proceed in an analogous way, this time multiplying the equation for $v_n$ in \eqref{SphereEq} by $\min\{u_n,v_n\}$.
	
\noindent	\textbf{Step 4. } Conclusion of the proof of the lemma. Take the functions 
	$
	f_n = (u_n-v_n)^+, g_n = (u_n-v_n)^- \in H^1(\partial B_1)
	$. Using the classical Friedman-Hayman inequality in the sphere \eqref{eq:FH_ineq}, we obtain:
	$$
	2
	\leq 
	\gamma\left(
	\frac{\int_{\partial B_1} |\nabla_\theta f_n|^2
	d\sigma(y)}
	{\int_{\partial B_1}f_n^2d\sigma(y)}
	\right) 
	+ 
	\gamma\left(
	\frac{\int_{\partial B_1} |\nabla_\theta g_n|^2
	d\sigma(y)}{\int_{\partial B_1} g_n^2d\sigma(y)}
	\right).
	$$
	We compute the $L^2(\partial B_1)$ norm of the gradient $\nabla_\theta f_n$: from $
	\sup_{y \in \partial B_1}\|\tilde{B}_n(y)|_{T_y \partial B_1}- Id|_{T_y \partial B_1}\| \leq \tilde{M}\tilde{c}_n
	$, the uniform boundedness of $f_n$, and the estimates 	\eqref{innerProductEstimate}, \eqref{squareGradientEstimate} proved in Step 3,
	\begin{align*}
	\int_{\partial B_1} &|\nabla_\theta f_n|^2 d\sigma(y)\leq \int_{\partial B_1}\langle\tilde{B}_n(y)\nabla_\theta f_n, \nabla_\theta f_n\rangle d\sigma(y)+	C\tilde{c}_n \\ 
	&=\int_{\{u_n>v_n\}} \left(\langle \tilde{B}_n(y)\nabla_\theta u_n, \nabla_\theta u_n \rangle + \langle \tilde{B}_n(y)\nabla_\theta v_n,\nabla_\theta v_n \rangle - 2\langle \tilde{B}_n(y)\nabla_\theta u_n, \nabla_\theta v_n \rangle\right)d\sigma(y) +	C\tilde{c}_n\\
&\leq   \int_{\partial B_1}\langle \tilde{B}_n(y)\nabla_\theta u_n, \nabla_\theta u_n \rangle  d\sigma(y)	+ C\tilde{c}_n+ Ck_n^{-\frac{\beta}{2(\beta+1)}} \\
&\leq 	\left(\int_{\partial B_1}\left(\langle \tilde{B}_n(y)\nabla_\theta u_n, \nabla_\theta u_n \rangle + k_nu_n^2v_n^2\right)d\sigma(y)- \epsilon_n\right) + C \left(\epsilon_n+k_n^{-\frac{\beta}{2(\beta+1)}} + \tilde{c}_n\right)\\
&\leq x_n 	+  C\left(\epsilon_n+k_n^{-\frac{\beta}{2(\beta+1)}} + \tilde{c}_n\right),
\end{align*}
where we recall from \eqref{def_of_x_n} the definition of $x_n$. 	On the other hand, by \label{GammaMinimization_norm},
	\begin{align*}
	1&=	\int_{\partial B_1}(1+\alpha \tilde{c}_n)\mu_n(y)u_n^2d\sigma(y) \leq \int_{\partial B_1} u_n^2+C\tilde c_n = \int_{\{u_n>v_n\}}u_n^2+\int_{\{v_n>u_n\}}u_n^2+C\tilde c_n\\
	   &=\int_{\partial B_1} f_n^2 + \int_{\{u_n>v_n\}} (2u_nv_n-v_n^2) + \int_{\{v_n>u_n\}}u_n^2+C\tilde c_n \leq \int_{\partial B_1} f_n^2 + C \tilde c_n+ Ck_n^{-\frac{\beta}{\beta+1}},
	\end{align*}
where we used estimates from Step 2. Then
	$$
	    \frac
	    {\int_{\partial B_1}
	        |\nabla_\theta f_n|^2
	        d\sigma(y)
	    }
	    {\int_{\partial B_1}
	        |f_n|^2
	        d\sigma(y)
	    }
	   \leq
	   \frac{x_n + C(\epsilon_n + \tilde{c}_n + k_n^\frac{-\beta}{2(\beta+1)})}{1-C\tilde{c}_n-Ck_n^{-\frac{\beta}{\beta+1}}}
	   \leq 
	   x_n
	   +
	   C'
	   (\epsilon_n + \tilde{c}_n + k_n^\frac{-\beta}{2(\beta+1)}).
	$$
Using moreover  the monotonicity and concavity of $\gamma$, we have:
	$$
	2
	\leq 
	\gamma\left(
	\frac{\int_{\partial B_1} |\nabla_\theta f_n|^2
	d\sigma(y)}
	{\int_{\partial B_1}f_n^2
	d\sigma(y)}
	\right) 
	+ 
	\gamma\left(
	\frac{\int_{\partial B_1} |\nabla_\theta g_n|^2
	d\sigma(y)}{\int_{\partial B_1} g_n^2d\sigma(y)}
	\right) 
	\leq 
	\gamma(x_n) 
	+ 
	\gamma(y_n) 
	+ 
	C\left(
	\epsilon_n
	+
	k_n^{-\frac{\beta}{2(\beta+1)}} 
	+ 
	\tilde{c}_n
	\right).$$
	Thus, we obtain the desired bound:
	$$
	2-C(\epsilon_n+k_n^{-\frac{\beta}{2(\beta+1)}} + \tilde{c}_n) 
	\leq  
	\gamma(x_n) + \gamma(y_n).$$
Observing that an arbitrary choice of $\beta \in ]0,1[$ yields an arbitrary choice of $\frac{\beta}{2(\beta+1)} \in ]0,\frac{1}{4}[$, we conclude the proof of the lemma.
\end{proof}

We are now able to prove the Alt-Caffarelli-Friedman- type monotonicity formula.

\begin{proof}[Proof of Theorem \ref{AltCaffMonotonicity}]
	We start by computing the derivative of 
	$
	\log\left(
	\frac{J_{1,n}(r)J_{2,n}(r)}{r^4}
	\right)
	$:
\begin{multline*}
	\frac{d}{dr}\log\left(
	\frac{J_{1,n}(r)J_{2,n}(r)}{r^4}
	\right)
	=
	-\frac{4}{r}\\
	+ 
	\frac{\int_{\partial B_r}\left(
		\langle A_n(y) \nabla u_{1,n}, \nabla u_{1,n}  \rangle
		-
		M_na_n(y)|u_{1,n}|^2|u_{2,n}|^2
		-
		u_{1,n}f_{1,n}(y,u_{1,n})|y|^{2-N}d\sigma(y)
		\right) }
	{\int_{B_r}\left(
		\langle A_n(y) \nabla u_{1,n}, \nabla u_{1,n} \rangle
		-
		M_na_n(y)|u_{1,n}|^2|u_{2,n}|^2
		-
		u_{1,n}f_{1,n}(y,u_{1,n})|y|^{2-N}dy
		\right)}\\
	+ 
	\frac{\int_{\partial B_r}\left(
		\langle A_n(y) \nabla u_{2,n}, \nabla u_{1,n}  \rangle
		-
		M_na_n(y)|u_{2,n}|^2|u_{1,n}|^2
		-
		u_{2,n}f_{2,n}(y,u_{2,n})|y|^{2-N}d\sigma(y)
		\right) }
	{\int_{B_r}\left(
		\langle A_n(y) \nabla u_{2,n}, \nabla u_{2,n}  \rangle
		-
		M_na_n(y)|u_{2,n}|^2|u_{1,n}|^2
		-
		u_{2,n}f_{2,n}(y,u_{2,n})|y|^{2-N}dy
		\right)}.
\end{multline*}
	Since the hypothesis for Lemma \ref{derivativeLemma} are satisfied, we obtain:
	\begin{equation}
	\label{OkEquationCaff}
	\frac{d}{dr}\log\left(
	\frac{J_{1,n}(r)J_{2,n}(r)}{r^4}
	\right)
	\geq
	-\frac{2}{r}\left(
	2
	-
	\gamma\left(
	\Lambda_{1,n}(r)
	\right)
	-
	\gamma\left(
	\Lambda_{2,n}(r)
	\right)
	\right)
	\end{equation}
Recalling the definition of the operator $B_n$ from Definition \ref{definition:operatorB} and the equivalent formulation of $\Lambda_{1,n}$ written in \eqref{eq:equivalent_LAMBDA}, 
	we can use the fact that $M_n < 0$, hypothesis ($h_2$) and the monotonicity of $\gamma$ to conclude:
	$$
	\gamma(\Lambda_{1,n}(r)) 
	\geq 
	\gamma\left(
	\frac{
		r^2\int_{\partial B_r}
		\left(
		\langle B_n(y) \nabla_{\theta} u_{1,n}, \nabla_{\theta} u_{1,n} \rangle -
		M_na_n(y)|u_{1,n}|^{2} |u_{2,n}|^2
		-
		(1+\alpha(rc_n))\epsilon_n \mu_n(y)u_{1,n}^2 
		\right)}
	{
		\int_{\partial B_r}
		(1+\alpha(rc_n))\mu_n(y)u_{1,n}^2
	}
	\right).
	$$

In order to rewrite the above as integrals over $\partial B_1$, we consider the change of variables given by:
	\begin{equation}
	\label{changeOfVarAltCaff}
	u_{i,n,r}(z) 
	= 
	\frac{u_{i,n}(rz)}
	{\sqrt{\frac{1}{r^{N-1}}\int _{\partial B_r}(1+\alpha(rc_n))\mu_n(y)u_{1,n}^2}}
	\end{equation}
and denote, for convenience,
	$$
	d_{n,r} = \frac{1}{r^{N-1}}\int_{\partial B_r}(1+\alpha(rc_n))\mu_n(y) u_{1,n}^2d\sigma(y)\quad \text{ and }\quad	m_{n,r} = (1+\alpha(rc_n)).
	$$ 
Therefore,
	\begin{multline*}
	\frac{
		r^2\int_{\partial B_r} \langle B_n(y) \nabla_{\theta} u_{1,n}, \nabla_{\theta} u_{1,n} \rangle - M_na_n(y)|u_{2,n}|^{2} |u_{1,n}|^2-\epsilon_nm_{n,r} \mu_n(y)u_{1,n}^2 d}
	{\int_{\partial B_r}
		m_{n,r}\mu_n(y)u_{1,n}^2}\\
	=
	{\int_{\partial B_1} \left(
		\langle B_n(rz) \nabla_\theta u_{1,n,r}, \nabla_{\theta} u_{1,n,r}  \rangle 
		-    
		d_{n,r}M_nr^2a_n(rz)|u_{2,n,r}|^2|u_{1,n,r}|^2
		-
		r^2\epsilon_nm_{n,r}\mu_n(rz)u_{1,n,r}^2
		\right)}
\end{multline*}
	and
\begin{multline*}
	\frac{
		r^2\int_{\partial B_r} \langle B_n(y) \nabla_{\theta} u_{2,n}, \nabla_{\theta} u_{2,n} \rangle - M_na_n(y)|u_{2,n}|^{2} |u_{1,n}|^2
		-
		\epsilon_nm_{n,r} \mu_n(y)u_{2,n}^2 }
	{
		\int_{\partial B_r}
		m_{n,r}\mu_n(y)u_{2,n}^2
	}\\
	=
	\frac{
		{\int_{\partial B_1} \left(
			\langle B_n(rz) \nabla_\theta u_{2,n,r}, \nabla_{\theta} u_{2,n,r}  \rangle 
			-    
			d_{n,r}M_nr^2a_n(rz)|u_{2,n,r}|^2|u_{1,n,r}|^2
			-
			r^2\epsilon_nm_{n,r}\mu_n(rz)u_{2,n,r}^2
			\right)}
	}
	{
		\int_{\partial B_1}
		m_{n,r}\mu_n(y)u_{2,n,r}^2
	}.
\end{multline*}
Thus, by hypotheses ($h_0$), ($h_1$), and Lemma \ref{MatrixBoundsOnB}, we know that:
	$$
	\sup_{y \in \partial B_1}\|B_n(ry)|_{T_y \partial B_1}-Id_{T_y \partial B_1}\| 
	\leq 
	Cc_nr,
	$$
while by $(h_3)$  we know that there exists $w>0$ such that:
	\begin{equation*}
	\label{inequalityDnAltCaff}
	d_{n,r} = \frac{1}{r^{N-1}}\int_{\partial B_r}(1+\alpha c_nr)\mu_n(y)u_{1,n}^2 
	\geq 
	w.
	\end{equation*}
Combining this with the monotonicity of $\gamma$, and since $M_n<0$, we conclude:
\begin{multline*}
	\gamma\left(
	\int_{\partial B_1} \left(
	\langle B_n(rz) \nabla_\theta u_{1,n,r}, \nabla_{\theta} u_{1,n,r}  \rangle 
	-    
	d_{n,r}M_nr^2a_n(rz)|u_{2,n,r}|^2|u_{1,n,r}|^2
	-
	r^2
	\epsilon_nm_{n,r}\mu_n(rz)u_{1,n,r}^2
	\right)
	\right)\\
	+\gamma\left(
	\frac{
		{\int_{\partial B_1} \left(
			\langle B_n(rz) \nabla_\theta u_{2,n,r}, \nabla_{\theta} u_{2,n,r}  \rangle 
			-    
			d_{n,r}M_nr^2a_n(rz)|u_{2,n,r}|^2|u_{1,n,r}|^2
			-
			r^2
			\epsilon_nm_{n,r}\mu_n(rz)u_{2,n,r}^2
			\right)}
	}
	{
		\int_{\partial B_1}
		m_{n,r}\mu_n(rz)u_{2,n,r}^2
	}
	\right)\\
	\geq
	\gamma\left(
	\int_{\partial B_1} \left(
	\langle B_n(rz) \nabla_\theta u_{1,n,r}, \nabla_{\theta} u_{1,n,r}  \rangle 
	-    
	wM_nr^2a_n(rz)|u_{2,n,r}|^2|u_{1,n,r}|^2
	-
	r^2
	\epsilon_nm_{n,r}\mu_n(rz)u_{1,n,r}^2
	\right)
	\right)\\
	+\gamma\left(
	\frac{
		{\int_{\partial B_1} \left(
			\langle B_n(rz) \nabla_\theta u_{2,n,r}, \nabla_{\theta} u_{2,n,r}  \rangle 
			-    
			wM_nr^2a_n(rz)|u_{2,n,r}|^2|u_{1,n,r}|^2
			-
			r^2
			\epsilon_nm_{n,r}\mu_n(rz)u_{2,n,r}^2
			\right)}
	}
	{
		\int_{\partial B_1}
		m_{n,r}\mu_n(rz)u_{2,n,r}^2
	}
	\right).
\end{multline*}
	By the change of variables \eqref{changeOfVarAltCaff}, we have:
	$$
	\int_{\partial B_1} m_{n,r}\mu_n(rz)u_{1,n,r}d\sigma(z) 
	= 
	1,
	$$
	and, by \eqref{changeOfVarAltCaff} and hypothesis $(h_3)$:
	$$
	\frac{1}{\lambda}
	\leq
	\int_{\partial B_1} m_{n,r}\mu_n(rz)u_{2,n,r}^2d\sigma(z)
	=
	\frac
	{
	    \int_{\partial B_r}
	    \mu_n(y)u_{2,n}^2d\sigma(y)
	}
	{
	    \int_{\partial B_r}
	    \mu_n(y)u_{1,n}^2d\sigma(y)
	}
	\leq
	\lambda.
	$$
	
	We can now apply Lemma \ref{sphereLemma} with $\tilde B_n(z) = B_n(rz)$, $\tilde{c}_n = c_nr$, $\lambda_n = \int_{\partial B_1} m_{n,r}\mu_n(rz)u_{2,n,r}^2d\sigma(z)$, $\tilde{\epsilon}_n = \epsilon_n r^2<\epsilon_nR_n<(\frac{N-2}{2})^2-\delta$ and $k_n = wM_nr^2\min(a_n)$, concluding the existence of $C>0$ such that, for $\eta \in ] 0, \frac{1}{4}[$:
\begin{multline*}
\gamma(\Lambda_{1,n}(r)) + \gamma(\Lambda_{2,n}(r)) \nonumber\\
	\geq \gamma\left(
	\int_{\partial B_1} \left(
	\langle B_n(rz) \nabla_\theta u_{1,n,r}, \nabla_{\theta} u_{1,n,r}  \rangle 
	-    
	wM_nr^2a_n(rz)|u_{2,n,r}|^2|u_{1,n,r}|^2
	-   
	r^2
	\epsilon_n m_{n,r}\mu_n(rz)u_{1,n,r}^2
	\right)
	\right) \nonumber\\
	+\gamma\left(
	\frac{
		\int_{\partial B_1} \left(
		\langle B_n(rz) \nabla_\theta u_{2,n,r}, \nabla_{\theta} u_{2,n,r}  \rangle 
		-    
		wM_nr^2a_n(rz)|u_{2,n,r}|^2|u_{1,n,r}|^2
		-   
		r^2\epsilon_nm_{n,r}\mu_n(rz)u_{2,n,r}^2
		\right)}
	{
		\int_{\partial B_1}m_{n,r}\mu_n(rz) u_{2,n,r}
	}
	\right) \nonumber\\ 
	\geq 
	2
	-
	C\left(             
	-
	|M_n|^{-\eta}
	r^{-2\eta} 
	- 
	\epsilon_nr^2 
	- 
	c_nr
	\right)	
\end{multline*}
	
Thus, combining this inequality with \eqref{OkEquationCaff}, we see that:
	\begin{align*}
		\frac{d}{dr}\left(\log\left(  
		\frac{J_{1,n}(r)J_{2,n}(r)}{r^4}
		\right)
		\right)
		&\geq 
		-
		\frac{2}{r}\left(2
		-
		C\left(             
		-
		M_n^{-\eta}
		r^{-2\eta} 
		- 
		\epsilon_nr^2 
		- 
		c_nr
		\right)
		-
		2
		\right) \\
		&= 
		-C\left(
		|M_n|^{-\eta}
		r^{-2\eta-1} - \epsilon_nr - c_n
		\right),
	\end{align*}
and the proof is finished.	
\end{proof}

%%%%%%%%%%%%%

\begin{remark}\label{rem_finalSection4} In case $\gamma>1$, a result like this also holds true. The only necessary changes are in the definitions of $J_{i,n}$ and $\Lambda_{i,n}$, where the terms $M_n a_n(y)  |u_{1,n}|^2|u_{2,n}|^2$ should be replaced by $M_n a_n(y)  |u_{1,n}|^{\gamma+1}|u_{2,n}|^{\gamma+1}$, and in the proof of Lemma \ref{sphereLemma}, whenever Lemma \ref{estimateLemma}-(2) is used, one should use instead Lemma  \ref{estimateLemma}-(1).
\end{remark}
%%%%%%%%%%%%%%%%%%%

\begin{remark}\label{rem:dimension} For this section we will consider the dimension $N\geq 3$, since we need it for the classical Alt-Caffarelli-Friedman formula. Thus, to obtain Theorem \ref{DesiredTheorem} for dimensions $N\leq 2$, if $u_\beta(x)$, $x\in \mathbb{R}^N$, is a solution to the system \eqref{equation}, we consider the new vector solution $\tilde{u}_\beta(x,y) =u_\beta(x) $ with $x\in \mathbb{R}^N$, $y\in \mathbb{R}^{3-N}$, obtaining a system in  dimension $N=3$. This new system will still be of type \eqref{equation}, and one can apply the Alt-Caffarelli-Friedman type formula.
\end{remark}

\section{Interior Lipschitz bounds}\label{sec:Lip}
 In this section, we conclude the proof of Theorem \ref{DesiredTheorem}. As observed in  Remark \ref{rem:dimension}, we just need to consider the case $N\geq 3$ (in particular, the results of the previous sections are true). In Section \ref{chapter:background}, under the contradiction assumption that $\{u_{\beta_n}\}$ was not uniformly Lipschitz, we introduced in \eqref{blowUp} a blowup sequence $\{v_n\}$, solution to \eqref{eq:system_rescaled}. This sequence was defined in such a way that it has bounded Lipschitz-seminorm, it concentrates at a point where the gradient blowsup, and solves a system where the differential operator is a perturbation of the Laplacian close to the blowup point. In that section we arrived at Proposition \ref{concludeProp}: this sequence converges locally uniformly to a limiting profile $v=(v_1,\ldots, v_l)$, of which at most $v_1$ and $v_2$ are nontrivial. Much more information is necessary to arrive at a contradiction. In the previous two sections, we proved Almgren and Alt-Caffareli-Friedman-type monotonicity formulas; we will now apply them to the sequence $\{v_n\}$ to achieve the desired contradiction. Here we follow the structure of \cite[Section 4]{SoaveZilio}, with the necessary modifications that arise from the fact that we are dealing with a system with divergence type operators with variable coefficients.

We recall the sequence of functions $\tilde{u}_{i,\beta_n}>0$ defined in \eqref{tildeU} by:
$$
    \tilde{u}_{i,\beta_{n}}(x) = u_{i,\beta_n}(x_n+A(x_n)^{\frac{1}{2}}x).
$$
This sequence, by Lemma \ref{lemmaForMatrixTilde}, satisfies 
\begin{equation}\label{eq:systems_for_tildeu_rep}
-\div(\tilde{A}_n(x)\nabla \tilde{u}_{i,\beta_n})
	=
		f_i(x_n+A(x_n)^{\frac{1}{2}}x,\tilde{u}_{i,\beta_n}) 
	+ 
		a(x_n+A(x_n)^{\frac{1}{2}}x)
		\mathop{\sum_{j=1}^l}_{j\neq i}\beta_n 
		|\tilde{u}_{j, \beta_n}|^{\gamma + 1}
		|\tilde{u}_{i, \beta_n}|^{\gamma -1}\tilde{u}_{i, \beta_n},
\end{equation}
where $\tilde{A}_n(x) = A(x_n)^{-\frac{1}{2}}A(x_n +A(x_n)^{\frac{1}{2}}x)A(x_n)^{-\frac{1}{2}}$ is such that $\tilde{A}_n(0)=Id$.  By hypotheses \eqref{boundedness} and \eqref{boundForF}, there exist $m>0$ and $d>0$ such that:
$$
\max_{i=1,...,l}\|\tilde{u}_{i,\beta_n}\|_{L^\infty(B_{1/M^{\frac{1}{2}}})}<m,\qquad
\max_{i=1,...,l}\sup_{y\in[0,m]}f_i(x_n+A(x_n)^\frac{1}{2}x , y)\leq d |y|
$$
for all $n \in \mathbb{N}$, where $M$ is as in \textbf{(A2)}.
Also by Lemma \ref{lemmaForMatrixTilde} there exists $C>0$ such that for all $n \in \mathbb{N}$ we have: 
$$
    \langle 
        \tilde{A}_n(x)\xi,
        \xi
    \rangle
\geq
    \frac{\theta}{M}|\xi|^2,
\qquad
    \|D\tilde{A}_n\|_{L^\infty(B_{1/M^{\frac{1}{2}}})}
\leq
    C,
\qquad
    \|\tilde{A}_n\|_{L^\infty(B_{1/M^{\frac{1}{2}}})}
\leq
    C.
$$
Moreover, at the point $0$ we have $\tilde{A}(0) = Id$. We also define $\tilde{\mu}_n(y) = \langle \tilde{A}_n(y)\frac{y}{|y|},\frac{y}{|y|} \rangle$. 

We may, therefore, apply all the results of Section \ref{chapter:implementation} to the sequence $\{\tilde{u}_{\beta_n}\}$. In particular, Lemma \ref{derivOfH} and Theorem \ref{almgrenMonotonicity} imply the following.

\begin{prop}
	\label{MonotonicityForu} Let $\gamma \geq 1$ be such that $\frac{\gamma N}{\gamma+1}<2$.
	Then there exists $\tilde{r}$ and $\tilde{C}>0$ such that, for every $n \in \mathbb{N}$, the functions:
	$$(N_{\beta_n}(\tilde{u}_{\beta_n},r) + 1)e^{\tilde{C}r}
	\quad
	\text{ and }
	\quad
	H_{i,\beta_n}(\tilde{u}_{\beta_n},r)e^{\tilde{C}e}=
	\frac{1}{r^{N-1}}
	\int_{\partial B_r}
	\tilde{\mu}_n(y)
	\tilde{u}_{i,\beta_n}^2
	d\sigma(y)
	e^{\tilde{C}r}
	$$
	are monotone nondecreasing for $r \in ]0,\tilde{r}[$ and all $i \in \{1,...,l\}$. We recall that $N_{\beta_n}$ is defined in \eqref{Neq}.
\end{prop}
To ease notation, from now on in this section we omit the lower index $\beta_n$ in the functions $N_{\beta_n}(\tilde{u}_{\beta_n},r)$.

Now we introduce the quantity given by:
\begin{equation}
\label{BigRdefinition}
    R_{\beta_n} 
:= 
    \sup
    \left\{r \in ]0,\tilde{r}[: (N(\tilde{u}_{\beta_n},r) + 1)e^{\tilde{C}r}<2-r
    \right\}.
\end{equation}
\begin{lemma}
We have $R_{\beta_n}>0$ for all $n \in \mathbb{N}$.
\end{lemma}
\begin{proof}
    Fix $n \in \mathbb{N}$.
    Since $\tilde{u}_{i,\beta_n}$ is positive and of class $C^1$, there exist $\delta,\epsilon,C>0$ such that $\delta <u_{i,\beta_n}(x)<m$ and  $|\nabla \tilde{u}_{i,\beta_n}(x)| < C$ whenever  $|x|<\epsilon$.
    
    With this, for $r < \epsilon$ we conclude:
    \begin{align*}
     |N(\tilde{u}_{\beta_n},r)|
    =\left|
        \frac
        {\frac{1}{r^{N-2}}
		\sum_ {i=1}^l
		\int_{\partial B_r} \tilde{u}_{i, \beta_n}\langle  \tilde{A}_n(x) \nabla \tilde{u}_{i,\beta_n}, \nu_x\rangle d\sigma(x)
        }{
         \frac{1}{r^{N-1}}\sum_{i=1}^l\int_{\partial B_r}\tilde{\mu}_n(x)|\tilde{u}_{i,\beta_n}|^2d\sigma(x)
        } 
    \right|
    \leq
    \left(\frac{m MC l}{\delta^2 \tilde{\theta}}\right)r,
    \end{align*}
    and so $N(\tilde{u}_{\beta_n},r) \rightarrow 0$ as $r \rightarrow 0$; this implies that  $R_{\beta_n}>0$.
\end{proof}

\begin{lemma}
	\label{radGoToZero}
	$R_{\beta_n} \rightarrow 0$ as $n \rightarrow \infty$.
\end{lemma}
\begin{proof}
Recall from Lemma \ref{remarkOnStuff} that (up to a subsequence) $\tilde{u}_{\beta_n} \rightarrow \tilde{u}_\infty$ in $C^{0,\alpha}(B_{1/(2M^\frac{1}{2})})\cap H^1(B_{1/(2M^\frac{1}{2})})$ for every $\alpha\in (0,1)$, and that $\tilde{u}_\infty(0) = 0$. Moreover, by using Lemma \ref{lemmaForMatrixTilde} and Ascoli Arzela's Theorem, there exists $\tilde{A}(\cdot) \in C^1(B_{1/(2M^\frac{1}{2})},\text{Sym}^{N\times N})$ such that $\tilde{A}_n(x) \rightarrow \tilde{A}(x)$
	uniformly for $x\in B_{1/(2M^\frac{1}{2})}$.

	Let $\tilde{\mu}(x) = \langle \tilde{A}(x)\frac{x}{|x|}, \frac{x}{|x|} \rangle$, and consider the Almgren's quotient associated to \eqref{eq:systems_for_tildeu_rep}, namely
	\begin{align*}
	E(\tilde{u}_\infty,r) 
	&=
	\frac{1}{r^{N-2}}\sum_{i=1}^l
	\int_{B_r}
	\left(
	\langle \tilde{A}(x)\nabla \tilde{u}_{i,\infty}, \nabla \tilde{u}_{i,\infty} \rangle
	-
	f_i(x_\infty+A(x_\infty)^\frac{1}{2}x,\tilde{u}_{i,\infty})\tilde{u}_{i,\infty}
	\right)dx,\\
	H(\tilde{u}_\infty,r)
	&=
	\frac{1}{r^{N-1}}\sum_{i=1}^l\int_{\partial B_r} 
	\tilde{u}_{i,\infty}^2\tilde{\mu}(x)d\sigma(x),\qquad
	N(\tilde{u}_\infty,r)
	=
	\frac{E(\tilde{u}_\infty,r)}{H(\tilde{u}_\infty,r)}.
	\end{align*}
	We divide the rest of the proof in two steps.
	
\noindent	\textbf{Step 1.} $\lim_{r \rightarrow 0^+}N(\tilde{u}_\infty,r) \geq 1$.
		
	To prove this step, we notice that, because of the convergence of $\tilde{u}_{i,\beta_n} \rightarrow \tilde{u}_\infty$ in the spaces $C(B_{M^{-\frac{1}{2}}})\cap H^1(B_{M^{-\frac{1}{2}}})$, we have that, for each $r \in ]0,\tilde{r}[$: 
	\begin{equation}
	\label{EandHlimits}
	\lim_n
	E(\tilde{u}_{\beta_n},r)
	=
	E(\tilde{u}_{\infty},r),
	\qquad
	\lim_n
	H(\tilde{u}_{\beta_n},r) = H(\tilde{u}_\infty,r).
	\end{equation}
	Notice that the convergence of $E(\tilde{u}_{\beta_n},r)$ also comes from the fact that $
	\beta_n
        \sum_{j\neq i}
	\int_{B_r}
	\tilde{a}_n(x)
	|
	\tilde{u}_{i,\beta_n}
	|^2
	|
	\tilde{u}_{j,\beta_n}
	|^2
	dx
	\rightarrow 0
	$ (see for instance \cite[Theorem 1.4]{uniformHolderBoundsHugoTerraciniNoris}).

A direct computation (see for instance  \cite[Lemma C.5]{HugoHolderVariable} for the details) yields
	\begin{align*}
		H'(\tilde{u}_{\beta_n},r) &= \frac{1-N}{r}H(\tilde{u}_{\beta_n},r) + \frac{2}{r}E(\tilde{u}_{\beta_n},r) + 
		\sum_{i=1}^l
		\frac{1}{r^{N-1}}\int_{\partial B_r} \tilde{u}_{i,\beta_n}^2\div(\tilde{A}_n(x)\nabla|x|)
		d\sigma(x)
	\end{align*}
	and so, passing to the limit and using \eqref{EandHlimits}:
	$$
	H'(\tilde{u}_{\beta_n},r) \rightarrow
	\frac{1-N}{r}H(\tilde{u}_{\infty},r) + \frac{2}{r}E(\tilde{u}_{\infty},r) +
	\sum_{i=1}^l
	\frac{1}{r^{N-1}}\int_{\partial B_r} \tilde{u}_{i,\infty}^2\div(\tilde{A}(x)\nabla|x|)d\sigma(x).
	$$
On the other hand:
	\begin{align*}
	    H(\tilde{u}_\infty,s_2)
	    -
	    H(\tilde{u}_\infty, s_1)
	&=
	    \lim_n
	    H(\tilde{u}_{\beta_n},s_2)
	    -
	    H(\tilde{u}_{\beta_n}, s_1)
	=
	    \lim_n
	    \int_{s_1}^{s_2}
	    H'(\tilde{u}_{\beta_n},r)
	    dr\\
	&=
	    \int_{s_1}^{s_2}
	    \left(
    	    \frac{1-N}{r}H(\tilde{u}_{\infty},r) + \frac{2}{r}E(\tilde{u}_{\infty},r) +
        	\sum_{i=1}^l
        	\frac{1}{r^{N-1}}
        	\int_{\partial B_r} \tilde{u}_{\infty,i}^2\div(\tilde{A}(x)\nabla|x|)
        	d\sigma(x)
    	\right),	
     \end{align*}
    and from this we conclude that $\lim_n H'(\tilde{u}_{\beta_n},r) \to H'(\tilde{u}_\infty,r)$. Thus, applying these convergences, by passing to the limit the results in Lemma \ref{derivOfH}, we have that there exists $C>0$ such that:
	\begin{equation}
	\label{derivHBound}
	\Big|
	H'(\tilde{u}_\infty,r)-\frac{2}{r}E(\tilde{u}_\infty,r)
	\Big|
	\leq
	CH(\tilde{u}_\infty,r).
	\end{equation}
	Notice also that $\Big(N(\tilde{u}_{\beta_n},r) + 1\Big) e^{\tilde{C}r}$ is monotone nondecreasing, thus so is $(N(\tilde{u}_\infty,r)+1)e^{\tilde{C}r}$.
	
We now suppose by contradiction that $\lim_{r \rightarrow 0^+} N(\tilde{u}_{\infty}, r) < 1$. Then there exists $\delta>0$ such that
	\begin{equation}
	\label{contradictionBoundOnN}
	N(\tilde{u}_{\infty}, r)
	<
	1-\delta
	\qquad
	\forall r \in ]0,\overline{r}[.
	\end{equation}
	Now by using equation \eqref{derivHBound} and \eqref{contradictionBoundOnN}, for $r < \overline{r}$, we have:
	\begin{align}
	\frac{d}{dr}\log(H(\tilde{u}_\infty,r))
	&=
	\frac{H'(\tilde{u}_\infty,r)}
	{H(\tilde{u}_\infty,r)}
	\leq
	\frac{2}{r}
	\frac{E(\tilde{u}_\infty, r)}
	{H(\tilde{u}_\infty, r)}
+
    C
	=
	\frac{2N(\tilde{u}_\infty, r)}{r}
	+
	C
	\leq
	\frac{2(1-\delta)}{r} + C.
	\label{difEqOfHinfinity}
	\end{align}
	Integrating \eqref{difEqOfHinfinity} from $r \in ]0,\overline{r}[$ up to $\overline{r}$, we obtain:
	$$
	\frac{
		H(\tilde{u}_\infty,\overline{r})
	}{
		H(\tilde{u}_\infty,r)
	}
	\leq
	\Big(\frac{\overline{r}}{r}\Big)^{2(1-\delta)}
	e^{C(\overline{r}-r)},
	$$
	thus we conclude that there exists $c := \frac{H(\tilde{u}_\infty,\overline{r})}{\overline{r}^{2(1-\delta)}e^{C\overline{r}}}>0$ such that:
	$$
	cr^{2(1-\delta)}< H(\tilde{u}_\infty,r)
	\qquad 
	\forall r \in ]0,\overline{r}[.
    $$
On the other hand, since $\tilde{u}_\infty$ is bounded in $C^{0,\alpha}(B_{\overline{r}})$ for all $\alpha \in ]0,1[$, and since $\tilde{u}_\infty(0)=0$, there exists $C_\alpha>0$ such that:
	$$
	    |\tilde{u}_\infty(x)|   
	=
	    |\tilde{u}_\infty(x)
	    -
	    \tilde{u}_\infty(0)|
	    \leq
	    C_\alpha |x|^\alpha.
	$$
	From this we then have the following bound:
	\begin{align*}
	    H(\tilde{u}_\infty, r)
	 &=
	    \frac{1}{r^{N-1}}\sum_{i=1}^l\int_{\partial B_r} 
	\tilde{u}_{i,\infty}^2\mu(x)d\sigma(x)
	\leq
	 	    C_\alpha   |\partial B_1|
	    \|D\tilde{A}\|_{
	        L^\infty(B_{\overline{r}})
	    }
	    r^{2\alpha}
	\end{align*}
Therefore, for $C = |\partial B_1|
	    \cdot C_\alpha \cdot
	    \|D\tilde{A}\|_{
	        L^\infty(B_{\overline{r}})
	    }$, we have that:
	$$
	cr^{2(1-\delta)}<H(\tilde{u}_\infty,r)\leq Cr^{2\alpha}
	\qquad
	\forall r \in ]0,\overline{r}[
	$$
	which is a contradiction for $r$ small, by choosing $2\alpha > 2(1-\delta)$ . So we conclude that $\lim_{r\rightarrow 0}N(\tilde{u}_\infty,r) \geq 1$.
	
\noindent	\textbf{Step 2. } $R_{\beta_n} \rightarrow 0$.
	
	Since $r\mapsto (N(\tilde{u}_{\beta_n},r)+1)e^{\tilde{C}r}$ is a continuous monotone nondecreasing function, converging  pointwisely to  the continuous function$r\mapsto (N(\tilde{u}_{\infty}, r)+1)e^{\tilde{C}r}$ in $]0,\tilde{r}[$, then  the convergence is actually uniform over any compact subset in $]0,\tilde{r}]$ (see for example \cite[Lemma 4.3]{SoaveZilio}). We suppose by contradiction that $R_{\beta_n} \rightarrow R_\infty>0$. Then, using the definition \eqref{BigRdefinition} and the the uniform convergence:
	\begin{align*}
	2 > 2-R_\infty &= \lim_n (2-R_{\beta_n}) \geq \lim_{n} (N(\tilde{u}_{\beta_n}, R_{\beta_n}) + 1)e^{\tilde{C}R_{\beta_n}} \\
	&= 
	(N(\tilde{u}_{\infty},R_\infty)+1)e^{\tilde{C}R_{\infty}}\geq (N(\tilde{u}_\infty, 0^+)+1)\geq 2,
	\end{align*}
	which is a contradiction. In the last inequality we used Step 1, $N(\tilde{u}_\infty, 0^+)\geq 1$.
\end{proof}

Now we apply Almgren's monotonicity formula to the blowup sequence $\{v_n\}$ given by (\ref{blowUp}).

\begin{lemma}
	\label{MonotonicityForBlowup}
	Given the constants $\tilde{r}$ and $\tilde{C}>0$ of Lemma \ref{MonotonicityForu} for every $n$, let $\mu_n(x) = \langle A_n(x)\frac{x}{|x|}, \frac{x}{|x|}\rangle$. Then the functions:
	$$
	r\mapsto\left(N(v_n,r)+1\right)e^{\tilde{C}r_n r}, \qquad
	r
	\mapsto H_{i}(v_n,r)e^{\tilde C r_n r}=
	\big(
	\frac{1}{r^{N-1}}
	\int_{\partial B_r}
	\mu_n(x) 	v_{i,n}^2
	d\sigma(x)
	\big)
	e^{\tilde{C}r_nr},
	$$
	for all $i \in \{1,...,l\}$,
	are monotone nondecreasing for $r \in ]0,\frac{\tilde{r}}{r_n}[$.
\end{lemma}
\begin{proof}
Recalling from \eqref{blowUp} that
	$v_n(x)=\frac{\eta(x_n)}{L_nr_n} \tilde{u}_{\beta_n}(r_n x)$, we have
\[
E(v_n,r) = \frac{\eta^2(x_n)}{L_n^2 r_n^2}E(\tilde{u}_{\beta_n}, r_nr),
\quad H_{i}(v_n,r)=  \frac{\eta^2(x_n)}{L_n^2r_n^2}H_i(\tilde{u}_{\beta_n},r) \text{ and } N(v_n,r) = N(\tilde{u}_{\beta_n},r_nr).
\]
	The claim follows as a direct application of Lemma \ref{MonotonicityForu}.
\end{proof}

Let $v_1 = \lim_n v_{1,n}$ and $v_2 = \lim_n v_{2,n}$ be the limits given by Proposition \ref{concludeProp} (the only two possible limiting components of the blowup sequence $v_n$). Next, we prove that \emph{both} $v_1$ and $v_2$ are nonconstant.

\begin{lemma}
	\label{nonTriviality}
	Let $\theta$ be the constant from hypothesis \textbf{(A1)}.
	Then there exists $C = C(\theta, N)>0$, independent of $n$, such that:
	\begin{equation}\label{nonTriviality_eq}
	\frac{1}{r^{N-1}}\int_{\partial B_r}\mu_n(y) v_{i,n}^2d\sigma(y) \geq C 
	\end{equation}
	for every $r \in [2N/\theta^\frac{1}{2},\frac{\tilde{r}}{r_n}]$ and $i=1,2$. In particular, both $v_1$ and $v_2$ are nonconstant in $B_r$ for every $r \in [2N/\theta^\frac{1}{2},\frac{\tilde{r}}{r_n}]$.
\end{lemma}

\begin{remark}
The appearance of the constant $2N/\theta^{\frac{1}{2}}$ is directly related with Lemma \ref{harmonicLemma} in appendix: for a harmonic function $u$ such that $u(0) = 1$ and $|\nabla u (0)|\geq \theta^\frac{1}{2}$, such  lemma implies that $u$ necessarily changes sign in $B_{2N/\theta^{\frac{1}{2}}}$. This fact is used in the following proof.
\end{remark}

\begin{proof}[Proof of Lemma \ref{nonTriviality}] By the monotonicity formula Lemma \ref{MonotonicityForBlowup}, we know that:
	$$
	\frac{1}{r^{N-1}}\int_{\partial B_r}\mu_n(y) v_{i,n}^2
	d\sigma(y) 
	\geq 
	\left(
	\frac{1}{(2N/\theta^\frac{1}{2})^{N-1}} 
	\int_{\partial B_{2N/\theta^\frac{1}{2}}} 
	\mu_n(y) 
	v_{i,n}^2
	d\sigma(y)
	\right)
	e^{\tilde{C}r_n(2N/\theta^\frac{1}{2}-\frac{\tilde{r}}{r_n})}
	$$
	and so we only need to show that there exists $C>0$ such that $\int_{\partial B_{2N/\theta^\frac{1}{2}}}\mu_n(y) v_{i,n}^2d\sigma(y)>C$ for all $n \in \mathbb{N}$. We divide according to the asymptotic behaviour of $M_n$ (recall Proposition \ref{concludeProp}).

\noindent	\textbf{Case (i)} Assume that $M_n$ is bounded. By Proposition \ref{concludeProp} we know that $v_n \rightarrow v$ in $H^1_{loc}(\mathbb{R}^N)\cap C_{loc}(\mathbb{R}^N)$ and that $v_1,v_2$ are nonnegative functions satisfying the system \eqref{boundedEquation} and in particular are subharmonic in $\mathbb{R}^N$; moreover, $v_1(0)+v_2(0) = 1$. From this last fact, we may assume without loss of generality that there exists $C\geq \frac{1}{2}$ such that $v_1(0)\geq C>0$. Using the subharmonicity of $v_1$, we obtain:
	$$
	\int_{\partial B_{2N/\theta^\frac{1}{2}} }  v_1^2
	d\sigma(y)
	> |\partial B_{2N/\theta^\frac{1}{2}}|
	v_1^2(0)
	= 
	C>0,
	$$
	and so, since $v_n \rightarrow v$ in $C_{loc}(\mathbb{R}^N)$, for $n$ large enough we know that:
	$$
	\int_{\partial B_{2N/\theta^\frac{1}{2}}} v_{1,n}^2
	d\sigma(y)
	\geq 
	\frac{1}{2}C.
	$$
	Now we suppose by contradiction that:
	\begin{equation}\label{eq_anothercontradassumption}
	\int_{\partial B_{2N/\theta^\frac{1}{2}}}v_{2,n}^2
	d\sigma(y)
	\rightarrow
	0.
	\end{equation}
	Using the subharmonicity of $v_2$, we conclude that $ v_2(x) = 0$ for $x \in B_{2N/\theta^\frac{1}{2}}$. Since $\gamma\geq 1$, $M_n<0$ and  $v_2(x) = 0$ for all $x \in B_{2N/\theta^\frac{1}{2}}$, by the strong maximum principle we conclude that $v_2(x) = 0$ for all $x \in \mathbb{R}^N$. Going back to the system \eqref{boundedEquation}, then we conclude that $v_1$ is a nonnegative, nontrivial harmonic function in $\mathbb{R}^N$, which is a contradiction. Therefore \eqref{nonTriviality_eq} holds true.
	
		Function $v_1$ is nonconstant since $|\nabla v_1(0)|\geq \theta^{\frac{1}{2}}$, while $v_2$ is nonconstant by \eqref{boundedEquation}.

\noindent	\textbf{Case (ii)} If $M_n \rightarrow - \infty$, we know that $v_n \rightarrow v = (v_1,...,v_l)$ and $v_1, v_2$ satisfy the system \eqref{unboundedEquation}. Doing the same proof by contradiction as in Case (i), assume that $v_1(0)\geq C>0$ and that \eqref{eq_anothercontradassumption} holds true; then, as before, we conclude that: 
	\begin{equation}
	\label{V2IsZero}
	v_2(y) = 0 
	\qquad 
	\forall y \in B_{2N/\theta^\frac{1}{2}}.
	\end{equation}
	Now if $v_1(x_0) = 0$ for some $x_0 \in B_{2N/\theta^\frac{1}{2}}$, then $x_0$ is a zero of the limit profile $v = (v_1,...,v_l)$. The function $v$ satisfies the hypothesis of Theorem \ref{mult2Theoremyeah} in appendix, thus we have that $x_0$ is a zero of multiplicity at least 2. Since $v_1$ and $v_2$ are the only nontrivial components of $v$, that would imply there exists $y \in B_{2N/\theta^\frac{1}{2}}$ such that $v_2(y) > 0$, which is a contradiction with \eqref{V2IsZero}. Thus:
	\begin{equation*}
	    v_1(y)
	    >
	    0
	    \qquad
	    \forall
	    y \in
	    B_{2N/\theta^\frac{1}{2}}
	\end{equation*}
	and, by the system \eqref{unboundedEquation}, this implies that $v_1$ is harmonic in the set $B_{2N/\theta^\frac{1}{2}}$.
	Now $v_1$ is nonnegative harmonic and satisfies $v_1(0) = 1$ and $|\nabla v_1(0)|\geq \theta^\frac{1}{2}$. A harmonic function with these properties by Lemma \ref{harmonicLemma} changes sign in $B_{2N/\theta^{\frac{1}{2}}}$, in contradiction with the fact that $v_1$ is positive in $B_{2N/\theta^{\frac{1}{2}}}$.

In conclusion, we have  deduced \eqref{nonTriviality_eq}.	From \eqref{unboundedEquation} we know that $v_1v_2 = 0$, and using this fact we conclude that both $v_1$ and $v_2$ are nonconstant.
\end{proof}

We now introduce the following quantity:
\begin{equation}
\label{defOfBlowUpRadius}
\overline{r}_n 
:= 
\frac{R_{\beta_n}}{r_n} 
= 
\sup \left\{
r \in ]0, \frac{\tilde{r}}{r_n}[:
\left(
N(v_n,r) +1
\right)e^{\tilde{C}r_nr}
\leq
2-rr_n
\right\}.
\end{equation}
By Lemma \ref{radGoToZero} we know that: 
\begin{equation}
\label{productLimit}
\overline{r}_nr_n = R_{\beta_n} \rightarrow 0.
\end{equation}

The term $\overline{r}_n$ is a kind of threshold between sublinear  (see for instance Lemma \ref{blowDownSequence}) and superlinear behavior for $v_n$.  We refer to \cite[p. 640]{SOAVE2016388}  for more insights. Recall that $\{v_n\}$ satisfies an Almgren monotonicity formula for $r\in ]0,\frac{\tilde r}{r_n}[$ (Lemma \ref{MonotonicityForBlowup}). For $r \in [\overline{r}_n, \frac{\tilde{r}}{r_n}]$, the function
\[
\frac{
	E(v_n,r) + H(v_n,r)
}{
	r^2
}
\]
is almost monotone (Lemma \ref{squaredMonotonicity} below). If $\overline{r}_n$ is bounded, it is not hard to obtain a contradiction (see the proof of Lemma \ref{radBlowUp}). Instead, if $\overline{r}_n\to \infty, $, then for $r \in [\overline{r}_n, \frac{\tilde{r}}{r_n}]$, we will see that $\{v_n\}$ satisfies an Alt-Caffarelli-Friedman type monotonicity formula (Lemma \ref{caffMonot}).  All this information is combined to provide a contradiction also in this situation..

\begin{lemma}
	\label{sharpEstimate}
	Let $C$ be the constant from Lemma \ref{derivOfH}. Then:
	$$\Big|
	\frac{d}{dr}H(v_n,r)
	-
	\frac{2}{r}E(v_n,r)
	\Big|
	\leq 
	C r_nH(v_n,r)\quad \forall r \in]0,\frac{\tilde{r}}{r_n}[$$
\end{lemma}
\begin{proof}
    Since $v_n = \tilde{u}_{\beta_n}(r_nx)$ is a scaling, we just use Lemma \ref{derivOfH} and the identities of the proof of Lemma \ref{MonotonicityForBlowup} to conclude that:
    \begin{align*}
        \Big| 
    	\frac{d}{dr}&H(v_n,r)
    	-
    	\frac{2}{r}E(v_n,r)
    	\Big|
    =
        \frac{\eta^2(x_n)}
        {L_n^2 r_n^2}
        \Big|
    	\frac{d}{dr}
    	\big(
        	H(\tilde{u}_{\beta_n},
        	rr_n)
    	\big)
    	-
    	\frac{2}{r}E(\tilde{u}_{\beta_n},rr_n)
    	\Big|\\
   & =
        \frac{\eta^2(x_n)}
        {L_n^2 r_n^2}
        \Big|
    	r_n
    	H'(\tilde{u}_{\beta_n},
    	rr_n)
    	-
    	\frac{2}{r r_n}r_n E(\tilde{u}_{\beta_n},rr_n)
    	\Big| \leq
        C
        r_n
        \frac{\eta^2(x_n)}
        {L_n^2 r_n^2}
        \Big|
    	H(\tilde{u}_{\beta_n},
    	rr_n)
    	\Big|
    =
        C
        r_n
        H(v_n,r).\qedhere
    \end{align*} 
\end{proof}

\begin{lemma}
	\label{doublingDownRefined}. 
There exists $C>0$ such that:
	\begin{enumerate}
\item If there exists $\tilde{r}$ and $R$ such that $N(v_n,r)\leq d$ for all $0 \leq \tilde{r} \leq r \leq R \leq \frac{\tilde{r}}{r_n}$, then:
	$$
	r \mapsto \frac{H(v_n,r)}{r^{2d}}e^{-Cr_nr}\quad \text{ 	is monotone nonincreasing for $r \in ]\tilde{r}, R[$.}
	$$

	\item If there exists $\tilde{r}$ and $R$ such that $N(v_n,r)\geq \gamma$ for all $0 \leq \tilde{r} \leq r \leq R \leq \frac{\tilde{r}}{r_n}$, then:
	$$
	r \mapsto \frac{H(v_n,r)}{r^{2\gamma}}e^{Cr_nr}\quad \text{	is monotone nondecreasing for $r \in ]\tilde{r},R[$.}
	$$

	\end{enumerate}
\end{lemma}
\begin{proof}
	Follows from the proof of Lemma \ref{doublingLemma}, using the same scaling argument as in Lemma \ref{sharpEstimate}.
\end{proof}

We now define an auxiliary function that will be used in the next lemma. Given the constant $C$ of Lemma \ref{sharpEstimate} and $\tilde{C}$ of Lemma \ref{MonotonicityForBlowup} we define:
\begin{equation*}
\label{DefOfPhiAuxilliary}
	\varphi_n(r):=
	2
	\int_{\overline{r}_n}^{r}
	\left(
	2\frac{
		e^{-\tilde{C}r_nt}-1
	}{
		t
	}
	-
	\frac{
		r_n\overline{r}_ne^{-\tilde{C}r_nt}
	}{
		t
	}
	-
	\frac{Cr_n}{2}
	\right)
	dt
	.
\end{equation*}

\begin{remark}
\label{uniformBoundOnPhi}
	We notice that this sequence is uniformly bounded in $L^\infty([0,\frac{\tilde{r}}{r_n}])$: there exists $\tilde{K}$ such that:
	\begin{align*}
		\varphi_n(r) &\leq 2 \int_{\overline{r}_n}^{r} \left( 2\frac{1-e^{-\tilde{C}r_nt}}{t} + \frac{ r_n \overline{r}_n e^{-\tilde{C}r_nt} }{	t } + \frac{	Cr_n }{ 2 } \right) dt \leq 2\left( \int_{\overline{r}_n}^{\frac{\tilde{r}}{r_n}} 2 \tilde{C} r_n dt+	\overline{r}_nr_n \int_{\overline{r}_n}^{\frac{\tilde{r}}{r_n}}\frac{dt}{t} + \frac{C\tilde{r}}{2}\right)\\
		&\leq 4\tilde{C}\tilde{r} + 2r_n\overline{r}_n\left(|\log(\tilde{r})|+|\log(\overline{r}_nr_n)|\right)\leq \tilde{K},
	\end{align*}
	since, by \eqref{productLimit}, we have $r_n\overline{r}_n \rightarrow 0$ and thus $r_n\overline{r}_n|\log(\overline{r}_nr_n)| \rightarrow 0$.
\end{remark}

\begin{lemma}
	\label{squaredMonotonicity}
	Let $\tilde{C}>0$ be the constant from Lemma \ref{MonotonicityForBlowup}.
	For every $n \in \mathbb{N}$, the function:
	$$
	r 
	\mapsto 
	\frac{
		E(v_n,r) + H(v_n,r)
	}{
		r^2
	}
	e^{\tilde{C}r_nr-\varphi_n(r)}
	$$
	is monotone nondecreasing for $r \in [\overline{r}_n, \frac{\tilde{r}}{r_n}]$.
\end{lemma}
\begin{proof}
	By the definition of $\overline{r}_n$ in \eqref{defOfBlowUpRadius}, we know that:
	$$
	(N(v_n,\overline{r}_n)+1)e^{\tilde{C}r_n\overline{r}_n} = 2-\overline{r}_nr_n.
	$$
	Thus, for $r\in [\overline{r}_n, \frac{\tilde{r}}{r_n}]$, using the monotonicity of $(N(v_n,r)+1)e^{\tilde{C}r_nr}$ from Lemma \ref{MonotonicityForBlowup}, we obtain:
	$$
	(N(v_n,r)+1)e^{\tilde{C}r_nr}
	\geq                     \left(
	N(v_n,\overline{r}_n)+1
	\right)e^{\tilde{C}r_n\overline{r}_n} 
	=
	2-\overline{r}_nr_n
	$$
and
	$$
	N(v_n,r)-1
	\geq
	2\left( e^{-\tilde{C}r_nr}-1
	\right)
	-
	r_n\overline{r}_ne^{-\tilde{C}r_nr}.
	$$
	Now, using Lemma \ref{sharpEstimate}, there exists $C>0$ independent of $n$ such that $H'(v_n,r)\geq \frac{2}{r}E(v_n,r)-Cr_nH(v_n,r)$, so:
\begin{multline*}
	\frac{d}{dr}\log\left(
	\frac{
		H(v_n,r)
	}
	{
		r^2
	}
	\right)
	=
	\frac{H'(v_n,r)}{H(v_n,r)}
	-
	\frac{2}{r}
	\geq
	\frac{2}{r}\left(
	\frac{E(v_n,r)}{H(v_n,r)}-1
	\right) - Cr_n\\
	= 
	\frac{2}{r}\left(
	N(v_n,r) - 1
	\right) - Cr_n
	\geq 
	\frac{4}{r}\left(
	e^{-\tilde{C}r_nr}-1
	\right)
	-
	\frac{2}{r}r_n\overline{r}_ne^{-\tilde{C}r_nr}
	-
    Cr_n = \varphi'_n(r)
	\end{multline*}
	for $r \in [\overline{r}_n, \frac{\tilde{r}}{r_n}]$. This is equivalent to:
	$$
	    \frac{d}{dr}
	    \left(
	        \frac{H(v_n,r)}{r^2}
	    \right)
	    -
	    \varphi_n'(r)\frac{H(v_n,r)}{r^2}
	\geq
	    0
	$$
	and  integrating we deduce that
	$$
	r\mapsto \frac{H(v_n,r)}{r^2}e^{-\varphi_n(r)}
	$$
	is monotone nondecreasing for $r \in [\overline{r}_n, \frac{\tilde{r}}{r_n}]$. To conclude the proof we observe that, by Lemma \ref{MonotonicityForBlowup} and the above observations, for $r \in [\overline{r}_n, \frac{\tilde{r}}{r_n}]$:
\begin{multline*}
	\frac{d}{dr}\log\left(
	\frac{E(v_n,r)+H(v_n,r)}{r^2}e^{\tilde{C}r_nr-\varphi_n(r)}
	\right) = 
	\frac{d}{dr}\log\left(
	(N(v_n,r)+1)e^{\tilde{C}r_nr}
	\frac{H(v_n,r)}{r^2}e^{-\varphi_n(t)}
	\right)\\
	=
	\frac{d}{dr}\log\left(
	(N(v_n,r)+1)e^{\tilde{C}r_nr}
	\right)
	+
	\frac{d}{dr}\log\left(
	\frac{H(v_n,r)}{r^2}e^{-\varphi_n(t)}
	\right)
	\geq
	0\qedhere
\end{multline*}
\end{proof}
\begin{lemma}
	\label{radBlowUp}
	It holds that $\overline{r}_n \rightarrow \infty$ as $n \rightarrow \infty$.
\end{lemma}
\begin{proof}
	We suppose by contradiction that there exists $\overline{r}$ such that, up to a subsequence, $\overline{r}_n \leq \overline{r}$. We know by Proposition \ref{concludeProp} that $v_n \rightarrow v$ in $C_{loc}(\mathbb{R}^n)$ and $H^1_{loc}(\mathbb{R}^n)$. We claim that the limit $v$ satisfies $E(v,r) \geq 0$. We divide the proof of this claim in two cases, according to Proposition \ref{concludeProp}. in case $M_n \rightarrow -\infty$, then:
	$$
	\lim_n
	E(v_n,r)
	=
	E(v,r) 
	= 
	\frac{1}{r^{N-2}}
	\sum_{i=1}^l 
	\int_{B_r}
	|\nabla v_i|^2 \geq 0,
	$$
while, if $M_n \rightarrow M_\infty < 0$, then:
	$$
	\lim_n
	E(v_n,r)
	=
	E(v,r)
	=
	\frac{1}{r^{N-2}}\sum_{i=1}^l\left(
	\int_{B_r}
	|\nabla v_i|^2
	-
	M_\infty\sum_{j\neq i}
	v_i^\gamma
	v_j^{\gamma+1}
	\right) \geq 0.
	$$

	Then, for all $r \in [\overline{r}+1, \frac{\tilde{r}}{r_n}]$, using the monotonicity formula from Lemma \ref{squaredMonotonicity}, we obtain:
	\begin{align}
	0  &\leq  \frac{E(v,r) + H(v,r)}{r^2}  = \lim_{n \rightarrow \infty} \frac{E(v_n,r) + H(v_n,r)}{r^2}\leq \lim_{n \rightarrow \infty} \frac{E(v_n,r) + H(v_n,r)}{r^2}	e^{\tilde{C} r_nr - \varphi_n(r)} e^{\varphi_n(r)- \tilde{C}r_nr } \nonumber \\
	&\leq \lim_{n \rightarrow \infty} r_n^2\frac{E(v_n,\frac{\tilde{r}}{r_n})+H(v_n,\frac{\tilde{r}}{r_n})}{\tilde{r}^2} \sup_{r\in[\overline{r}_n,\frac{\tilde{r}}{r_n}]} e^{\tilde{C}(\tilde{r}-r_nr)+\varphi_n(r)-\varphi_n(\frac{\tilde{r}}{r_n})}  \nonumber \\
	&\leq\lim_{n \rightarrow \infty}  C' r_n^2\frac{E(v_n,\frac{\tilde{r}}{r_n}) + H(v_n,\frac{\tilde{r}}{r_n})}{\tilde{r}^2} = \lim_{n \rightarrow \infty} C'\eta^2(x_n)\frac{E(\tilde{u}_{\beta_n}, \tilde{r}) + H(\tilde{u}_{\beta_n}, \tilde{r})}{L_n^2\tilde{r}^2}, \label{eq:another_chain_of_inequalities_}
	\end{align}
	where we used the bound $e^{\tilde{C}(\tilde{r}-r_nr)+\varphi_n(r)-\varphi_n(\frac{\tilde{r}}{r_n})}\leq e^{2C\tilde{r}+2\tilde{K}} = C'$ (by Remark \ref{uniformBoundOnPhi}, $\varphi_n$ is uniformly bounded by $\tilde{K}$ in $[0,\frac{\tilde{r}}{r_n}]$).
	Now both $E(\tilde{u}_{\beta_n}, \tilde{r})$ and $H(\tilde{u}_{\beta_n}, \tilde{r})$ are uniformly bounded (for $E(\tilde{u}_{\beta_n},\tilde{r})$ proceed as in equation \eqref{boundOnDerivativeProp2} of Proposition \ref{ListProp}, using Lemma \ref{lemmaForMatrixTilde}). Thus, since $L_n \rightarrow \infty$, the last term in the chain of inequalities \eqref{eq:another_chain_of_inequalities_} to $0$, and so we conclude that $v(x)= 0$ for all $x \in B_{\overline{r}+1}(0)$, in contradiction with Lemma \ref{nonTriviality}.
\end{proof}

Up to this point, we just applied the results of Section \ref{chapter:implementation} (Almgren's monotonicity formula) to the blow up sequence $\{v_n\}$. It is now time to use the results of Section \ref{chapter:resultsChap4} (Alt-Caffarelli-Friedman-type monotonicity formula) to $\{v_n\}$ in the interval $r \in [2N/\theta^\frac{1}{2},\frac{\overline{r}_n}{3}]$.

Given the functionals:
$$
J_{1,n}(r) 
:= 
\int_{B_r}\left(
\langle A_n(y)\nabla v_{1,n} , \nabla v_{1,n} \rangle
- M_na_n(y)v_{1,n}^2v_{2,n}^2 - v_{1,n}f_{1,n}(y,v_{1,n})
\right)|y|^{2-N}dy
$$
$$
J_{2,n}(r) 
:= 
\int_{B_r}\left(
\langle A_n(y)\nabla v_{2,n}, \nabla v_{2,n} \rangle
- M_na_n(y)v_{1,n}^2v_{2,n}^2 - v_{2,n}f_{2,n}(y,v_{2,n})
\right)|y|^{2-N}dy,
$$
we also define:
$$J_n(r) = \frac{J_{1,n}(r)J_{2,n}(r)}{r^4}.$$

\begin{lemma}
	\label{caffMonot}
	There exists $c,C>0$, independent of $n$, such that for every $0<\eta<\frac{1}{4}$ we have 
	$$r\mapsto J_n(r)e^{-C|M_n|^{-\eta}r^{-2\eta} + Cr_n^2r^2 + Cr_nr}$$
	is monotone nondecreasing in the interval $[2N/\theta^{-\frac{1}{2}}, \overline{r}_n/3]$ and $
	J_n(2N/\theta^\frac{1}{2})
	=
	\frac
	{
    	J_{1,n}(2N/\theta^\frac{1}{2}) J_{2,n}(2N/\theta^\frac{1}{2})
	}
	{
	    (2N/\theta^\frac{1}{2})^4
	}
	\geq c$.
\end{lemma}

The proof of this lemma consists in showing that the sequence $\{v_n\}$ satisfies conditions $(h0)$--$(h6)$ from Section \ref{chapter:resultsChap4} in the interval $[2N/\theta^{-\frac{1}{2}}, \overline{r}_n/3]$. From this, Lemma \ref{caffMonot} is a direct consequence of Theorem \ref{AltCaffMonotonicity}. However, the proof of such conditions is a delicate and long process. In order not to break the pace of this section, we leave it for later (it will be the content of Section \ref{sec:conditions} below). Instead, assuming the validity of Lemma \ref{caffMonot}, we immediatly pass to the proof of the main result of this paper.  As previously done in Section \ref{chapter:resultsChap4}, for simplicity we focus on the case of $\gamma = 1$, remembering that the proof for general $\gamma$ follows from using Lemma \ref{estimateLemma}-(1) whenever we use Lemma  \ref{estimateLemma}-(2).

\begin{proof}[Proof of Theorem \ref{DesiredTheorem}]
	The claim that the following chain of inequalities hold true:
	\begin{align}
	\label{firstOfTheLast}
	0
	<
	c
	\leq
	J_n(2N/\theta^\frac{1}{2})
	&\leq
	CJ_n(\frac{\overline{r}_n}{3})
	\leq
	C\left(
	\frac{E(v_n,\overline{r}_n) + H(v_n,\overline{r}_n)}{\overline{r}_n^2} + o_n(1)
	\right)^2\\
	&\leq
	C\left(
	r_n^2\frac{E(v_n,\frac{\tilde{r}}{r_n}) + H(v_n,\frac{\tilde{r}}{r_n})}{\tilde{r}^2} + o_n(1)
	\right)^2
	\label{LatsOfTheLast}
	\end{align}
	where $o_n(1)\rightarrow 0$.  A 	contradiction follows as soon as we prove this claim; indeed,  using the same arguments as in Lemma \ref{radBlowUp}, we have that:
	$$
	0
	<
	c
	\leq
	C\left(
	r_n^2\frac{E(v_n,\frac{\tilde{r}}{r_n}) + H(v_n,\frac{\tilde{r}}{r_n})}{\tilde{r}^2} + o_n(1)
	\right)^2
=
    C\left(
	\eta^2(x_n)\frac
	{E(\tilde{u}_{\beta_n},\frac{\tilde{r}}{r_n}) + H(\tilde{u}_{\beta_n},\frac{\tilde{r}}{r_n})
	}
	{
	L_n^2\tilde{r}^2
	} 
	+ 
	o_n(1)
	\right)^2
	\rightarrow 0,
	$$ 
	which results in contradiction, and the theorem is proved.

Now we prove the claim.	 The first two inequalities of \eqref{firstOfTheLast} follow from Lemma \ref{caffMonot}, since:
	$$
	    0
	<
	    c
	\leq
	    J_n(2N/\theta^\frac{1}{2})
	\leq
	    J_n(\frac{\overline{r}_n}{3})
	    e^{
	        -C|M_n|^{-\eta}
	        \left(
	        (\frac{\overline{r}_n}{3})^{-2\eta}
	        -
	        (
	        2N/\theta^\frac{1}{2}
	        )^{-2\eta}
	        \right)
	        +
	        Cr_n^2
	        \left(
	            (
	            \frac
	            {\overline{r}_n}
	            {3}
	            )^2
	            -
	            (2N/\theta^\frac{1}{2})^2
	        \right)
	        +
	        Cr_n
	        \left(
	            \frac
	            {\overline{r}_n}
	            {3}
	            -
	            (2N/\theta^\frac{1}{2})
	        \right)
	    }.
	$$
	By equation \eqref{productLimit} we have that $r_n \overline{r}_n \rightarrow 0$ and, by \eqref{Domain}, $r_n \rightarrow 0$. Furthermore, since by Proposition \ref{concludeProp} we either have $|M_n| \rightarrow \infty$ or $|M_n| \rightarrow |M_\infty|>0$, then this implies that
	$|M_n|^{-2\eta}(2N/\theta^\frac{1}{2})^{-2\eta}$ is bounded in $n$. With these observations, we conclude that:
	$$
	    e^{
	        -C|M_n|^{-2\eta}
	        \left(
	        \frac{\overline{r}_n}{3}^{-2\eta}
	        -
	        (
	        2N/\theta^\frac{1}{2}
	        )^{-2\eta}
	        \right)
	        +
	        Cr_n^2
	        \left(
	            \frac
	            {\overline{r}_n}
	            {3}^2
	            -
	            (2N/\theta^\frac{1}{2})^2
	        \right)
	        +
	        Cr_n
	        \left(
	            \frac
	            {\overline{r}_n}
	            {3}
	            -
	            (2N/\theta^\frac{1}{2})
	        \right)
	    }
	$$
	is bounded in $n$, and so the second inequality of \eqref{firstOfTheLast} follows. The last inequality \eqref{LatsOfTheLast} follows from Lemma \ref{squaredMonotonicity}. 
	
	The only thing left to prove is the middle inequality
	\begin{equation}
		\label{IneqProof}
		J_n(\frac{\overline{r}_n}{3})
		\leq
		C\left(
		\frac{E(v_n,\overline{r}_n) + H(v_n,\overline{r}_n)}{\overline{r}_n^2} + o_n(1)
		\right)^2
	\end{equation}
and we proceed to prove it. First, we notice that:
	\begin{align*}
	\frac{1}{\overline{r}_n^2} \int_{B_{\overline{r}_n}/3} &\left(
	\langle A_n(y)\nabla v_{1,n},\nabla v_{1,n} \rangle
	- 
	M_na_n(y)v_{1,n}^2v_{2,n}^2
	- 
	v_{1,n}f_{1,n}(y,v_{1,n})
	\right)
	|y|^{2-N}dy
	\\
	\leq&
	\frac{1}{\overline{r}_n^2}\int_{B_{\overline{r}_n}}\left(
	\langle A_n(y)\nabla v_{1,n},\nabla v_{1,n} \rangle
	- 
	M_na_n(y)v_{1,n}^2v_{2,n}^2
	- 
	v_{1,n}f_{1,n}(y,v_{1,n})
	\right)
	|y|^{2-N}dy
	\\
	&+
	\frac{1}{\overline{r}_n^2}\int_{B_{\overline{r}_n}\setminus B_{\overline{r}_n/3}} v_{1,n}f_{1,n}(y,v_{1,n})|y|^{2-N}dy.
	\end{align*}
	We divide the proof of \eqref{IneqProof} in two steps.
	
\noindent	\textbf{Step 1.} We show that:
	\begin{equation}
		\label{vanishTerm}
		\lim_{n \rightarrow \infty}
		\Big|
		\frac{1}{\overline{r}_n^2}\int_{B_{\overline{r}_n}-B_{\overline{r}_n/3}} v_{1,n}f_{1,n}(y,v_{1,n})|y|^{2-N}dy
		\Big| 
		= 
		0.
	\end{equation}
	Indeed,	using the bound $|f_{1,n}(x,v_{1,n})|\leq dr_n^2|v_{1,n}|$, obtained in Lemma \ref{lemmaMatrixBlowLimit}, we have that:
	\begin{multline*}
	\Big|
	\frac{1}{\overline{r}_n^2}
	\int_{B_{\overline{r}_n}(0)-B_{\overline{r}_n/3}(0)} v_{1,n}
	f_{1,n}(y,v_{1,n})
	|y|^{2-N}dy
	\Big|
	\leq
	\frac{d r_n^2}{\overline{r}_n^2}\int_{B_{\overline{r}_n}(0)-B_{\overline{r}_n/3}(0)}v^2_{1,n}|y|^{2-N}dy\\
	=
	\frac{d\eta^2(x_n)}{L_n^2\overline{r}_n^2r_n^2}
	\int_{B_{\overline{r}_nr_n}(0)-B_{\overline{r}_nr_n/3}(0)}
	\frac{\tilde{u}_{1,\beta_n}^2}{|y-x_n|^{N-2}}dy
	\leq
	\frac{
		Cm^2
	}{
		L^2_n\overline{r}_n^Nr_n^N
	}\int_{B_{\overline{r}_nr_n}(0)}1dy
	\leq
	\frac{C}{L_n^2} \rightarrow 0,
	\end{multline*}
	where we use the uniform $L^\infty$-bound, $|\tilde{u}_{1,\beta_n}| \leq m$.
	
\noindent	\textbf{Step 2.} Show that:
	$$
	\frac{1}{\overline{r}_n^2}
	\int_{B_{\overline{r}_n}(0)}\left(
	\langle A_n(y)\nabla v_{1,n}, \nabla v_{1,n} \rangle
	-
	M_na_n(y)v_{1,n}^2v_{2,n}^2
	-
	v_{1,n}f_{1,n}(y,v_{1,n})
	\right)|y|^{2-N}
	dy
	$$
	$$
	\leq
	C\frac{
		E(v_n,\overline{r}_n) 
		+
		H(v_n,\overline{r}_n)}{\overline{r}_n^2}.
	$$
	To prove this, we use equation \eqref{randomEqJ} to conclude that there exists $\alpha>0$ such that:
	\begin{align}
		    J_{1,n}(r_n)
		    &=\int_{B_{\overline{r}_n}}
			\left(
			\langle A_n(y)\nabla v_{1,n}, \nabla v_{1,n} \rangle
			-
			M_na_n(y)v_{1,n}^2v_{2,n}^2
			-
			v_{1,n}
			f_{1,n}(y,v_{1,n})
			\right)|y|^{2-N}
			\nonumber \\
			&\leq
			\frac{1}{\overline{r}_{n}^{N-2}}
        	\int_{B_{\overline{r}_n}}\left(
			\langle A_n(y)\nabla v_{1,n}, \nabla v_{1,n} \rangle
			-
			M_na_n(y)v_{1,n}^2v_{2,n}^2
			-
			v_{1,n}
			f_{1,n}(y,v_{1,n})
			\right)dy
			\nonumber \\
			&\qquad+
			\frac{(N-2)(1+\alpha r_n\overline{r}_n)}{2\overline{r}_n^{N-1}}\int_{\partial B_{\overline{r}_n}}\mu_n(y)v_{1,n}^2d\sigma(y). 		\label{ineqFinal}
	\end{align}
	Now we notice that for all $n$ and $i\in\{1,...,l\}$, since $M_n<0$ and $a_n(y)>0$, we have that:
	\begin{multline*}
		\frac{1}{\overline{r}_n^{N-2}}
		\int_{B_{\overline{r}_n}}\left(
		\langle A_n(y)\nabla v_{i,n}, \nabla v_{i,n} \rangle
		-
		M_na_n(y)v_{i,n}^2
		\mathop{\sum_{j=1}^l}_{j\neq i}
		v_{j,n}^2
		-
		v_{i,n}f_{i,n}(y,v_{i,n})
		\right)
		dy
		\\
		\qquad+
		\frac{(N-2)(1+\alpha r_n\overline{r}_n)}{2\overline{r}_n^{N-1}}\int_{\partial B_{\overline{r}_n}}\mu_n(y)v_{i,n}^2d\sigma(y)\\
		\geq
		\frac{\eta^2(x_n)}{L_n^2r_n^2}\left(
		\int_{B_{{r_n\overline{r}_n}}}
		\left(
		    \frac{1}{(r_n\overline{r}_n)^{N-2}}
    		\langle \tilde{A}_n(y)\nabla \tilde{u}_{i,\beta_n},
    		\nabla \tilde{u}_{i,\beta_n}
    		\rangle
    		-
    		\frac{d(\overline{r}_nr_n)^2}{(\overline{r}_nr_n)^N} \tilde{u}_{i,\beta_n}^2
		\right)
		dy
		\right.\\
		\qquad+
		\left.
		\frac{(N-2)(1+\alpha r_n\overline{r}_n)}{2(\overline{r}_nr_n)^{N-1}}
		\int_{\partial B_{\overline{r}_nr_n}}
		\mu_n(\frac{y}{r_n})
		\tilde{u}_{i,\beta_n}^2
		d\sigma(y)
		\right)\\
		\geq
		\frac{\eta^2(x_n)}{L_n^2r_n^2}\left(
		\int_{B_{{r_n\overline{r}_n}}}
		\left(
    		\frac{1}{(r_n\overline{r}_n)^{N-2}}
    		\frac{
    			\theta
    		}{
    			M
    		}
    		|\nabla \tilde{u}_{i,\beta_n}|^2
    		-
    		\frac{d(\overline{r}_nr_n)^2}{(\overline{r}_nr_n)^N} \tilde{u}_{i,\beta_n}^2
		\right)
		dy
		%\right)+\\
		\right.\\
		\qquad
		\left.
		+
		%\frac{\eta^2(x_n)}{L_n^2r_n^2}
		\frac{
			(N-2)(1+\alpha r_n\overline{r}_n)(1-Cr_n\overline{r}_n)
		}{
			2(\overline{r}_nr_n)^{N-1}
		}
		\int_{
			\partial B_{\overline{r}_nr_n}
		}
		\tilde{u}_{i,\beta_n}^2
		d\sigma(y)
		\right).
		\label{starEquation123}
	\end{multline*}
	In the first inequality we made a change of variables using the definition of $v_{i,n}$ in \eqref{blowUp} and the inequality of $f_{i,n}$ from Lemma \ref{lemmaMatrixBlowLimit}. 
	In the last inequality, we used the ellipticity constant for $\tilde{A}_n(y) = A_n(\frac{y}{r_n})$ given in Lemma \ref{lemmaForMatrixTilde} and that:
	\begin{align}
	    \mu_n(\frac{y}{r_n})
	    =&
	    \langle
            A_n(\frac{y}{r_n})
            \frac{y}{|y|}
            ,
            \frac{y}{|y|}
	    \rangle
	    =
	    1
	    +
	    \langle
	        (
	        I
	        -
            A_n(\frac{y}{r_n})
            ) 
            \frac{y}{|y|}
            ,
            \frac{y}{|y|}
	    \rangle=
	    1
	    +
	    \langle
	        (
	        I
	        -
            \tilde{A}_n(y)
            ) 
            \frac{y}{|y|}
            ,
            \frac{y}{|y|}
	    \rangle    
	    \geq 
	    (1-Cr_n\overline{r}_n)
	\end{align}
	for $y \in B_{r_n\overline{r}_n}$.
	Now, using Poincar\'e's inequality (Lemma \ref{PoincareIneq}) we have that \eqref{starEquation123} above is larger than or equal to:
	\begin{multline*}
		\frac{\eta^2(x_n)}{L_n^2r_n^2}
		\Bigg[
		\frac{1}{(\overline{r}_nr_n)^{N-2}}
		\left(
		\frac{\theta}{M} 
		- 
		\frac{
			d(r_n\overline{r}_n)^2
		}{
			N-1
		}
		\right)
		\int_{B_{r_n\overline{r}_n}}
		|\nabla \tilde{u}_{i,\beta_n}|^2
		dy
		\\
		 +
		\left(
		\frac{
			(N-2)(1+\alpha r_n\overline{r}_n)(1-Cr_n\overline{r}_n)
		}{
			2
		}
		-
		\frac{d(r_n\overline{r}_n)^2}{N-1}
		\right)
		\frac{1}{(\overline{r}_nr_n)^{N-1}}
		\int_{
			\partial B_{\overline{r}_nr_n}(x_n)
		}
		\tilde{u}_{i,\beta_n}^2
		d\sigma(y)
		\Bigg].
	\end{multline*}
	Since $r_n\overline{r}_n \rightarrow 0$ the above is positive for $n$ large enough, thus we can assume that:
	\begin{align}
			&\frac{1}{\overline{r}_n^{N-2}}
			\int_{B_{\overline{r}_n}}\left(
			\langle A_n(y)\nabla v_{i,n}, 
			\nabla
			v_{i,n} \rangle
			-
			M_na_n(y)v_{i,n}^2
			\mathop{\sum_{j=1}^l}_{j\neq i}
			v_{j,n}^2
			-
			v_{i,n}f_{i,n}(y,v_{i,n})
			\right)dy
\nonumber			\\
			&\qquad+
			\frac{(N-2)(1+\alpha r_n\overline{r}_n)}{2\overline{r}_n^{N-1}}\int_{\partial B_{\overline{r}_n}}\mu_n(y)v_{i,n}^2(y)d\sigma(y)
			\geq 0 		\label{ineqForAll}
	\end{align}
	for all $i \in \{1,...,l\}$ and large $n$. Now, coming back to equation \eqref{ineqFinal}, using inequality \eqref{ineqForAll} for $i=2,...,l$ we obtain:
	\begin{align*}
		&
		\int_{B_{\overline{r}_n}}\left(
		\langle A_n(y)\nabla v_{1,n}, \nabla v_{1,n} \rangle
		-
		M_na_n(y)v_{1,n}^2v_{2,n}^2
		-
		v_{1,n}f_{1,n}(y,v_{1,n})
		\right)|y|^{2-N}
		dy\\
		&\leq \sum_{i=1}^l\left[
		\frac{1}{\overline{r}_n^{N-2}}
		\int_{B_{\overline{r}_n}}\Big(
		\langle A_n(y)\nabla v_{i,n}, \nabla v_{i,n} \rangle
		-
		M_na_n(y)v_{i,n}^2
		\mathop{\sum_{j=1}^l}_{j\neq i}
		v_{j,n}^2
		-
		v_{i,n}f_{i,n}(y,v_{i,n})
		\Big)dy
		\right.\\
		\qquad & 
		\qquad+ 
		\frac{(N-2)(1+\alpha r_n\overline{r}_n)}{2\overline{r}_n^{N-1}}\int_{\partial B_{\overline{r}_n}}\mu_n(y)v_{i,n}^2d\sigma(y)
		\Bigg]\\
		&\leq
		E(v_n,\overline{r}_n) + \frac{(N-2)(1+\alpha r_n\overline{r}_n)}{2}H(v_n,\overline{r}_n)\\
		&\leq C\left(E(v_n,\overline{r}_n) + H(v_n,\overline{r}_n)\right).
	\end{align*}
	
\noindent	\textbf{Step 3.} $J_n(\frac{\overline{r}_n}{3})\leq C\left(
		\frac{E(v_n,\overline{r}_n) + H(v_n,\overline{r}_n)}{\overline{r}_n^2} + o_n(1)
		\right)^2.$
	
	We can also do the same calculation for $v_{2,n}$ obtaining a vanishing term as in \eqref{vanishTerm} and inequalities similar to the ones from Step 1. and Step 2, and so we have:
	$$
	    J_{1,n}
	    (\frac{\overline{r}_n}{3})
	\leq
	    C\left(
		\frac{E(v_n,\overline{r}_n) + H(v_n,\overline{r}_n)}{\overline{r}_n^2} + o_n(1)
		\right),
		\qquad
		J_{2,n}
	    (\frac{\overline{r}_n}{3})
	\leq
	    C\left(
		\frac{E(v_n,\overline{r}_n) + H(v_n,\overline{r}_n)}{\overline{r}_n^2} + o_n(1)
		\right).
	$$
Taking the product of these inequalities we obtain the desired inequality \eqref{IneqProof}, concluding the proof.
\end{proof}

\section{Conditions for the Alt-Caffarelli-Friedman-type monotonicity formula: proof of Lemma \ref{caffMonot}}\label{sec:conditions}

In this section we prove that Lemma \ref{caffMonot} is true (which was already used in the previous section to show Theorem \ref{DesiredTheorem}). This is a consequence of showing that conditions $(h_0)$-$(h_6)$ are satisfied by the blowup sequence $v_n$ defined in \eqref{blowUp}, to which we apply the Alt-Caffarelli-Friedman-type monotonicity formula, Theorem \ref{AltCaffMonotonicity}. Similarly to that section, for simplicity of notation we restrict our attention to the case $\gamma=1$. Remembering the main quantities at the beginning of Section \ref{chapter:resultsChap4}, we will set what some of them are in this context, while others will only be obtained through the proof of certain lemmas.
We set
\[
c_n := Cr_n,\qquad R_n := \frac{\overline{r}_n}{3},
\]
where $r_n$ is defined in \eqref{Domain} and $\overline{r}_n$ in \eqref{defOfBlowUpRadius}. By Lemma \ref{radBlowUp}, for $n$ large enough, $R_n>1$. The constant $C$ is obtained from Lemma \ref{lemmaMatrixBlowLimit} in such a way that:
$$
\sup_{y \in B_r(0)}\|A_n(y)-Id\|\leq Cr_n r.
$$
With this, conditions ($h_0$) and ($h_4$) are automatically satisfied, since $r_n \overline{r}_n \rightarrow 0$.

We now proceed to prove conditions $(h_1)$, $(h_2)$. With this we also define the quantity of Section \ref{chapter:resultsChap4} 
$$\epsilon_n := dr_n^2$$

\begin{lemma}
\label{easyLemmaForConditions123}
    Provided $n$ is sufficiently large, there holds:
    $$
        |f_{i,n}(x,v_{i,n})|
    \leq
        dr_n^2v_{i,n},
    \qquad
        R_n^2\epsilon_n=\frac{dr_n^2\overline{r}_n^2}{9}
    \leq
        \big(\frac{N-2}{2}\big)^2-\delta,
        \qquad
        \text{ for small }
        \delta>0.
    $$
    In particular, conditions $(h_1)$ and $(h_2)$ are satisfied.
\end{lemma}
\begin{proof}
This is an easy consequence of Lemma \ref{lemmaMatrixBlowLimit} and the fact that $r_n\overline{r}_n \rightarrow 0$ as $n \rightarrow \infty$.
\end{proof}

The rest of the conditions are proved in different lemmas below. Condition $(h_3)$ is proved in Lemma \ref{sameWeightAltCafCond}, $(h_5)$ is proved in Lemma \ref{almgrenLemmaHypoithesisStuff} and $(h_6)$ is proved in Lemmas \ref{LambdasMoreThanZeroAltCafCond} and \ref{LastLemma}.

Next, we are going to prove a couple of lemmas that will be the main tools for the rest of this section. In particular, Lemma \ref{blowDownSequence} concerns the characterization of certain blowdown-sequences.

\begin{lemma}
	\label{doublingForAlt}
	There exists $\sigma \in ]0,1[$ such that 
	\[
\text{	$\sigma \leq N(v_n,r) \leq 1$ for every $r \in [2N/\theta^\frac{1}{2},\overline{r}_n]$ for every $n$. }
	\]As a consequence, from Lemma \ref{doublingDownRefined}, there exists $C>0$ such that:	
$$
	r \mapsto \frac{H(v_n, r)}{r^2}e^{-Crr_n}\quad \text{	is monotone nonincreasing for $r \in [2N/\theta^\frac{1}{2}, \overline{r}_n]$,}
	$$
and:
	$$
	r \mapsto \frac{H(v_n,r)}{r^{2\sigma}}e^{Crr_n}\quad \text{ 	is monotone nondecreasing for $r \in [2N/\theta^\frac{1}{2},\overline{r}_n]$. }
	$$
\end{lemma}
\begin{proof}
 	From the Almgren monotonicity formula (Lemma \ref{MonotonicityForBlowup}) and the definition of $\overline{r}_n$ in \eqref{defOfBlowUpRadius}, we obtain:
	\begin{align*}
    	N(v_n,r)
    	+1
    	\leq
		\left(
		N(v_n,r)
		+
		1
		\right)
		e^{\tilde{C}r_nr}
		\leq
		\left(
		N(v_n,\overline{r}_n)
		+
		1
		\right)e^{\tilde{C}r_n\overline{r}_n}
		=
		2
		-
		r_n\overline{r}_n,
	\end{align*}
	so that
	$$
	    N(v_n,r)
	\leq
	    (2-r_n\overline{r}_n)
	-
	    1
	=
	    1-r_n\overline{r}_n
	$$
	for all $r \in [0,\overline{r}_n]$. This gives the upper bound on $N(v_n,r)$.

	For the lower bound, we use again the Almgren monotonicity formula to conclude that:
	$$
	\left(
	N(v_n,r)
	+
	1	
	\right)
	e^{\tilde{C}r_nr}
	\geq
	\left(
	N(v_n,2N/\theta^{\frac{1}{2}})
	+
	1
	\right)
	e^{\tilde{C}r_n2N/\theta^{\frac{1}{2}}}
	$$
	for every $r \in [2N/\theta^\frac{1}{2},\overline{r}_n]$, which implies:
	\begin{equation}
		\label{eqinN}
		N(v_n,r) 
		\geq 
		\left(
		N(v_n,2N/\theta^{\frac{1}{2}})
		+
		1
		\right)
		e^{-\tilde{C}r_nr}
		-
		1.
	\end{equation}
	By Lemma \ref{nonTriviality}, the limits $v_{1}$ and $v_{2}$ are nonconstant in $B_{2N/\theta^\frac{1}{2}}$, thus there exists $\tilde{C}>0$ such that:
	$
	    \tilde{C}
	    \leq
	    \sum_{i=1}^l
	    \int_{B_{2}}
	    |\nabla v_i|^2
	    dx.
	$
	Now, by Lemma \ref{lemmaMatrixBlowLimit}, we know that $f_{i,n}(x,v_{i,n})$ converges locally uniformly to zero, thus by the convergence of $v_n$ in $H^1(B_{2N/\theta^\frac{1}{2}})\cap C(\overline{B}_{2N/\theta^\frac{1}{2}})$ and that $M_n < 0$, $a_n(x)>0$ we obtain:
	\begin{align*}
	    0<\tilde{C}
	&\leq
	    \sum_{i=1}^l
	    \int_{B_{2}}
	    |\nabla v_i|^2
	    dx
	\leq
	    E(v,2N/\theta^\frac{1}{2})\\
	&\leq
	    \lim_n
	    \sum_{i=1}^l
    	\int_{B_{2N/\theta^{\frac{1}{2}}}}\left(
    	\langle A_n(x)\nabla v_{i,n}, \nabla v_{i,n} \rangle
    	-
    	2M_n\sum_{i<j}a_n(x)v_{i,n}^2v_{j,n}^2
    	\right)dx\leq
        \lim_n
        E(v_n,2N/\theta^\frac{1}{2}).
	\end{align*}
	Also by the local uniform convergence $v_n \rightarrow v$ we have $\alpha =H(v,2N/\theta^\frac{1}{2}) = \lim_n H(v_n,2N\theta^\frac{1}{2})$ thus:	
	$N(v_n,2N/\theta^{\frac{1}{2}})\geq \frac{\tilde{C}}{2\alpha}>0$
	for $n$ large enough.
	Since $r_nr \leq r_n \overline{r}_n \rightarrow 0$ as $n \rightarrow \infty$, coming back to equation \eqref{eqinN} we obtain:
	\[
	N(v_n,r)
	\geq
	(\frac{\tilde{C}}{2\alpha}+1)e^{-\tilde{C}r_nr}-1
	\geq
	\sigma
	>
	0\quad  \text{	for all $r \in [2N/\theta^\frac{1}{2},\overline{r}_n]$. } \qedhere
	\]
\end{proof}

We will now show that the limit of the blowdown sequence defined below in \eqref{linearLimit} behaves linearly in a ball $B_1$. For the next lemma, we recall that $\overline{r}_n \rightarrow \infty$ by Lemma \ref{radBlowUp}.

\begin{lemma}
	\label{blowDownSequence}
	Let $(\rho_n)$ be a sequence such that $\rho_n \rightarrow \infty$ and $\rho_n \leq \frac{\overline{r}_n}{3}$. Then there exists $h,k \in \{1,...l\}$ and $\gamma_h,\gamma_k>0$ such that the blowdown sequence:
	\begin{equation}
	\label{linearLimit}
	\tilde{v}_{i,n}(x)
	:=
	\frac{
		v_{i,n}(\rho_n x)
	}{
		\sqrt{H(v_n,\rho_n)}
	} 
	\end{equation}
	converges in $H^1(B_1)\cap C(\overline{B_1})$, up to a rotation, to a function $\tilde{v}=(\tilde{v}_1,\ldots, \tilde{v}_l)$ defined by:
	$$
    	\tilde{v}_h(x) = \gamma_h x_1^+,
    	\qquad
    	\tilde{v}_k(x) = \gamma_k x_1^-,
    	\qquad
    	\tilde{v}_j(x) = 0,
    	\qquad
    	\forall j \neq h,k.
	$$
\end{lemma}
\begin{proof}
	First we observe that the sequence $\tilde{v}_{n}$ satisfies the system:
	\begin{equation}
		\label{blowDownEq}
		-\div(A_n(\rho_n x)\nabla \tilde{v}_{i,n})
		=
		\frac{
			\rho_n^2
		}{
			\sqrt{H(v_n,\rho_n)}
		}
		f_{i,n}(\rho_n x,v_{i,n}(\rho_n x))
		+
		\rho_n^2
		H(v_n,\rho_n)
		M_n
		\tilde{v}_{i,n}
		\sum_{j \neq i}
		a_n(\rho_n x)\tilde{v}_{j,n}^2
	\end{equation}
	in a set $B_3 \subset \tilde{\Omega}_n = \frac{\Omega_n}{\rho_n}$ (this follows from Proposition \ref{ListProp}, since $B_{1/M^\frac{1}{2}r_n} \subset \Omega_n$ and $r_n\rho_n\rightarrow 0$).
	Since $\rho_n \rightarrow \infty$, by Lemma \ref{nonTriviality} and by Proposition \ref{concludeProp} we know there exists $C>0$ small enough such that $H(v_n,\rho_n)\geq C$ and $M_n \leq -C<0$ for all $n$.
	
	With this we can conclude that the competition parameter in equation \eqref{blowDownEq} satisfies $\rho_n^2H(v_n,\rho_n)M_n\rightarrow -\infty$. Also, by Lemma \ref{lemmaMatrixBlowLimit}, it follows that:
	\begin{align*}
		-\div(A_n(\rho_n x)\nabla \tilde{v}_{i,n}) 
		&\leq
		\frac{
			\rho_n^2
		}{
			\sqrt{H(v_n,\rho_n)}
		}
		f_{i,n}(\rho_nx, v_{i,n}(\rho_n x))\leq
			\frac{
			\rho_n^2
		}{
			\sqrt{H(v_n,\rho_n)}
		}
		dr_n^2v_{i,n}(\rho_n x)
		\leq
		d(\rho_n r_n)^2 \tilde{v}_{i,n}(x)
	\end{align*}
	in $B_3$. By the ellipticity of $A_n(\rho_n x)$ and a Brezis-Krato-type argument (see for instance \cite[Appendix B.2, B.3]{Struwe}), if we show that $\tilde{v}_{i,n}$ has a uniform bound in $H^1(B_3)$ then it follows that the sequence $\tilde{v}_{i,n}$ has a uniform bound in $L^\infty(B_2)$. %We know that $\rho_nr_n \rightarrow 0$, so we just need to show a uniform bound in $\tilde{v}_{i,n}$ in $H^1(B_3)$.
	
	Now by Lemma \ref{doublingForAlt} we know that:
	\begin{equation}
	\label{upBoundOnN1236}
	N(\tilde{v}_n,\rho) 
	= 
	N(v_n,\rho \rho_n)
	\leq
	1
	\end{equation}
	for every $0\leq \rho\leq 3$. Thus, by Lemma \ref{doublingForAlt}, for all $1 \leq \rho \leq 3$ there exists $C>0$ such that:
	\begin{align*}
	H(\tilde{v}_n,\rho)
&=
    \frac{1}{\rho^{N-1}}
    \sum_{i=1}^l
    \int_{\partial B_\rho}
    \langle
        A_n(\rho_ny)
        \frac{y}{|y|}
        ,
        \frac{y}{|y|}
    \rangle
    \tilde{v}_{i,n}^2
    d\sigma(y)=
    \frac{1}{(\rho\rho_n)^{N-1}H(v_n,\rho_n)}
    \sum_{i=1}^l
    \int_{\partial B_{\rho\rho_n}}
    \mu_n(y)
    v_{i,n}^2
    d\sigma(y)\\
&=
	\frac{
		H(v_n,\rho_n\rho)
	}{
		H(v_n,\rho_n)
	}
\leq
	e^{-C(r_n\rho_n\rho-r_n\rho_n)}\rho^2.
	\end{align*}
	Therefore, 	since $r_n\rho_n \leq r_n \overline{r}_n \rightarrow 0$,
	$$
	E(\tilde{v}_n,3)
	=
	N(\tilde{v}_n,3)
	H(\tilde{v}_n,3)
	\leq
	9e^{-3Cr_n\rho_n}.
	$$
	
	With this upper bound for the energy we are able to show the upper bound for the $H^1(B_3)$ norm. Indeed, using the ellipticity constant for $A_n$ and the property for $f_i$ from Lemma \ref{lemmaMatrixBlowLimit}, and Poincar\'e's inequality (Lemma \ref{PoincareIneq}): %(here $C>0$ will just be a positive constant independent of $n$ such that the inequality is satisfied):
	\begin{align*}
		E(\tilde{v}_n,3)
		&\geq
		\sum_{i=1}^l
		\frac{1}{3^{N-2}}
		\int_{B_3}	
		\left(
		\langle 
		A_n(\rho_n x)\nabla \tilde{v}_{i,n}
		, 
		\nabla \tilde{v}_{i,n}
		\rangle
		+
		\frac{\rho_n^2}{\sqrt{H(v_n,\rho_n)}}
		f_{i,n}(\rho_nx, v_{i,n}(\rho_nx))
		\tilde{v}_{i,n}
		\right)
		dx\\
		&\geq
		\sum_{i=1}^l
		\int_{B_3}	
		\left(
		\frac{\theta}{M 3^{N-2}}
		|\nabla \tilde{v}_{i,n}|^2
		-	
		\frac{d3^2(\rho_n r_n)^2}{3^{N}}
		\tilde{v}^2_{i,n}
		\right)
		dx\\
		&\geq
		\sum_{i=1}^l
		\int_{B_3}
		\frac{\theta}{M 3^{N-2}}
		|\nabla \tilde{v}_{i,n}|^2
		dx
		-	
		\frac{d3^2(\rho_n r_n)^2}{(N-1)}
		\left(
		\int_{B_3}
		\frac{1}{3^{N-2}}
		|\nabla \tilde{v}_{i,n}|^2
		dx
		+
		\int_{\partial B_3}
		\frac{1}{3^{N-1}}
		\tilde{v}_{i,n}^2
		d\sigma(x)
		\right)
		\\
		&\geq
		\frac{1}{3^{N-2}}
		(\frac{\theta}{M}-\frac{d3^2(\rho_n r_n)^2}{(N-1)})
		\int_{B_3}
		\sum_{i=1}^l
		|\nabla \tilde{v}_{i,n}|^2
		dx
		-
		\frac{d3^2M}{\theta(N-1)}(\rho_n r_n)^2
		H(\tilde{v}_n,3)\\
		&\geq
		C\int_{B_3}\sum_{i=1}^l|\nabla \tilde{v}_{i,n}|^2dx - o_n(1)
	\end{align*}
	by the observations above and since $\rho_nr_n\to 0$. This gives the desired $H^1(B_3)$ bound, and so $\tilde{v}_n$ is uniformly bounded in $L^\infty(B_2)$.

	With this, using Proposition \ref{propOfBlowUpInNiceSpace} in appendix and $\rho_n \rightarrow \infty$, we conclude that there exists $\tilde{v} \in C(B_{\frac{3}{2}}) \cap H^1(B_{\frac{3}{2}})$ such that $\tilde{v}_n \rightarrow \tilde{v}$ in both $C(B_{\frac{3}{2}}) \cap  H^1(B_{\frac{3}{2}})$. Moreover, we have that $\tilde{v} \in \mathcal{G}(B_{\frac{3}{2}})$ (see Definition \ref{GspaceDef} below). By Lemma \ref{lemmaMatrixBlowLimit}, we have:
	$$
	\Big|
	\frac{\rho_n^2}{\sqrt{H(v_n,\rho_n)}}
	f_{i,n}(\rho_n x, v_{i,n}(\rho_nx))
	\Big|
	\leq
	C(\rho_n r_n)^2
	\|\tilde{v}_{i,n}\|_{L^\infty(B_2)}
	\rightarrow
	0
	$$
	By a proof like the one of Proposition \ref{concludeProp},  since $\|I-A_n(\rho_n y)\| \leq Cr_n\rho_n|y| \rightarrow 0$ and $\rho_n^2
		H(v_n,\rho_n)
		M_n \rightarrow -\infty$, we obtain:
	$$
	\Delta \tilde{v}_i(x) = 0\quad \text{     	for $x \in \{\tilde{v}_i>0\}$.}
	$$
Moreover, we have that $0 \in \{\tilde{v}=0\}$, since the sequence $(v_n(0))$ is bounded, while by Lemma \ref{doublingForAlt},
	$$
	H(v_n,\rho_n)
	\geq
	\frac{
		H(v_n,2)
	}{
		4^\sigma
	}
	\rho_n^{2\sigma}
	e^{Cr_n(2-\rho_n)}
	\rightarrow
	+\infty
	$$
	as $n \rightarrow \infty$ since $r_n\rho_n \rightarrow 0$.
	This shows that $\tilde{v}(0) 
	= \lim \tilde{v}_n(0)
	=\lim \frac{v_n(0)}{\sqrt{H(v_n,\rho_n)}} = 0$.
	
	Using Proposition \ref{theoremForNandFreeStuff}, since $\tilde{v}(0) = 0$, that is $0 \in \{\tilde{v} = 0\}$, we have that $1\leq N(\tilde{v},0^+)$. Also, since $\tilde{v} \in \mathcal{G}(B_{\frac{3}{2}})$, by Proposition \ref{homogenousSegregatedLemma}, we conclude that the function $N(\tilde{v},r)$ is monotone increasing for $r \in [0,\frac{3}{2}]$. Also by equation \eqref{upBoundOnN1236} we have that $N(\tilde{v},r)\leq 1$ for all $r \in [0,\frac{3}{2}]$. From this we conclude the chain of inequalities
	$$
	1
	\leq
	N(\tilde{v},0^+)
	\leq
	N(\tilde{v},r)
	\leq 
	1
	$$
	for all $r \in ]0,\frac{3}{2}[$. This implies that $N(\tilde{v},r)$ is constant equal to $1$, thus by Proposition \ref{homogenousSegregatedLemma} we conclude that $\tilde{v}$ is a homogenous function of degree $1$ at zero. By Theorem \ref{mult2Theoremyeah} we conclude that there must exist two nontrivial components of $\tilde{v}$ around zero, since it is the limit of $\tilde{v}_n$, solutions of competition systems. 
	
	By Lemma \ref{harmonic2compsystemCharacter} there must exist indices $h,k \in \{1,...,l\}$ and constants $\gamma_h, \gamma_k>0$ such that up to a rotation:
\[
	\tilde{v}_h(x)
	=
	\gamma_h x_1^+
	,\qquad
	\tilde{v}_k(x)
	=
	\gamma_k x_1^-
	,\qquad
	\tilde{v}_j(x)
	=
	0 \quad \forall j \neq h,k. \qedhere
\]
\end{proof}

\begin{lemma}
\label{sameWeightAltCafCond}
	There exists $\lambda > 0$ independent of $n$ such that:
	$$
	\frac{1}{\lambda}
	\leq
	\frac{
		\int_{\partial B_r} \mu_n(y)v_{1,n}^2
		d\sigma(y)
	}{
		\int_{\partial B_r} \mu_n(y)v_{2,n}^2 d\sigma(y)
	}
	\leq
	\lambda
	$$
	for every $2N/\theta^{\frac{1}{2}} \leq r \leq \frac{\overline{r}_n}{3}$. On the contrary for $j = 3,...,N$ we have:
	$$
	\sup_{r \in [2N/\theta^\frac{1}{2},\overline{r}_n/3]}
	\frac{
		\int_{\partial B_r}
		\mu_n(y)
		v_{j,n}^2
		d\sigma(y)
	}{
		\int_{\partial B_r} 
		\mu_n(y)
		v_{1,n}^2
		d\sigma(y)
	}
	\rightarrow 0
	$$
	as $n \rightarrow \infty$. This combined with Lemma \ref{nonTriviality} shows hypothesis $(h_3)$ of Section \ref{chapter:resultsChap4}. 
\end{lemma}
\begin{proof}
For this we use Lemma \ref{simplexLemma} from the appendix.
	Given the sequence $\{v_n\}$, we consider the auxiliary functions:
	$$
	g_{i,n}(\rho)
	:=
	\begin{cases}
		\frac{1}{H(v_n,\rho \overline{r}_n/3)}
		\frac{1}{(\rho \overline{r}_n/3)^{N-1}}
		\int_{\partial B_{\rho \overline{r}_n/3}}
		\mu_n(y) v_{i,n}^2
		d\sigma(y)
		& \text{ for } 
		\frac{6N/\theta^\frac{1}{2}}{\overline{r}_n}
		\leq
		\rho
		\leq
		1\\
		\frac{1}{H(v_n,2N/\theta^{\frac{1}{2}})}
		\frac{1}{(2N/\theta^\frac{1}{2})^{N-1}}
		\int_{
		\partial B_{2N/\theta^\frac{1}{2}}
		}
		\mu_n(y) v_{i,n}^2
		d\sigma(y)
		& \text{ for }
		0\leq \rho \leq 
		\frac
		{6N/\theta^\frac{1}{2}}
		{\overline{r}_n}.
	\end{cases}
	$$
	The proof is finished once we have proved the assumptions of Lemma \ref{simplexLemma}, that is $\lim_n \dist(g_n([0,1]),\Sigma_{2,l}) = 0$, where
    $
        \Sigma_{2,l}
	:=
	\Big\{
	x \in \mathbb{R}^l: 
	\exists i,j \in \{1,...l\},\ i\neq j,\ 
	\text{ such that }
	x_h = 0 \quad
	\forall h \neq i,j
	\Big\}.
    $
    By Lemma \ref{nonTriviality}, we must have that the indices for the nontrivial components should be $i=1,2$.
	
	By construction, each $g_{i,n}$ is continuous, $g_{i,n} \geq 0$, and $\sum_{i=1}^lg_{i,n}(x) = 1$ for all $x \in [0,1]$. We divide the proof into two steps:

\noindent	\textbf{(i)} First, we prove that there exists $\epsilon \in ]0,1[$ such that $g_{i,n}(x)\leq 1-\epsilon$ for all $x \in [0,1]$, $n \in \mathbb{N}$, $i \in \{1,...,l\}$. By contradiction, we assume there exists an index $i \in \{1,...,l\}$ and a sequence $s_n \in [0,1]$ such that:
	\begin{equation}
		\label{contradHypNonTriv}
			g_{i,n}(s_n) \rightarrow 1, \qquad
			g_{j,n}(s_n) \rightarrow 0
			\qquad
			i,j \in \{1,...,l\}
			\text{ and }
			i \neq j.
	\end{equation}
	By Lemma \ref{nonTriviality} and the local uniform convergence $v_n \rightarrow v$, we conclude that $s_n\overline{r}_n \rightarrow \infty$. Indeed, were this not true and \eqref{contradHypNonTriv} would not be possible, since if $s_n \overline{r}_n \rightarrow \tilde{r}$ then:
	$$
	    \lim_n
	    g_{i,n}(s_n)
	    =
	    \int_{\partial B_{\tilde{r}/3}}
	    \frac{
	    \mu_n(y)
	    v_{i}^2
	    }
	    {
	        H(v,\tilde{r}/3)
	        (\tilde{r}/3)^{N-1}
	    }
	    d\sigma(y)
	    >0
	$$
for $i=1,2$ in case $\tilde{r}\geq 6N/\theta^\frac{1}{2}$, while	in case $\tilde{r}< {6N}/{\theta^\frac{1}{2}}$ we also conclude that
	$$
	    \lim_n
	    g_{i,n}(s_n)
	    =
	    \int_{\partial B_{2N/\theta^\frac{1}{2}}}
	    \frac{\mu_n(y)
	    v_{i}^2
	    }
	    {
	        H(v,2N/\theta^\frac{1}{2})
	    (2N/\theta^\frac{1}{2})^{N-1}
	    }
	    d\sigma(y)
	    >0\qquad  \text{for $i=1,2$.}
	    $$	
We consider the blowdown sequence given by
	\begin{equation}\label{blowdownapplied}
	\tilde{v}_{i,n}(x)
	:=
	\frac{
		v_{i,n}(s_n\overline{r}_nx/3)
	}{
		\sqrt{H(v_n,s_n\overline{r}_n/3)}
	},
	\end{equation}
to which we apply Lemma \ref{blowDownSequence} with $\rho_n:=s_n\overline{r}_n/3$, concluding that the uniform limit of $\tilde{v}_n$ contains two two nontrivial components, in contradiction with (\ref{contradHypNonTriv}).
	
\noindent	\textbf{(ii)} Now we prove $\lim_{n} \dist(g_n([0,1]), \Sigma_{2,l}) = 0$. We assume by contradiction that there exists $\epsilon>0$ and three different indices $i,j,k$ and a sequence $s_n \in ]0,1[$ such that up, to a subsequence,
	\begin{align}
		\label{contrad2NonTriv}
		g_{i,n}(s_n)\geq \epsilon ,
		\quad\quad
		g_{j,n}(s_n)\geq \epsilon ,
		\quad\quad
		g_{k,n}(s_n)\geq \epsilon.
	\end{align} 
	Again, we must have $s_n\overline{r}_n \rightarrow \infty$, otherwise, since by Proposition \ref{concludeProp} the limit $v = \lim v_n$ has a maximum of two nontrivial components, we would have	
$  \lim_n
	    g_{i,n}(s_n)
	=0$ for $i \neq 1,2$.
	
	Exactly as before, considering again the blowdown sequence \eqref{blowdownapplied} and since $s_n\overline{r}_n/3 \rightarrow \infty$, we apply Lemma \ref{blowDownSequence} to conclude that the uniform limit of $v_n$ contains exactly two nontrivial components, in contradiction with equation \eqref{contrad2NonTriv}.
\end{proof}

\begin{lemma}
\label{LambdasMoreThanZeroAltCafCond}
	There exists $C>0$ independent of $n$ such that:
	$$
	\Lambda_{1,n}(r), \\ \Lambda_{2,n}(r)\geq C
	$$
	for $r \in [2N/\theta^\frac{1}{2},\frac{\overline{r}_n}{3}]$. In particular the second condition of $(h_6)$ holds true.
\end{lemma}
\begin{proof}
	By contradiction, we assume there exists $\rho_n \in [2N/\theta^\frac{1}{2},\overline{r}_n/3]$ such that $\lim_n\Lambda_{1,n}(\rho_n) \leq 0$, that is:
	\begin{equation}
		\label{LambdaContradiction}
		\lim_n
		\rho_n^2\frac{
			\int_{\partial B_{\rho_n}}
			\left(
			\langle
			B_n(x)\nabla_{\theta}v_{1,n},
			\nabla_{\theta} v_{1,n}
			\rangle
			-
			M_na_n(x)v_{2,n}^2 v_{1,n}^2
			+
			v_{1,n}f_{1,n}(x,v_{1,n})
			\right)
			d\sigma(x)
		}{
		\int_{\partial B_{\rho_n}} (1+\alpha \rho_n r_n)\mu_n(x) v_{1,n}^2
		d\sigma(x)
		}
		\leq
		0.
	\end{equation}
	We either have that $\rho_n$ is bounded or $\rho_n \rightarrow \infty$.

\noindent	\textbf{(i)} If $\rho_n \rightarrow \infty$, then we consider the scaled blowdown sequence:
	$$
	\tilde{v}_{i,n}(x)
	:=
	\frac{
		v_{i,n}(\rho_nx)
	}{
		\sqrt{H(v_n,\rho_n)}
	}
	$$
	where $\rho_n \leq \overline{r}_n/3$. From Lemma \ref{blowDownSequence} we know that $\tilde{v}_n \rightarrow \tilde{v}$ uniformly, such that (up to a rotation):
	$$
	\tilde{v}_i = \gamma_i x_1^+ ,\,\,
	\tilde{v}_j = \gamma_j x_1^-,\,\,
	\tilde{v}_k = 0
	$$
	for all $k \neq i,j$ and $\gamma_i, \gamma_j > 0$. Due to Lemma \ref{nonTriviality} we conclude that $i = 1$ and $j =2$.
	
	The idea is to turn the uniform convergence of $\tilde{v}_{1,n} \rightarrow \tilde{v}_1$ into a $C^{1,\alpha}$ convergence for $0<\alpha<1$ in a set away from the free boundary given by $\{\gamma_1x_1 > 2\tilde{\delta}\}$ for some $\tilde{\delta} > 0$. We take $\tilde{\delta}$ sufficiently small so that $\{\gamma_1x_1 > 2\tilde{\delta}\} \cap \partial B_1 \neq \emptyset$.
	
	If $x_0 \in B_2(0) \cap \{\gamma_1 x_1 >\tilde{\delta}\}$ and $\rho>0$ is small enough so that $B_\rho(x_0) \subset \{\gamma_1x_1>\tilde{\delta}/2\}$, then by uniform convergence of $\tilde{v}_{1,n}$ to $\tilde{v}_1$:
	\begin{equation}
	\label{lowerBoundForBlowDown}
	\tilde{v}_{1,n}\geq \frac{\tilde{\delta}}{4}>0 
	\qquad
	\text{ in } B_2\cap\{\gamma_1x_1>\tilde{\delta}/2\}.
	\end{equation}
	
	Equation \eqref{lowerBoundForBlowDown}, equation \eqref{blowDownEq} and \textbf{(a)} -- which provides  $a_n(x)\geq \delta$ --, imply the following inequality for $j\neq 1 = i$:
	\begin{align}
		-\div &(A_n(\rho_n x)\nabla \tilde{v}_{j,n})
		\leq
		\left(
		d\rho_n^2r_n^2
		-
		2
		Ca_n
		(\rho_n x)
		H(v_n,\rho_n)\rho_n^2
		|M_n|
		\right)
		\tilde{v}_{j,n}(x) \nonumber \\
		& \leq
		\left(
		d\rho_n^2r_n^2
		-
		2
		C
		\delta
		H(v_n,\rho_n)\rho_n^2|M_n|
		\right)
		\tilde{v}_{j,n}(x)
		\leq
		-C
		\delta
		H(v_n,\rho_n)\rho_n^2|M_n|\tilde{v}_{j,n}(x), 	    \label{boundOnDivergenceAltCafConditions}
	\end{align}
	since $\rho_nr_n \rightarrow 0$ and $H(v_n,\rho_n)\rho_n^2 M_n \rightarrow -\infty$ using Proposition \ref{concludeProp} and Lemma \ref{nonTriviality}.

	Applying Lemma \ref{estimateLemma}-1 to \eqref{boundOnDivergenceAltCafConditions}, we conclude the uniform bound: 
	$$
	|H(v_n,\rho_n)\rho_n^2M_n\tilde{v}_{j,n}(x_0)| \leq C \quad \text{	for every $x_0 \in B_{2-\rho}(0) \cap \{\gamma_1x_1>2\tilde{\delta}\}$},
	$$
 which implies the uniform boundedness of $\div(A_n(\rho_n x)\nabla \tilde{v}_{1,n})$. This together with the uniform convergence $\tilde{v}_{1,n} \rightarrow \tilde{v}_1$ implies that it also converges in $C^{1,\alpha}(B_{2-\rho}(0) \cap \{\gamma_1x_1>2\tilde{\delta}\})$ for all $0< \alpha <1$ by standard elliptic estimates.
	
	Now we reach contradiction since, using equation \eqref{LambdaContradiction}, $\rho_n\rightarrow \infty$ and $M_n \leq 0$, we have:
	\begin{align*}
		0
		&\geq
		\lim_n
		\rho_n^2\frac{
			\int_{\partial B_{\rho_n}}
			\left(
			\langle
			B_n(x)\nabla_{\theta}v_{1,n},
			\nabla_{\theta} v_{1,n}
			\rangle
			-
			M_na_n(x)v_{2,n}^2 v_{1,n}^2
			+
			v_{1,n}f_{1,n}(x,v_{1,n})
			\right)
			d\sigma(x)
		}{
			\int_{\partial B_{\rho_n}} 
			(1+\alpha \rho_n r_n)
			\mu_n(x)
			v_{1,n}^2
			d\sigma(x)
		}
		\\
		&\geq 
		\lim_n\frac{
			\int_{\partial B_1}
			\langle
			B_n(\rho_n x)
			\nabla_\theta \tilde{v}_{1,n},
			\nabla_\theta \tilde{v}_{1,n}
			\rangle
			d\sigma(x)
		}{
			\int_{\partial B_{1}}(1+\alpha \rho_nr_n)\mu_n(x)
			\tilde{v}_{1,n}^2
			d\sigma(x)
		}
		-
		\lim_n
		\rho_n^2
		\frac{
			\int_{\partial B_{\rho_n}}
			v_{1,n}f_{1,n}(x,v_{1,n})
			d\sigma(x)
		}{
			\int_{\partial B_{\rho_n}}(1+\alpha \rho_nr_n)\mu_n(x)v_{1,n}^2
			d\sigma(x)
		}\\
		& \geq
		\frac{
			\int_{\partial B_1\cap \{\gamma_1x_1>2\tilde{\delta}\}}
			|\nabla_\theta\big(\gamma_1 x_1\big)|^2
			d\sigma(x)
		}{
			\int_{\partial B_1}\big(\gamma_1x_1\big)^2
			d\sigma(x)
		}
-2d\rho_n^2r_n^2\geq 
		C
		>
		0,
	\end{align*}
where in the second to last inequality we used the $C^{1,\alpha}$--convergence of $\tilde{v}_n$,  the fact that $B_n(\rho_n x) \rightarrow I$ uniformly over compact sets, and the bound for $f_{i,n}$ given by Lemma \ref{lemmaMatrixBlowLimit}.
	
\noindent	\textbf{(ii)} In the case where $\rho_n$ is bounded, there exists $\overline{\rho}$ such that $\rho_n \rightarrow \overline{\rho}$.
	
	If $M_n \rightarrow -\infty$ then $v_n \rightarrow v$ where $v$ satisfies the system \eqref{unboundedEquation} of Proposition \ref{concludeProp}. Similarly to above we have that
	$$
	    	\lim_n
		\Big|
		\rho_n^2
		\frac{
			\int_{\partial B_{\rho_n}}
			v_{1,n}f_{1,n}(x,v_{1,n})
			d\sigma(x)
		}{
			\int_{\partial B_{\rho_n}}(1+\alpha \rho_nr_n)\mu_n(x)v_{1,n}^2
			d\sigma(x)
		}
		\Big|
		\leq
		2d\rho_n^2r_n^2
		\rightarrow 0
	$$
	and also,
	$$
	    \frac{
		-\int_{\partial B_{\overline{\rho}}}
		M_na_n(x)v_{2,n}^2 v_{1,n}^2
		d\sigma(x)
	}{
		\int_{\partial B_{\overline{\rho}}}
		(1+\alpha \rho_n r_n)\mu_n(x) v_{1,n}^2
		d\sigma(x)
	}
	\geq
	0.
	$$
	Thus we must have that:
	\begin{equation}
	\label{contradictComplication123}
	    \lim_n
		\frac{
			\int_{\partial B_{\overline{\rho}}}
			\langle
			B_n(x)
			\nabla_\theta 
			v_{1,n},
			\nabla_\theta 
			v_{1,n}
			\rangle
			d\sigma(x)
		}{
			\int_{\partial B_{\overline{\rho}}}
			(1+\alpha \rho_nr_n)
			\mu_n(x)
			v_{1,n}^2
			d\sigma(x)
		}
		=
		0
	\end{equation}
	By an argument similar to the one above in \textbf{(i)} we can conclude $C^{1,\alpha}$ convergence in sets where $\{v_1>0\}$, thus:
	$$
	    \lim_n
		\frac{
			\int_{\partial B_{\rho_n}}
			\langle
			B_n( x)
			\nabla_\theta v_{1,n},
			\nabla_\theta v_{1,n}
			\rangle
			d\sigma(x)
		}{
			\int_{\partial B_{\rho_n}}
			(1+\alpha \rho_nr_n)
			\mu_n(x)
			\tilde{v}_{1,n}^2
			d\sigma(x)
		}
		\geq
		\frac{
    	    \int_{\partial B_{\overline{\rho}}\cap \{v_1>0\}}
    	    |\nabla_\theta v_1|^2
    	    d\sigma(x)
	    }
	    {
    	    \int_{\partial B_{\overline{\rho}}}
    	    v_1^2
    	    d\sigma(x)
	    }.
	$$
	By Lemma \ref{nonTriviality} we must have that $\int_{\partial B_{\overline{\rho}}}v_1^2d\sigma(x) > 0$ and $\int_{\partial B_{\overline{\rho}}}v_2^2d\sigma(x)>0$ and also $v_1\cdot v_2 = 0$. This implies that the set $\{v_1>0\}$ is non-empty and that $|\nabla_\theta v_1|$ must be different from zero since otherwise $v_1$ would be constant different from zero in $B_{\overline{\rho}}$ and since $\int_{B_{\overline{\rho}}}v_2^2d\sigma(x) > 0$ and $v_1\cdot v_2 = 0$ this can't happen. Thus:
	$$
	    \frac{
    	    \int_{\partial B_{\overline{\rho}}\cap \{v_1>0\}}
    	    |\nabla_\theta v_1|^2
    	    d\sigma(x)
	    }
	    {
    	    \int_{\partial B_{\overline{\rho}}}
    	    v_1^2
    	    d\sigma(x)
	    }>0,
	$$
	in contradiction with \eqref{contradictComplication123}.

	On the other hand, if $M_n$ is bounded, then $v_n \rightarrow v$ in $C^{1,\alpha}_{loc}(\mathbb{R}^N)$. By Lemma \ref{nonTriviality} we know that both $v_{1}$ and $v_2$ are nonnegative nontrivial, and by the strong maximum principle we have $v_1, v_2 > 0$ in $\mathbb{R}^N$. This implies that:
	$$
	\frac{
		-\int_{\partial B_{\overline{\rho}}}
		M_na_n(x)v_{2,n}^2 v_{1,n}^2
		d\sigma(x)
	}{
		\int_{\partial B_{\overline{\rho}}}
		(1+\alpha \rho_n r_n)\mu_n(x) v_{1,n}^2
		d\sigma(x)
	}
	\geq
	C
	>
	0
	$$
	and this allows us to reach contradiction.
\end{proof}

\begin{lemma}
\label{almgrenLemmaHypoithesisStuff}
    There exists $C$ such that, for $r,s \in ]0,R_n[ = ]0,\frac{\overline{r}_n}{3}[$ such that $r \leq s$, then
    \begin{equation*}
        \frac{1}{r^{N-1}}
        \int_{\partial B_r}
        v_{i,n}^2
        d\sigma(y)
    \leq
        \frac{C}{s^{N-1}}
        \int_{\partial B_s}
        v_{i,n}^2
        d\sigma(y).
    \end{equation*}
    In particular, this proves $(h_5)$.
\end{lemma}
\begin{proof}
    Using Lemma \ref{MonotonicityForBlowup} there exists $\tilde{C}>0$ such that for each $i \in \{1,...,l\}$, the function:
    $$
        r
        \mapsto
        \left(
            \frac{1}{r^{N-1}}
            \int_{\partial B_r}
            \mu_n(y)
            v_{i,n}^2
            d\sigma(y)
        \right)
        e^{\tilde{C}r_nr}
    $$
    is monotone nondecreasing for $r \in ]0,R_n[\subset ]0, \frac{\tilde{r}}{r_n}[$.  Using the matrix bounds from Lemma \ref{lemmaMatrixBlowLimit} and $r_n\overline{r}_n\rightarrow 0$, we conclude $\frac{\theta}{M} \leq \mu_n(y) \leq Cr_nR_n\leq C'$. Thus, given $r,s \in ]0, R_n[$ and $r < s$, we conclude:
    \begin{align*}
        \frac{1}{r^{N-1}}
        \int_{\partial B_r}
        v_{i,n}^2
        d\sigma(y)
    &\leq
        \frac{1}{\theta}
        \frac{1}{r^{N-1}}
        \int_{\partial B_r}
        \mu_n(y)
        v_{i,n}^2
        d\sigma(y)
    \leq
        \frac{1}{\theta}
        \left(
            \frac{1}{s^{N-1}}
        \int_{\partial B_s}
        \mu_n(y)
        v_{i,n}^2
        d\sigma(y)
        \right)
        e^{\tilde{C}r_n(s-r)}\\
    &\leq
        \frac{M}{\theta}
        \left(
            \frac{1}{s^{N-1}}
        \int_{\partial B_s}
        v_{i,n}^2
        d\sigma(y)
        \right)
        e^{\tilde{C}r_n\overline{r}_n/3}
    \leq
        C
        \left(
            \frac{1}{s^{N-1}}
        \int_{\partial B_s}
        v_{i,n}^2
        d\sigma(y)
        \right)
    \end{align*}
    since $r_n\overline{r}_n \rightarrow 0$, taking $C = \sup_{n}\frac{M}{\theta}e^{\tilde{C}r_n\overline{r}_n}$.
\end{proof}

It remains to show that also $J_{1,n}(r)$ and $J_{2,n}(r)$ are positive in the whole range $[2N/\theta^\frac{1}{2},\frac{\overline{r}_n}{3}]$, which is condition ($h_6$).

\begin{lemma}
\label{LastLemma}
We have that:
$$
    J_{i,n}(r)
>
    0
    \qquad
    \forall r \in [2N/\theta^\frac{1}{2}, \frac{\overline{r}_n}{3}],
$$
for all $n \in \mathbb{N}$ and $i=1,2$. In particular, this together with Lemma \ref{LambdasMoreThanZeroAltCafCond} implies that $(h_6)$ holds true. Also there exists $c>0$ such that $J_n(2N/\theta^\frac{1}{2}) 
= 
\frac
{
J_{1,n}(2N/\theta^\frac{1}{2})
J_{2,n}(2N/\theta^\frac{1}{2})
}
{(2N/\theta^\frac{1}{2})^4}
>
c.
$
\end{lemma}
\begin{proof}
    First of all, there exists $\overline{C}>0$ such that $J_{i,n}(r)\geq \overline{C}$ for every $r \in [2N/\theta^\frac{1}{2}, 10N/\theta^\frac{1}{2}]$ and $i=1,2$. This is a consequence of $v_{i,n} \rightarrow v_i$ in $C(B_{10N/\theta^\frac{1}{2}})\cap H^1(B_{10N/\theta^\frac{1}{2}})$, and $v_i$ (in particular $v_1$) is nonconstant in $B_{10N/\theta^\frac{1}{2}}$, and that $f_{i,n}(x,v_{i,n}) \rightarrow 0$ uniformly in $B_{10N/\theta^\frac{1}{2}}$ by Proposition \ref{ListProp} and $M_n$. This also proves the last part of the statement of the lemma.
    
    Define:
    $$
        s_n
    :=
        \sup
        \Big\{
            s \in ]2N/\theta^\frac{1}{2}, \overline{r}_n/3[:
            J_{i,n}(r)
            >
            0
            ,
            \text{ for every }
            r \in ]2N/\theta^\frac{1}{2},s[
        \Big\}.
    $$
    We wish to prove that $s_n = \overline{r}_n/3$. Using Lemmas \ref{easyLemmaForConditions123}, \ref{sameWeightAltCafCond} and \ref{LambdasMoreThanZeroAltCafCond}, the definition of $s_n$, the definitions of the constants $\epsilon_n = dr_n^2$ and $c_n = Cr_n$, all conditions $(h_0)$-$(h_6)$ of Section \ref{chapter:resultsChap4} are satisfied in the interval $]2N/\theta^\frac{1}{2},s_n[$, so by Theorem \ref{AltCaffMonotonicity} there exists $0<\eta<1$ and $C>0$ such:
    $$
        r
    \mapsto
        \frac
        {J_{1,n}(r)J_{2,n}(r)}
        {r^4}
        e^{
            -C|M_n|^{-\eta}
             r^{-2\eta}
        +
            Cr_n^2r^2
        +
            Cr_nr
        }
    $$
    is monotone nondecreasing for $r \in ]2N/\theta^\frac{1}{2}, s_n[$.
    Thus this implies that for all $s \in ]0,s_n[$ we have:
    \begin{align*}
        J_{1,n}(r)J_{2,n}(r)
    &=
        r^4J_n(r)\geq
        r^4
        J_n(2N/\theta^\frac{1}{2})
        e^{
            -C|M_n|^{-\eta}
        +
            Cr_n^2
        +
            Cr_n
        +
            C|M_n|^{-\eta}
            r^{-2\eta}
        -
            C
            r_n^2
            r^2
        -
            Cr_nr
        }\\
    &\geq
        \overline{C}^2
        e^{-C|M_n|^{-\eta}
        -
            C
            r_n^2
            r^2
        -
            Cr_nr},
    \end{align*}
    for all $r \in [2N/\theta^\frac{1}{2}, R_n]$ we have $r_nr \rightarrow 0$ and by Proposition \ref{concludeProp} there exists $\epsilon>0$ such that $|M_n|\geq \epsilon$, thus $|M_n|^{-\frac{\gamma}{2(\gamma+1)}} \leq \epsilon^{-\frac{\gamma}{2(\gamma+1)}}$. We conclude there exists $\tilde{c}>0$ such that $e^{-C|M_n|^{-\eta}
        -
            C
            r_n^2
            r^2
        -
            Cr_nr} \geq C,
    $
    and so:
    \begin{equation}
    \label{finalEquationForChapter}
        J_{1,n}(r)J_{2,n}(r)
    \geq
        \overline{C}^2C
    >
        0.
    \end{equation}
    By continuity of $J_{i,n}(r)$, using \eqref{finalEquationForChapter}, we have that $J_{i,n}(r)>0$ for all $r \in [0,s_n]$.
    This implies that there exists $\tilde{\epsilon}>0$ such that for $s_n$ there exists $J_{1,n}(s_n)>\tilde{\epsilon}$ and $J_{2,n}(s_n)>\tilde{\epsilon}$, and by the continuity of $J_{i,n}$ we conclude that for $\tilde{\delta}>0$ small enough we have $J_{1,n}(s_n+\tilde{\delta})>0$ and $J_{2,n}(s_n+\tilde{\delta})>0$ in contradiction with the definition of $s_n$ in case $s_n < R_n$. Thus we conclude that $s_n = R_n = \frac{r_n}{3}$
\end{proof}

\begin{proof}[Conclusion of the proof of Lemma \ref{caffMonot}]
With all the conditions $(h_0)$-$(h_6)$ satisfied by the sequence $\{v_n\}$ in the interval $[2N/\theta^\frac{1}{2},\overline{r}_n/3]$, we can apply Theorem \ref{AltCaffMonotonicity} to $\{v_n\}$,  conclude the validity of Lemma \ref{caffMonot}.
\end{proof}

\appendix

\section{Auxiliary Results}

In this appendix we state some auxiliary results, which are used in the course of this work.
\subsection{Divergence operator on the Sphere}
\label{chapter:sphereDivergence}

In this subsection we compute a divergence operator of the sphere $\partial B_1$ of dimension $N-1$, in terms of the divergence in $\mathbb{R}^{N-1}$, using the stereographic projection.

First fix a Riemannian manifold $(M,g)$ of dimension $\dim (M) = N-1$ with metric $g$, a coordinate system $\phi: \mathbb{R}^{N-1} \rightarrow M$, and a vector field $X$ that is $(X_1,...,X_{N-1})$ in the $\phi$ coordinates.

We then have that the divergence in $(M,g)$ is given by:
\begin{equation}
\label{divergenceInCoordinates}
    \div_g(X)
=
    \sum_{i=1}^{N-1}
    \frac{1}{\sqrt{\det g}}
    \frac
    {\partial}
    {\partial x_i}
    (\sqrt{\det g} X_i)
\end{equation}

\begin{prop}
    \label{DivergenceSphereProp}
    Let:
    \begin{itemize}
    \item $\phi : \mathbb{R}^{N-1} \rightarrow \partial B_1$ be the stereographic projection, $\phi(y) = (\frac{2y}{1+|y|^2}, \frac{|y|^2-1}{1+|y|^2})$;
    \item  $B(y): T_y \partial B_1 \rightarrow T_y \partial B_1$ a differentiable hermitian operator satisfying
        \begin{equation}
        \label{PreBoundsOfMatrixInSphere}
        \langle
            B(y)v,
            v
        \rangle_{\partial B_1}
    \geq
        \theta \langle
            v,
            v        \rangle_{\partial B_1}
        \,\,
        \forall v \in T_y \partial B_1,
        \qquad
        \|DB(y)\| \leq M;
    \end{equation}
    \item $M(y) = (d\phi)^{-1}_{\phi(y)}B(\phi(y))(d\phi)_y$.
    \end{itemize}
Let $u : \partial B_1 \rightarrow \mathbb{R}$ be a differentiable function and   take $\tilde{u}(y) = u(\phi(y))$. Then
    \begin{equation*}
        \div_{\partial B_1}
        (B(z)
        \nabla_\theta u)|_{z = \phi(y)}
    =
        (1+|y|^2)^{N-1}
        \div_{\mathbb{R}^{N-1}}
        (
            \frac
            {1}
            {4(1+|y|^2)^{N-3}}M(y)
            \nabla_{\mathbb{R}^{N-1}}
            \tilde{u}
        ).
    \end{equation*}
Moreover, given a compact set $K \subset B_R \subset \mathbb{R}^{N-1}$, and the constants $\theta > 0$ and $M$ from \eqref{PreBoundsOfMatrixInSphere}, there exists $C = C(K,M)$ such that:
    $$
    \left    \langle
            \frac
            {
                1
            }
            {
                4(1+|y|^2)^{N-3}
            }
            M(y)
            \xi,
            \xi
      \right  \rangle
        \geq
        \frac{1}
        {4(1+R^2)^{N-3}}
        \theta
       \left \langle
            \xi,
            \xi
        \right\rangle,\qquad
       \left\|D(\frac{1}{4(1+|y|^2)^{N-3}}            M(y))        \right\|
    \leq
        C,
    $$
    for all
  $ \xi \in  \mathbb{R}^{N-1},    y \in K.$
    
\end{prop}

\begin{proof}
This follows from directly computations, using formula \eqref{divergenceInCoordinates}. A detailed proof can be found in  \cite[Appendix A]{DiasTavares}. 
\end{proof}

\subsection{Results for functions in the class $\mathcal{G}(\Omega)$}\label{appendix:classG}

 The second author, jointly with S. Terracini, introduced in \cite{HugoTerraciniWeakReflectionLaw} the following set of vector valued functions $\mathcal{G}(\Omega)$. This set has a relation with the blowups of competitive systems; moreover, they also satisfy an  Almgren-monotonicity formulas (see Theorem \ref{theoremForNandFreeStuff} below). We give the definitions and some lemmas along with the respective reference.

\begin{definition}[{\cite[Definition 1.2]{HugoTerraciniWeakReflectionLaw}}]
\label{GspaceDef}
	Given $\Omega \subset \mathbb{R}^N$ an open set, we define $\mathcal{G}(\Omega)$ as the set of nontrivial vector valued functions $v = (v_1,...,v_l)$ whose components are nonnegative, locally Lipschitz continuous in $\Omega$ and such that:
	\begin{itemize}
		\item 
		$v_iv_j = $ for all $i \neq j$
		\item 
		for every $i$,		$-\Delta v_i = f_i(x,v_i) - \mu_i$ in $\Omega$ in the distributional sense, where $\mu_i$ is a nonnegative Radon-Measure supported on the set $\partial \{v_i>0\}$, and $f_i : \Omega \times \mathbb{R}^+ \rightarrow \mathbb{R}$ are $C^1$ functions such that $|f_i(x,s)| \leq d|s|$ uniformly in $x$.
		\item
		For $x_0 \in \Omega$, $r_0>0$ such that $B_{r_0}(x_0) \subset \Omega$, let
		\begin{equation*}
			E(v,x_0,r)
			:=
			\frac{1}{r^{N-2}}\left(
			\int_{B_r(x_0)}
			\sum_{i=1}^l
			|\nabla v_i|^2
			-
			\sum_{i=1}^l
			f_i(x,v_i)v_i
			\right).
		\end{equation*}
		We assume that $E(v,x_0,r)$ is absolutely continuous with respect to $r \in ]0,r_0[$, and that the derivative satisfies:
		\begin{align*}
		    \frac{d}{dr}
		    &E(v,x_0,r)
		=
		    \frac{1}{r^{N-2}}
		    \int_{B_r(x_0)}
		    \sum_{i=1}^l
		    \left(
		    \partial_\nu v_i
		    \right)^2
		    dx
		-
		    \frac{1}{r^{N-2}}
		    \int_{\partial B_r(x_0)}
		    \sum_{i=1}^l
		    f_i(x,v_i)v_i
		    dx\\
		&+
		    \frac{1}{r^{N-1}}
		    \int_{B_r(x_0)}
		    \left(
		        (N-2)
		        \sum_{i=1}^l
		        f_i(x,v_i)v_i
		    +
		        2
		        \sum_{i=1}^l
		        f_i(x,v_i)\nabla v_i\cdot(x-x_0)
		    \right)dx.
		\end{align*}
	\end{itemize}
\end{definition} 
Define, as before,
\begin{align*}
    H(v,x_0,r)
    =
    \sum_{i=1}^l
    \int_{\partial B_r}
        v_i
        d\sigma(x),\qquad 
    N(v,x_0,r)
    =
    \frac{E(v,x_0,r)}{H(v,x_0,r)}.
\end{align*}

\begin{thm}[{\cite[Theorem 2.2]{HugoTerraciniWeakReflectionLaw}}]
    \label{theoremForNandFreeStuff}
	Let $v \in \mathcal{G}(\Omega)$ and let $K \subset\subset \Omega$. There exists $\tilde{r}'$, $\tilde{C}'$ depending only on $d$ and on the dimension $N$, such that for every $x_0 \in K$ and $r \in ]0, \tilde{r}']$ it results that $H(v,x_0,r) \neq 0$ the function $N(v,x_0,r)$ is absolutely continuous in $r$ and:
	$$
	r
	\mapsto
	\left(
	N(v,x_0,r) + 1
	\right)
	e^{\tilde{C}'r}
	$$
	is monotone nondecreasing. Moreover, for every point of the free boundary $x_0 \in \{v=0\}$, we have $N(v,x_0, 0^+)\geq 1$.
\end{thm}

\begin{prop}[{\cite[Remark 2.4]{HugoTerraciniWeakReflectionLaw}}]
\label{homogenousSegregatedLemma}
    Let $v \in \mathcal{G}(\Omega)$ with $f_i = 0$ for every $i=1,...,l$. Then $r \mapsto N(v,x_0,r)$ is nondecreasing. Moreover, it holds $N(v,x_0,r) = \sigma$ for all $r \in [0,\overline{r}[$ if and only if $v$ is a nontrivial homogenous function of degree $\sigma$.
\end{prop}

Adapting slightly the proof of  of \cite[Theorem C.1]{HugoHolderVariable} (where $A_n\equiv A$ does not depend on $n$), in the spirit of \cite[Theorem 1.5]{SOAVE2016388} and \cite[Theorem 3.3]{HugoTerraciniWeakReflectionLaw} we can conclude the following proposition, which is a refinement of Theorem C.
\begin{prop}
\label{propOfBlowUpInNiceSpace}
Let $A_n$ be a sequence of matrices satisfying the bounds \textbf{(A1)}, \textbf{(A2)} uniformly and $A_n(x) \rightarrow Id$ uniformly over compact sets and $m>0$. Consider also a sequence of function $(f_{i,n}$ satisfying \textbf{(F)} and assume:
    $$
        f_{i,n} \rightarrow f_{i}
        \qquad
        \text{ in }
        C_{loc}(\Omega \times [0,m]).
    $$
    Let $k_n$ be a sequence such that $k_n \rightarrow -\infty$. 
    If $u_n$ is a sequence of solutions non negative solutions ($u_n \geq 0$) of the system:
    \begin{equation*}
    \label{limitSystemForSegregatedLimit}
        -\div(A_n(x) \nabla u_{i,n})
    =
        f_{i,n}(x, u_{i,n}(x))
    +
        k_n
        u_{i,n}^\gamma
        \sum_{j\neq i}
        u_{j,n}^{\gamma+1},
    \end{equation*}
    satisfying $\|u_n\|_{L^\infty(\Omega)}\leq m$. Then there exists $u \in \mathcal{G}(\Omega)$ such that up to a subsequence:
    $$
        u_n
    \rightarrow
        u
    \quad
    \text{ in }
    C_{loc}(\Omega)\cap H^1(\Omega).    
    $$
\end{prop}

\subsection{Excluding points of multiplicity 1} \label{chapter:mult1PointsApp}

Then the following result holds true.

\begin{thm}
    \label{mult2Theoremyeah}
    Let $u_n = (u_{1,n},...,u_{l,n}) \in C(B_R, \mathbb{R}^l)$ be a sequence of nonnegative functions, having a uniform bound $\|u_n\|_{L^\infty(B_R)} \leq m$, for some $m>0$, and satisfying:
    $$
        -\div(A_n(x)\nabla u_{i,n})
    =
        k_n
        u_{i,n}
        \sum_{j\neq i}
        u_{j,n}^2
        +
        f_{i,n}(x,u_{i,n}),
    $$
    where:
    \begin{itemize}
 \item $A_n \in C(B_R, \text{Sym}^{N\times N})$ is a sequence of matrices satisfying conditions \textbf{(A1)} and \textbf{(A2)} uniformly, and $A_n \rightarrow Id$ locally uniformly;
\item for each $i \in \{1,...,l\}$ let $f_{i,n} \in C(B_R\times \mathbb{R}, \mathbb{R})$ is a sequence of functions satisfying condition \eqref{boundForF}, and there exists let $f_i \in C(B_R \times \mathbb{R}, \mathbb{R})$ satisfying $f_{i,n} \rightarrow f_i$ locally uniformly;
\item  $k_n\rightarrow \infty$. 
\end{itemize}
Assume moreover that $u_{i,n} \rightarrow u = (u_1,...,u_l) \in \mathcal{G}(B_R)$. Then any $x_0 \in \{x: u(x) = (u_1(x),...,u_l(x)) = 0\}$ has at least multiplicity  2, in the sense that
	$$
	\#\Big\{
	i \in \{1,...,l\}:
	\text{meas}\{B_r(x_0) \cap \{u_i > 0\}\} 
	> 0 
	\text{ for every } r > 0
	\Big\}\geq 2.
	$$

\end{thm}

The proof is very similar to \cite{MULT1DANCERWANG}, which deals with the Laplace operator. Hence, we ommit it. The interested reader can find the details of  the proof of Theorem \ref{mult2Theoremyeah} in the (extended) arXiv version of the current paper, namely in \cite[Appendix B]{DiasTavares}.

\subsection{Other results}
\label{chapter:AuxResultsApp}
Below are some auxiliary results used in the thesis. Some of them are proved while others are simply referenced.

\begin{lemma}
	\label{harmonicLemma}
	If $u$ is an harmonic function in $B_{2NM}(0) \subset \mathbb{R}^N$ and:
	$$u(0)=1, \qquad |\nabla u(0)| \geq \frac{1}{M},$$
	then $u$ changes sign in $B_{2NM}(0)$.
\end{lemma}
\begin{proof}
Suppose, by contradiction, that $u$ does not change sign in $B_{2NM}(0)$ and so $u\geq 0$. Without loss of generality, assume that  $-x_1$ is the direction of the derivative at zero with norm $-\frac{\partial}{\partial x_1}u=|\nabla u(0)|=\alpha \geq \frac{1}{M}$ (otherwise one can rotate the domain, which does not change the fact that the function is harmonic). Then
	\begin{align}
		\label{localIneqHarmonic}
		v(x) := u(x) - 1 + \alpha x_1\geq -1+\alpha x_1
	\end{align}
satisfies $v(0) = 0$ and $\nabla v(0) = 0$, and it is harmonic, therefore
	\begin{equation}
	\label{harmonic1}
		\int_{\partial B_{2NM}(0)}
		v(x)
		d\sigma(x)
		=
		0.
	\end{equation}
	We also have that the function $\partial_{x_1}v$ is also harmonic, and $\partial_{x_1}v(0) = 0$;  so, using integration by parts,
	\begin{equation}
	\label{harmonic2}
		0 
		= 
		\int_{B_{2NM}(0)} 
		v_{x_i}(x)dx 
		= 
		\int_{\partial B_{2NM}(0)} 
		v(x)
		(\nu_{1})_x
		d\sigma(x).
	\end{equation}
	Summing up both equations \eqref{harmonic1} and \eqref{harmonic2} we have:
	\begin{align*}
		\int_{\partial B_{2NM}(0)}
		v(x)(1+(\nu_1)_x)d\sigma(x) = 0.
	\end{align*}
	Now notice that $(1+(\nu_1)_x)\geq 0$ for all $x$, thus we can use inequality \eqref{localIneqHarmonic} to obtain:
	\begin{align}
		\label{localRef}
		\begin{split}
			0
			&=
			\int_{\partial B_{2NM}(0)}
			v(x)
			(1+(\nu_1)_x)
			d\sigma(x) \geq 
			\int_{\partial B_{2NM}(0)}
			(-1+\alpha x_1)
			(1+\frac{x_1}{|x|})
			d\sigma(x)\\
			&=
			-|\partial B_{2NM}(0)| 
			+ 
			\alpha 
			\int_{\partial B_{2NM}(0)}
			\frac{
				x_1^2
			}{
				|x|
			}d\sigma(x).
		\end{split}
	\end{align} 
    By a symmetry argument:
	\begin{align}
		2NM|\partial B_{2NM}(0)|
		&=
		\int_{\partial B_{2NM}(0)}
		|x|
		d\sigma(x)
		=
		\int_{\partial B_{2NM}(0)}
		\sum_{i=1}^N
		\frac{|x_i|^2}{|x|}
		d\sigma(x)
		\geq
		N\int_{\partial B_{2NM}(0)}
		\frac{|x_1|^2}{|x|}
		d\sigma(x).
		\label{harmonic3}
	\end{align}
	Substituting equality \eqref{harmonic3} into \eqref{localRef}, and using also $\alpha \geq \frac{1}{M}$, we obtain
	\begin{align*}
		0 \geq -|\partial B_{2NM}(0)| +  \alpha2M |\partial B_{2NM}(0)|\geq		|\partial B_{2NM}(0)|>0
	\end{align*}
	which is a contradiction, concluding the proof.
\end{proof}

\begin{lemma}[Poincar\'e's Inequality]
	\label{PoincareIneq}
	If $u \in H^1_{loc}(\mathbb{R}^N)$, then:
	$$
	\frac{1}{r^{N-2}}
	\int_{B_r}
	|\nabla u|^2dx
	+
	\frac{1}{r^{N-1}}
	\int_{\partial B_r}
	u^2 d\sigma(x)
	\geq 
	\frac{N-1}{r^N}
	\int_{B_r}
	u^2dx
	$$
\end{lemma}
\begin{proof}
To prove it for $r = 1$ given $u \in H^1_{loc}(\mathbb{R}^N)$, simply apply the divergence theorem to $xu^2$. The result for a general $r$ follows from a scaling argument.
\end{proof}

\begin{lemma}[{{\cite[Lemma 4.14]{SoaveZilio}}}]
	\label{simplexLemma}
	Let $\epsilon \in ]0,1[$ and $g_n \in C([0,1], \mathbb{R}^l)$ be a sequence of continuous functions such that:
	$$
	g_n([0,1])
	\subset
	\Big\{
	x \in \mathbb{R}^l:
	x_i\geq 0,
	x_i\leq 1-\epsilon,
	\sum_{i=1}^l x_i=1
	\Big\}
	$$ 
	If $\displaystyle \lim\dist(g_n([0,1]), \Sigma_{2,l}) = 0$ holds true, where 
	\[
	\Sigma_{2,l}
	:=
	\Big\{
	x \in \mathbb{R}^l: 
	\exists i,j \in \{1,...l\},\ i\neq j,\ 
	\text{ such that }
	x_h = 0 \quad
	\forall h \neq i,j
	\Big\},
	\] 
	then, up to a subsequence, there exists $i \neq j$ such that 
\[
	\frac{\epsilon}{2}	<g_{i,n}(x),g_{j,n}(x) <1-\frac{\epsilon}{2} \text{ for $n$ sufficiently large, }\quad \text{ and }\quad g_{h,n}\rightarrow 0 \text{ uniformly for $h \neq i,j$}.
	\]

%The following two lemmas are taken from \cite[Lemma B.3]{HugoHolderVariable}.
\begin{lemma}
\cite[Lemma B.3]{HugoHolderVariable}
	\label{estimateLemma}
	Let $A(\cdot) \in C^1(B_{2r}, \text{Sym}^{N\times N})$ satisfy \textbf{(A1)} and \textbf{(A2)}, that is, there exist $\theta >0, M>0$, $a_0$ such that
	$$
	\theta |\xi|^2	
	\leq 
	\langle A(x)\xi,\xi \rangle 
	\quad
	\forall x \in B_{2r},
	\xi \in \mathbb{R}^N
	,\qquad
	\sup_{x\in B_{2r}}
	\|A(x)\|
	\leq
	M
	,\qquad
	\sup_{x \in B_{2r}}\|DA(x)\| 
	\leq
	a_0.
	$$
	\begin{enumerate}
\item 	Given $C>0$, $\delta > 0$ and $\gamma\geq 1$, if $u \in H^1(B_{2r}) \cap C^0(\overline{B}_{2r})$ is a nonnegative function that satisfies:
	$$-\div(A(x)\nabla u)\leq -Cu^\gamma + \delta \qquad \forall x \in B_{2r},$$
	then there exists $c>0$, depending only on $N$, $\theta$, $M$, $a_0$ such that:
	$$C\|u\|_{L^\infty(B_r)}^\gamma \leq \frac{c}{r+r^2} \|u\|_{L^\infty(B_{2r})} + \delta.$$
\item Given $C>0$, if $u \in H^1(B_{2r})\cap C^0(\overline{B_{2r}})$ be a nonnegative solution of:
	$$
	-\div(A(x)\nabla u) \leq -Cu
	\qquad \forall x \in B_{2r},
	$$
	then there exist constants $c_1, c_2>0$ depending only on $N$, $\theta$, $M$, $a_0$ such that:
	$$
	\|u\|_{L^\infty(B_r)} \leq c_1\|u\|_{L^\infty(B_{2r})}e^{-c_2r\sqrt{C}}.
	$$
	\end{enumerate}
\end{lemma}

\end{lemma}

We now will state two results taken from \cite[Corolaries-2.8,2.10]{SoaveZilio}, and proved in the last section of \cite{TerraciniHolderBoundsOG}. For Lemma \ref{equationNonTrivial} one can also see \cite[Lemma A.3.]{SOAVE2016388}.

\begin{lemma}
	\label{subHarmonicNonTrivial}
	Let $v=(v_1,...,v_l) \in C(\mathbb{R}^N) \cap H^1_{loc}(\mathbb{R}^N)$ such that each component is nonnegative and subharmonic and $v_i\cdot v_j = 0$ for $i \neq j$. If $v$ has a bound
	$$
	    \max_{i=1,...,l}
	    \mathop{\sup_{x\neq y}}_{x,y \in \mathbb{R}^N}
	    \frac{|v_i(x)-v_i(y)|}{|x-y|}
	 <
	    \infty,
	$$
	then $v$ can have at most two nontrivial components.
\end{lemma}

\begin{lemma}
	\label{equationNonTrivial}
	Let $(v_1,...,v_l)  \in C(\mathbb{R}^N) \cap H^1_{loc}(\mathbb{R}^N)$ be a vector function such that each component is positive, and satisfies the system:
	$$
	-\Delta v_i = C v_i^{\gamma}\sum_{i \neq j}^l v_j^{\gamma+1}.
	$$
	If $v$ has a bound
	$$
	    \max_{i=1,...,l}
	    \mathop{\sup_{x\neq y}}_{x,y \in \mathbb{R}^N}
	    \frac{|v_i(x)-v_j(y)|}{|x-y|}
	 <
	    \infty,
	$$
	then $v$ has at most two nontrivial components.
\end{lemma}

\begin{lemma}
    \label{harmonic2compsystemCharacter}
    Let $u = (u_1,...,u_l) \in C(\overline{B_1})$ be a homogenous function of degree $1$ at zero with two nontrivial components. Assume that $u$ also satisfies the system of equations:
    \begin{align}
        \label{systemOfEquationForRandomLemmaHarmonicuhh}
       -\Delta u_i = 0 \text{in}\quad \{u_i>0\},\qquad u_iu_j = 0 \quad \forall i,j \in \{1,...,l\} \text{ and } i \neq j,\qquad u_j \geq 0 \quad \forall j \in \{1,..,l\}.
    \end{align}
    Then there exist indices $h,k \in \{1,...,l\}$ and constants $\gamma_h,\gamma_k\geq 0$ such that, up to a rotation:
    $$
        v_h(x)
    =
        \gamma_h x_1^+,
    \qquad
        v_k(x)
    =
        \gamma_k x_1^-,
    \qquad
        v_j(x) = 0
    \quad
        \forall j \neq h,k.
    $$
\end{lemma}
\begin{proof}
    Since $u$ is homogenous of degree 1, for each $i \in \{1,...,l\}$ there exists $f_i \in C(\partial B_1)$ such that $u_i(x) = |x|f_i(\frac{x}{|x|})$. Computing the Laplacian in spherical coordinates we obtain:
    $$
        \Delta u_i
    =
        \frac
        {\partial^2 u_i}
        {\partial r^2}
        +
        \frac{N-1}
        {r}
        \frac{\partial u_i}
        {\partial r}
        +
        \frac{1}{r^2}
        \Delta_\theta u_i
    =
        \frac{1}{r}
        \left(
            (N-1)f
            +
            \Delta_\theta
            f
        \right)
    $$
    where $\Delta_\theta$ is the Laplace-Beltrami operator on the sphere. Thus by the equation \eqref{systemOfEquationForRandomLemmaHarmonicuhh} we must have $-\Delta_\theta f_i(z) = (N-1)f_i(z)$ for $z \in \{f_i >0\}$. Notice that $f_i f_j = 0$ for $i \neq j$. If $u$ has two nontrivial components, say $u_h,u_k$ for $h,k \in \{1,...,l\}$, then the sets $\{f_h>0\}$ and $\{f_k>0\}$ are two disjoint sets on the sphere such that:
    $$
        \lambda(\{f_h>0\})
        =
        \lambda(\{f_k>0\})
        =
        (N-1)
    $$
    where $\lambda(\Omega)$ for $\Omega \subset \partial B_1$ is the value of the first Dirichlet eigenvalue of the Laplace-Beltrami operator. This is only possible if for some unit vector $e_1 \in \mathbb{R}^N$, the sets are given by:
    $$
        \{f_h>0\}
        =
        \{
        z \in\partial B_1 : 
        \langle z, e_1 \rangle>0
        \}
        ,
        \qquad
        \{f_k>0\}
        =
        \{
        z \in\partial B_1 : 
        \langle z, e_1 \rangle<0
        \}
        ,
        \qquad
    $$
    and there exist $\gamma_h, \gamma_k>0$ such that:
    $$
        f_h(z)
        =
        \gamma_h
        \langle z, e_1 \rangle,
        \qquad
        f_k(z)
        =
        \gamma_k
        \langle z, e_1 \rangle.
    $$
\end{proof}

\textbf{Acknowledgments.} Manuel Dias and Hugo Tavares are partially supported by the Portuguese government through FCT - Funda\c c\~ao para a Ci\^encia e a Tecnologia, I.P., under the projects\newline  UID/MAT/04459/2020 and PTDC/MAT-PUR/1788/2020.

%\bibliographystyle{plain}
%\bibliography{global}

\begin{thebibliography}{10}

\bibitem{AkAn}
Nail Akhmediev and Adrian Ankiewicz.
\newblock Partially coherent solitons on a finite background.
\newblock {\em Phys. Rev. Lett.}, 82:2661, 1999.

\bibitem{FriedmanSphere}
Hans~Wilhelm Alt, Luis~A. Caffarelli, and Avner Friedman.
\newblock Variational problems with two phases and their free boundaries.
\newblock {\em Transactions of the American Mathematical Society},
  282(2):431--461, 1984.

\bibitem{AmbrosioFusco}
Luigi Ambrosio, Nicola Fusco, and Diego Pallara.
\newblock {\em Functions of bounded variation and free discontinuity problems}.
\newblock Oxford Mathematical Monographs. The Clarendon Press, Oxford
  University Press, New York, 2000.

\bibitem{BartschDancerWang}
Thomas Bartsch, Norman Dancer, and Zhi-Qiang Wang.
\newblock A {L}iouville theorem, a-priori bounds, and bifurcating branches of
  positive solutions for a nonlinear elliptic system.
\newblock {\em Calc. Var. Partial Differential Equations}, 37(3-4):345--361,
  2010.

\bibitem{CaffarelliLin}
Luis Caffarelli and Fang-Hua Lin.
\newblock Singularly perturbed elliptic systems and multi-valued harmonic
  functions with free boundaries.
\newblock {\em J. Amer. Math. Soc.}, 21(3):847--862, 2008.

\bibitem{CLLL}
Shu-Ming Chang, Chang-Shou Lin, Tai-Chia Lin, and Wen-Wei Lin.
\newblock Segregated nodal domains of two-dimensional multispecies
  {B}ose-{E}instein condensates.
\newblock {\em Phys. D}, 196(3-4):341--361, 2004.

\bibitem{ChenLinZou}
Zhijie Chen, Chang-Shou Lin, and Wenming Zou.
\newblock Sign-changing solutions and phase separation for an elliptic system
  with critical exponent.
\newblock {\em Comm. Partial Differential Equations}, 39(10):1827--1859, 2014.

\bibitem{ChenZouARMA2012}
Zhijie Chen and Wenming Zou.
\newblock Positive least energy solutions and phase separation for coupled
  {S}chr{\"o}dinger equations with critical exponent.
\newblock {\em Arch. Ration. Mech. Anal.}, 205(2):515--551, 2012.

\bibitem{ChenZou2}
Zhijie Chen and Wenming Zou.
\newblock Positive least energy solutions and phase separation for coupled
  {S}chr\"{o}dinger equations with critical exponent: higher dimensional case.
\newblock {\em Calc. Var. Partial Differential Equations}, 52(1-2):423--467,
  2015.

\bibitem{ClappPistoia2018}
M\'{o}nica Clapp and Angela Pistoia.
\newblock Existence and phase separation of entire solutions to a pure critical
  competitive elliptic system.
\newblock {\em Calc. Var. Partial Differential Equations}, 57(1):Paper No. 23,
  20, 2018.

\bibitem{ClappPistoia2022}
M\'{o}nica Clapp and Angela Pistoia.
\newblock Fully nontrivial solutions to elliptic systems with mixed couplings.
\newblock {\em Nonlinear Anal.}, 216:Paper No. 112694, 19, 2022.

\bibitem{HugoHolderVariable}
M{\'o}nica Clapp, Angela Pistoia, and Hugo Tavares.
\newblock Yamabe systems, optimal partitions, and nodal solutions to the yamabe
  equation, arXiv:2106.00579, 2021.

\bibitem{ClappSzulkin2019}
M\'{o}nica Clapp and Andrzej Szulkin.
\newblock A simple variational approach to weakly coupled competitive elliptic
  systems.
\newblock {\em NoDEA Nonlinear Differential Equations Appl.}, 26(4):Paper No.
  26, 21, 2019.

\bibitem{TerraciniHolderBoundsOG}
Monica Conti, Susanna Terracini, and G.~Verzini.
\newblock Asymptotic estimates for the spatial segregation of competitive
  systems.
\newblock {\em Advances in Mathematics}, 195(2):524--560, 2005.

\bibitem{CorreiaJDE2016}
Sim\~{a}o Correia.
\newblock Characterization of ground-states for a system of {$M$} coupled
  semilinear {S}chr\"{o}dinger equations and applications.
\newblock {\em J. Differential Equations}, 260(4):3302--3326, 2016.

\bibitem{CorreiaNA2016}
Sim\~{a}o Correia.
\newblock Ground-states for systems of {$M$} coupled semilinear
  {S}chr\"{o}dinger equations with attraction-repulsion effects:
  characterization and perturbation results.
\newblock {\em Nonlinear Anal.}, 140:112--129, 2016.

\bibitem{DancerWeiWeth}
E.~N. Dancer, Juncheng Wei, and Tobias Weth.
\newblock A priori bounds versus multiple existence of positive solutions for a
  nonlinear {S}chr{\"o}dinger system.
\newblock {\em Ann. Inst. H. Poincar{\'e} Anal. Non Lin{\'e}aire},
  27(3):953--969, 2010.

\bibitem{DiasTavares}
Manuel Dias.
\newblock Optimal uniform bounds for competing variational elliptic systems
  with variable coefficients.
\newblock Master thesis, Instituto Superior T\'ecnico,
  https://fenix.tecnico.ulisboa.pt/cursos/mma/dissertacao/1972678479055450,
  November 2022.

\bibitem{MULT1DANCERWANG}
E.N.Dancer, Kelei Wang, and Zhitao Zhang.
\newblock The limit equation for the {G}ross-{P}itaevskii equations and {S}.
  {T}erracini's conjecture.
\newblock {\em Journal of Functional Analysis}, 262(3):1087--1131, 2012.

\bibitem{FriedlanHayman}
S.~Friedland and W.~K. Hayman.
\newblock Eigenvalue inequalities for the {D}irichlet problem on spheres and
  the growth of subharmonic functions.
\newblock {\em Comment. Math. Helv.}, 51(2):133--161, 1976.

\bibitem{Mariana2}
Nicola Garofalo, Arshak Petrosyan, and Mariana Smit Vega~Garcia.
\newblock An epiperimetric inequality approach to the regularity of the free
  boundary in the {S}ignorini problem with variable coefficients.
\newblock {\em J. Math. Pures Appl. (9)}, 105(6):745--787, 2016.

\bibitem{GAROFALO2014682}
Nicola Garofalo and Mariana {Smit Vega Garcia}.
\newblock New monotonicity formulas and the optimal regularity in the
  {S}ignorini problem with variable coefficients.
\newblock {\em Advances in Mathematics}, 262:682--750, 2014.

\bibitem{Kukavica}
Igor Kukavica.
\newblock Quantitative uniqueness for second-order elliptic operators.
\newblock {\em Duke Math. J.}, 91(2):225--240, 1998.

\bibitem{LinWei}
Tai-Chia Lin and Juncheng Wei.
\newblock Ground state of {$N$} coupled nonlinear {S}chr{\"o}dinger equations
  in {$\bold R^n$}, {$n\leq 3$}.
\newblock {\em Comm. Math. Phys.}, 255(3):629--653, 2005.

\bibitem{Mandel}
Rainer Mandel.
\newblock Minimal energy solutions for cooperative nonlinear {S}chr{\"o}dinger
  systems.
\newblock {\em NoDEA Nonlinear Differential Equations Appl.}, 22(2):239--262,
  2015.

\bibitem{uniformHolderBoundsHugoTerraciniNoris}
Benedetta Noris, Susanna Terracini, Hugo Tavares, and Gianmaria Verzini.
\newblock Uniform {H}{\"o}lder bounds for nonlinear {S}chr{\"o}dinger systems
  with strong competition.
\newblock {\em Communications on Pure and Applied Mathematics}, 63(3):267--302,
  2010.

\bibitem{OliveiraTavares}
Filipe Oliveira and Hugo Tavares.
\newblock Ground states for a nonlinear {S}chr\"{o}dinger system with sublinear
  coupling terms.
\newblock {\em Adv. Nonlinear Stud.}, 16(2):381--387, 2016.

\bibitem{PengPengWang2016}
Shuangjie Peng, Yan-Fang Peng, and Zhi-Qiang Wang.
\newblock On elliptic systems with {S}obolev critical growth.
\newblock {\em Calc. Var. Partial Differential Equations}, 55(6):Art. 142, 30,
  2016.

\bibitem{PengWangWang2019}
Shuangjie Peng, Qingfang Wang, and Zhi-Qiang Wang.
\newblock On coupled nonlinear {S}chr\"{o}dinger systems with mixed couplings.
\newblock {\em Trans. Amer. Math. Soc.}, 371(11):7559--7583, 2019.

\bibitem{Rogel_Salazar_2013}
Jesus Rogel-Salazar.
\newblock The {G}ross{\textendash}{P}itaevskii equation and
  {B}ose{\textendash}{E}instein condensates.
\newblock {\em European Journal of Physics}, 34(2):247--257, jan 2013.

\bibitem{Soave}
Nicola Soave.
\newblock On existence and phase separation of solitary waves for nonlinear
  {S}chr{\"o}dinger systems modelling simultaneous cooperation and competition.
\newblock {\em Calc. Var. Partial Differential Equations}, 53(3-4):689--718,
  2015.

\bibitem{SoaveTavares}
Nicola Soave and Hugo Tavares.
\newblock New existence and symmetry results for least energy positive
  solutions of {S}chr{\"o}dinger systems with mixed competition and cooperation
  terms.
\newblock {\em J. Differential Equations}, 261(1):505 -- 537, 2016.

\bibitem{SOAVE2016388}
Nicola Soave, Hugo Tavares, Susanna Terracini, and Alessandro Zilio.
\newblock H{\"o}lder bounds and regularity of emerging free boundaries for
  strongly competing {S}chr{\"o}dinger equations with nontrivial grouping.
\newblock {\em Nonlinear Analysis}, 138:388--427, 2016.
\newblock Nonlinear Partial Differential Equations, in honor of Juan Luis
  V{\'a}zquez for his 70th birthday.

\bibitem{SoaveWeth}
Nicola Soave and Tobias Weth.
\newblock The unique continuation property of sublinear equations.
\newblock {\em SIAM J. Math. Anal.}, 50(4):3919--3938, 2018.

\bibitem{SoaveZilio}
Nicola Soave and Alessandro Zilio.
\newblock Uniform bounds for strongly competing systems: The optimal
  {L}ipschitz case.
\newblock {\em Archive for Rational Mechanics and Analysis}, 218(2):647--697,
  apr 2015.

\bibitem{Struwe}
Michael Struwe.
\newblock {\em Variational methods}, volume~34 of {\em Ergebnisse der
  Mathematik und ihrer Grenzgebiete. 3. Folge. A Series of Modern Surveys in
  Mathematics [Results in Mathematics and Related Areas. 3rd Series. A Series
  of Modern Surveys in Mathematics]}.
\newblock Springer-Verlag, Berlin, fourth edition, 2008.
\newblock Applications to nonlinear partial differential equations and
  Hamiltonian systems.

\bibitem{HugoTerraciniWeakReflectionLaw}
Hugo Tavares and Susanna Terracini.
\newblock Regularity of the nodal set of segregated critical configurations
  under a weak reflection law.
\newblock {\em Calculus of Variations and Partial Differential Equations},
  45:273--317, 2010.

\bibitem{TavaresYou}
Hugo Tavares and Song You.
\newblock Existence of least energy positive solutions to {S}chr\"{o}dinger
  systems with mixed competition and cooperation terms: the critical case.
\newblock {\em Calc. Var. Partial Differential Equations}, 59(1):Paper No. 26,
  35, 2020.

\bibitem{TavaresYouZou}
Hugo Tavares, Song You, and Wenming Zou.
\newblock Least energy positive solutions of critical {S}chr\"{o}dinger systems
  with mixed competition and cooperation terms: the higher dimensional case.
\newblock {\em J. Funct. Anal.}, 283(2):Paper No. 109497, 50, 2022.

\bibitem{TianWang}
Rushun Tian and Zhi-Qiang Wang.
\newblock Multiple solitary wave solutions of nonlinear {S}chr\"{o}dinger
  systems.
\newblock {\em Topol. Methods Nonlinear Anal.}, 37(2):203--223, 2011.

\bibitem{Timmermans}
Eddy Timmermans.
\newblock Phase separation of {B}ose-{E}instein condensates.
\newblock {\em Physical Review Letters}, 81(26):5718--5721, 1998.

\bibitem{SphereDensityAlt}
Kelei Wang.
\newblock On the {D}e {G}iorgi type conjecture for an elliptic system modeling
  phase separation.
\newblock {\em Communications in Partial Differential Equations},
  39(4):696--739, 2014.

\bibitem{WeiWu}
Juncheng Wei and Yuanze Wu.
\newblock Ground states of nonlinear {S}chr\"{o}dinger systems with mixed
  couplings.
\newblock {\em J. Math. Pures Appl. (9)}, 141:50--88, 2020.

\bibitem{YinZou}
Xin Yin and Wenming Zou.
\newblock Positive least energy solutions for {$k$}-coupled {S}chr\"{o}dinger
  system with critical exponent: the higher dimension and cooperative case.
\newblock {\em J. Fixed Point Theory Appl.}, 24(1):Paper No. 5, 39, 2022.

\end{thebibliography}

\bigbreak

\textbf{Manuel Dias and Hugo Tavares}\\
Departamento de Matem\'atica do Instituto Superior T\'ecnico\\
Universidade de Lisboa\\
Av. Rovisco Pais\\
1049-001 Lisboa, Portugal\\
\texttt{manuel.g.dias@tecnico.ulisboa.pt, hugo.n.tavares@tecnico.ulisboa.pt}

\end{document}